\documentclass[12pt]{article}
\usepackage[centertags,intlimits]{amsmath}                     
\usepackage{color}\usepackage{amsfonts}                     
\usepackage{amsthm}
\usepackage{graphics}
\allowdisplaybreaks[2]
\numberwithin{equation}{section}
\newtheorem{theorem}{Theorem}[section]
\newtheorem{lemma}{Lemma}[section]

\theoremstyle{remark}
\newtheorem{remark}{Remark}[section]

\providecommand{\abs}[1]{\lvert #1\rvert}
\providecommand{\norm}[1]{\lVert #1\rVert}

\newcommand{\nc}{\newcommand}
\nc{\vb}{\mathbf{v}}
\nc{\bx}{\mathbf{x}}
\nc{\by}{\mathbf{y}}
\nc{\bz}{\mathbf{z}}
\nc{\bu}{\mathbf{u}}
\nc{\bv}{\mathbf{v}}
\nc{\ba}{\mathbf{a}}
\nc{\bs}{\mathbf{s}}
\nc{\bq}{\mathbf{q}}
\nc{\bd}{\mathbf{d}}
\nc{\bb}{\mathbf{b}}
\nc{\bc}{\mathbf{c}}
\nc{\bi}{\mathbf{i}}
\nc{\bfr}{\mathbf{r}}
\nc{\bA}{\mathbf{A}}
\nc{\R}{\mathbb R}
\nc{\N}{\mathbb N}
\nc{\C}{\mathbb C}
\nc{\D}{\mathbb D}
\nc{\Z}{\mathbb Z}
\nc{\F}{\mathbf F}
\nc{\bbS}{\mathbb S}
\nc{\B}{\cal B}
\nc{\br}{\bigr}
\nc{\bl}{\bigl}
\nc{\Bl}{\Bigl}
\nc{\Br}{\Bigr}
\nc{\ind}{\mathbf{1}}
\nc{\bP}{\mathbf{P}}

\title{On long term investment optimality}
\author{Anatolii A. Puhalskii \footnote{Email: aapuhalski@aim.com}\\
 Institute for Problems in Information
Transmission}
\begin{document}

\maketitle
\sloppy
\begin{abstract}
We study the problem of optimal long term
investment with a view to beat a benchmark for a diffusion model of
asset prices.
Two kinds of objectives are considered. One criterion
 concerns the probability of
outperforming the benchmark and
 seeks  either to  minimise the decay rate of the
probability that a portfolio exceeds the benchmark or to maximise
the decay rate that the portfolio falls short. The other criterion concerns
the growth rate of the risk--sensitive  utility of wealth which has to
be either minimised, for a risk--averse investor, or maximised, for
a risk--seeking investor.
It is assumed that the
mean returns and volatilities of the securities are affected by an
economic factor,  possibly, in a nonlinear fashion. The economic factor
and the benchmark are modelled with  general It\^o differential
equations. The results identify optimal  portfolios and produce
the decay, or growth, rates. The portfolios have the form of
time--homogeneous  functions of the economic factor.
 Furthermore,  a uniform treatment is given to
 the out--
and under-- performance probability optimisation as well as
to  the risk--averse and risk--seeking
 portfolio optimisation.
It is shown that there exists
 a portfolio that optimises the decay
rates  of both the outperformance probability and the
underperformance probability. 
While earlier research on the subject has relied, for the most part, on the techniques of
  stochastic optimal control and dynamic programming,    
in this contribution the quantities of interest are studied
directly by employing  the methods of the large
deviation theory. 
The
key to the analysis
is to recognise  the setup in question as a case of coupled diffusions
with time scale separation, with the economic factor representing
''the fast motion''.
\end{abstract}

\section{Introduction}
\label{sec:introduction}

Recently, two approaches have emerged
 to constructing long--term optimal portfolios
for diffusion models of asset prices:
 optimising the risk--sensitive criterion and optimising
the probability of outperforming a benchmark. In the risk--sensitive
framework, one is concerned with the
 expected utility of wealth  $\mathbf Ee^{\lambda
 \ln Z_t}$\,, where $Z_t$ represents the portfolio's wealth at time $t$
and $\lambda$ is the risk--sensitivity parameter, also referred to as
a Hara parameter, which expresses the
investor's degree of risk aversion if $\lambda<0$ or of risk--seeking if
$\lambda>0$\,. When trying to beat the benchmark, $Y_t$, the expected
utility of wealth is given by 
$\mathbf Ee^{\lambda
 \ln( Z_t/Y_t)}$\,. 
Since typically those expectations grow, or decay, at an
exponential rate with $t$\,, one is   led to optimise that rate,
so an optimal portfolio for the risk--averse
investor (respectively, for the risk--seeking investor) is defined as
 the one that minimises (respectively, maximises) the limit, assuming
 it exists, of  
 $(1/t)\ln\mathbf Ee^{\lambda
 \ln( Z_t/Y_t)}$\,, as
 $t\to\infty$\,.
In a similar vein, there are two ways to define the criterion 
 when the objective is to outperform the benchmark. One
can either choose the limit of $(1/t)\ln \mathbf
P(\ln(Z_t/ Y_t)\le0)$\,, as $t\to\infty$\,, as the quantity to be minimised
or the limit of $(1/t)\ln \mathbf P(\ln(Z_t/ Y_t)\ge0)$
as the quantity to be maximised. Arguably, the former criterion is
favoured by the risk--averse investor and the latter, by the 
risk--seeking one.  
More generally, one may look at the limits of $(1/t)\ln \mathbf
P(\ln(Z_t/ Y_t)\le q)$ or of $(1/t)\ln \mathbf P(\ln(Z_t/ Y_t)\ge q)$\,, for
some threshold  $q$\,.

Risk--sensitive
optimisation  has received considerable attention in the literature and
 has been studied under various sets of hypotheses. Bielecki and Pliska
\cite{BiePli99}
 consider a setting with constant volatilitities and
with  mean returns of the 
securities being affine functions of an economic factor, which is
modelled as a
Gaussian process that satisfies
a  linear stochastic differential equation with constant diffusion
coefficients. For the risk--averse investor, they find an
asymptotically optimal
portfolio and the long term growth rate of the expected utility of
wealth.
Subsequent research has relaxed some of the assumptions made, such as the
independence of the diffusions driving the economic factor process and the
asset price process, see Kuroda and Nagai \cite{KurNag02}, Bielecki and
Pliska \cite{BiePli04}. Fleming
 and Sheu  \cite{FleShe00}, \cite{FleShe02}  analyse both the
 risk--averse and the risk--seeking setups.
A benchmarked setting is studied by Davis and Lleo \cite{DavLle08},
 \cite{DavLle11}, \cite{DavLle13}, the latter two papers
 being concerned with diffusions with jumps as driving processes.
Nagai \cite{Nag03}  
assumes general mean returns and volatilities
 and the factor process
being the solution to a general stochastic differential equation and
obtains an optimal portfolio for the risk--averse investor when there
is no benchmark involved. 
Special one--dimensional models are treated in Fleming
 and Sheu 
\cite{FleShe99} and Bielecki, Pliska, and Sheu \cite{BiePliShe05}.
The methods of the aforementioned papers rely on the tools of stochastic
optimal control. A Hamilton--Jacobi--Bellman equation 
is invoked in order to identify a
portfolio that minimises
 the expected
utility of wealth on a finite horizon. Afterwards, a limit
is taken as the
length of time goes to infinity. The optimal portfolio is expressed in
terms of a solution to a Riccati algebraic equation in the affine
case, and to an ergodic Bellman equation, in the general case.

The criterion of the probability of outperformance is
considered in Pham \cite{Pha03}, who studies a one--dimensional
benchmarked setup. The
minimisation of the underperformance probability 
 for   the Bielecki and
Pliska \cite{BiePli99} model is addressed 
in Hata, Nagai, and Sheu \cite{Hat10}, who look at a  no benchmark
setup.  Nagai \cite{Nag12} studies the 
general model with the riskless asset as the benchmark.  
Those authors build on the foundation laid by the work
on the risk--sensitive optimisation by applying stochastic control methods in
order to  identify an optimal
risk--sensitive portfolio, 
first, and, afterwards, use duality considerations 
to optimise the probabilities of out/under performance.
The  risk--sensitive optimal
portfolio for an appropriately chosen risk--sensitivity parameter is found
to be optimal for the out/under performance probability criterion,
although a proof of that fact
 is missing  for the general model in Nagai \cite{Nag12}. The parameter
is between zero and one for the outperformance case and is negative, for the
underperformance case.
Puhalskii \cite{Puh11}  analyses the out/under performance
probabilities 
directly and
obtains a portfolio that is
asymptotically optimal both for the outperformance and
underperformance probabilities, the limitation of their study being
that it  is confined to a
geometric Brownian motion model of the asset prices with no economic
factor involved. Puhalskii and Stutzer \cite{PuhStu16} study the
underperformance probability  for
  the model in Nagai \cite{Nag12} with a general benchmark 
by aplying direct methods.
Their results imply that the portfolio found in Nagai \cite{Nag12}
is optimal.

Whereas the cases of a negative Hara parameter for risk--sensitive
optimisation  and of the underperformance probability
minimisation 
seem to be fairly well understood, the setups of a positive Hara
parameter for risk--sensitive optimisation 
and of the outperformance probability optimisation
 are lacking
clarity. The reason seems to be twofold. Firstly,
 the expected utility of wealth
may  grow at an infinite exponential rate for certain $\lambda\in[0,1]$\,, see
  Fleming and Sheu \cite{FleShe02}. Secondly, the analysis of the ergodic
  Bellman equation presents difficulty because   no
  Lyapunov function is readily available, cf., condition (A3) in Kaise and Sheu
  \cite{KaiShe06}. 
 Although Pham \cite{Pha03} carries out a
  detailed study  and identifies the threshold value
  of $\lambda$ when ''the blow--up'' occurs for an affine
 model of one security and
  one factor,  for the
  multidimensional case, 
we are unaware of  results that produce asymptotically optimal portfolios 
either for the risk--sensitive criterion with a positive Hara parameter or
for maximising  the outperformance probability.

The purpose of this paper is to fill in the aforementioned gaps.
As in   Puhalskii and Stutzer \cite{PuhStu16},
we study  the benchmarked version of the general
  model introduced in Nagai \cite{Nag03},
\cite{Nag12}. Capitalising on the insights in Puhalskii and Stutzer
\cite{PuhStu16},
  we identify an optimal
portfolio for maximising the outperformance probability.
For the risk--sensitive setup, we prove that
 there
is a threshold value  $\overline\lambda\in(0,1]$ such that for all 
$\lambda<\overline\lambda$
 there exists an asymptotically optimal 
risk--seeking
portfolio.
It is 
arrived at   as an  optimal outperformance portfolio 
 for  certain  threshold
$q$\,.
 If $\lambda>\overline\lambda$\,, there is a portfolio such that 
the expected utility of
wealth grows at an infinite exponential rate.   
 Furthermore, we 
give a uniform treatment 
 to the out--
and under-- performance probability optimisation
as well as to the risk--averse and risk--seeking
 portfolio optimisation.
Not only is that of methodological value, but the proofs for the case of
a positive Hara parameter rely on the optimality 
 properties of a  portfolio with a negative Hara parameter. We
show  that the
 same portfolio optimises both the underperformance and outperformance
 probabilities,  in line with conclusions in Puhalskii
 \cite{Puh11}. Similarly, the same procedure can be used for finding
 optimal risk--sensitive portfolios both for the risk--averse investor
 and for the
 risk--seeking investor.
As in Nagai \cite{Nag03,Nag12} and Puhalskii and Stutzer \cite{PuhStu16},  the 
portfolios are expressed in terms of solutions to
  ergodic Bellman equations.

Since we use the methods of Puhalskii and Stutzer \cite{PuhStu16}, 
  no stochastic
control techniques are invoked
 and  standard tools of large deviation theory are employed,
such as a change of a probability measure and an exponential Markov inequality.
The key is to recognise that one deals with  a case of coupled diffusions
with time scale separation and introduce the empirical measure of the
factor process which is ''the fast motion''. 
 Another notable feature   is an extensive use of the
minimax theorem and a characterisation of the optimal portfolios in terms
of  saddle points. 
Being more direct than the one based on  the
 stochastic optimal control theory, this approach streamlines 
 considerations,  e.g., there is no need to contend
with  a  Hamilton--Jacobi--Bellman equation on finite time,  thereby
enabling us 
 both to obtain new results and 
  relax or drop altogether 
a number of assumptions present in the earlier  research on the subject.
 For instance, 
we do not restrict the class of portfolios under consideration to
 portfolios whose total wealth is a sublinear
function of  the economic factor, nor do we require that
   the limit growth rate of the expected
utility of wealth 
  be
 an essentially smooth (or ''steep'')
 function of the Hara parameter, which conditions are needed in 
Pham \cite{Pha03}
even for a one--dimensional model.
On the other hand, when  optimizing the underperformance
probability and  when 
optimizing the risk--sensitive criterion  with a negative Hara
parameter, we produce $\epsilon$--asymptotically optimal portfolios, rather
than asymptotically optimal portfolios as in Hata, Nagai, and Sheu
 \cite{Hat10} and in Nagai \cite{Nag12}, which distinction does not
 seem to be of  great  significance.
Besides, our conditions seem to be less restrictive.

 The proofs of certain saddle--point 
properties  for positive Hara parameters
 relying on the associated properties for negative Hara parameters,
 this paper includes a substantial portion of the
developments in Puhalskii and Stutzer \cite{PuhStu16}. 
The presentation, however, is self--contained and does not
 depend on any of
the results of Puhalskii and Stutzer \cite{PuhStu16}.
This is how this paper is organised. In Section \ref{sec:model}, we
define the model and state the main results.
In addition, more detail is given on the relation to earlier work. 
 The proofs are provided in Section \ref{sec:proof-bounds}
whereas Section \ref{sec:prelim} and the
 appendix are concerned with laying the groundwork and shedding
 additional light on  the model of Pham \cite{Pha03}.

\section{A model description and  main results}
\label{sec:model}

We start by recapitulating the setup of Puhalskii and Stutzer
\cite{PuhStu16}. 
One is concerned with a portfolio consisting of $n$ risky securities priced
$S^1_t,\ldots,S^n_t$ at time $t$ and a safe security of price $S^0_t$ at
time $t$\,.
We assume that, 
 for $i=1,2,\ldots,n$,
\begin{equation*}
  \dfrac{dS^i_t}{S^i_t}= a^i(X_t)\,dt+{b^i(X_t)}^T\,dW_t
\end{equation*}
and that 
\begin{equation*}
\frac{dS^0_t}{S^0_t}=r(X_t)\,dt\,,
\end{equation*}
where $X_t$ represents an economic factor.
It is governed by the  equation 
\begin{equation}
  \label{eq:14}
  dX_t=\theta(X_t)\,dt+\sigma(X_t)\,dW_t\,.
\end{equation}
In the equations above, the
$a^i(x)$ are real-valued functions, the $b^i(x)$ are 
$\R^k$-valued functions,
 $\theta(x)$ is  an $\R^l$-valued function,
$\sigma(x)$ is an 
$l\times k$-matrix,
$W_t$ is a $k$-dimensional standard Wiener process, and $S^i_0>0$\,,
${}^T$ is used to denote the transpose of a matrix or a vector.
Accordingly, the process $X=(X_t\,, t\ge0)$ is $l$-dimensional.

Benchmark $Y=(Y_t\,,t\ge0)$ follows an equation similar to
those for the risky securities:
\begin{equation*}
  \dfrac{dY_t}{Y_t}=\alpha(X_t)\,dt+\beta(X_t)^T\,dW_t,
\end{equation*}
where $\alpha(x)$ is an $\R$-valued function, $\beta(x)$ is an
$\R^k$-valued function, and $Y_0>0$\,.

 All processes
 are defined on a complete probability space
 $(\Omega,\mathcal{F},\mathbf{P})$\,. 
It is assumed, furthermore, 
that the processes $S^i=(S^i_t\,,t\ge0)$\,, $X$\,, and
$Y=(Y_t\,,t\ge0)$ are adapted 
 to (right--continuous)
 filtration $\mathbf{F}=(\mathcal{F}_t\,,t\ge0)$ and that
 $W=(W_t\,,t\ge0)$ is an   $\mathbf{F}$-Wiener process.

We let $a(x)$ denote the $n$-vector with entries  $a^1(x),\ldots,a^n(x)$,
let $b(x)$ denote the $n\times k$ matrix with rows
${b^1(x)}^T,\ldots,{b^n(x)}^T$ and let
$\mathbf{1}$ denote the $n$-vector with unit entries.  
 The matrix functions
 $b(x)b(x)^T$ and $\sigma(x)\sigma(x)^T$ are assumed to be
uniformly positive definite and bounded.
 The functions 
${a(x)}$\,,    ${r(x)}$\,,  ${\theta(x)}$\,,
  $\alpha(x)$\,,  $b(x)$\,, $\sigma(x)$\,,     and $\beta(x)$
are assumed to be continuously differentiable with bounded derivatives
  and the function $\sigma(x)\sigma(x)^T$ is assumed to be twice continuously
 differentiable.
In addition,  the following ''linear growth'' condition is assumed:
for some $K>0$ and all $x\in \R^l$\,,
\begin{equation*}
\abs{a(x)}+\abs{r(x)}+\abs{\alpha(x)}+
    \abs{\theta(x)}\le K(1+\abs{x})\,.
\end{equation*}
The function  $\abs{\beta(x)}^2$  is assumed to be  bounded and
 bounded
away from zero. 
(We will also indicate how the results change if the benchmark ''is not
volatile'' meaning  that $\beta(x)=0$\,.)
Under those hypotheses, the processes $S^i$\,, $X$\,, and $Y$ are well
defined, see, e.g., chapter 5 of Karatzas and Shreve \cite{KarShr88}.

For the factor process,
we assume that 
\begin{equation}
  \label{eq:45}
  \limsup_{\abs{x}\to\infty}\,
\theta(x)^T\,\frac{ x}{\abs{x}^2}<0\,.
\end{equation}
Thus, $X$ has a unique invariant measure, see, e.g., Bogachev,
Krylov, and R\"ockner \cite{BogKryRoc}. 
As for the initial condition, 
we will assume that
\begin{equation}
  \label{eq:77}
\mathbf Ee^{\gamma\abs{X_0}^2}<\infty\,, \text{ for some }\gamma>0\,.
\end{equation}
Sometimes it will be required that $\abs{X_0}$ be,  moreover,
bounded.

The investor holds $l^i_t$ shares of risky security $i$ and $l^0_t$ shares
 of the safe security at time $t$\,,
 so the total wealth
  is given by
$Z_t=\sum_{i=1}^nl^i_tS^i_t+l^0_tS^0_t$\,.
Portfolio
$  \pi_t=(\pi^1_t,\ldots,\pi^n_t)^T$
specifies the proportions of the total wealth invested in the risky
securities so that, for $i=1,2,\ldots,n$,
$l^i_tS^i_t=\pi^i_t Z_t$\,.
The processes $\pi^i=(\pi^i_t\,,t\ge0)$ are assumed to be
$(\mathcal{B}\otimes\mathcal{F}_t,\,t\ge0)$--progressively measurable, where $\mathcal{B}$
denotes the Borel $\sigma$--algebra on $\R_+$,
 and such that
$\int_0^t{\pi^i_s}^2\,ds<\infty$ a.s. We do not impose 
 any other restrictions on the magnitudes of the $\pi^i_t$ so that
unlimited borrowing and shortselling are allowed.

Let
\begin{equation*}
    L^\pi_t=\frac{1}{t}\,\ln\bl(\frac{Z_t}{Y_t}\br)\,.
\end{equation*}

Since
the amount of wealth invested in the safe
security is $(1-\sum_{i=1}^n \pi^i_t)Z_t$\,, in a standard fashion
by using the self--financing condition, one obtains that
\begin{equation*}
  \dfrac{dZ_t}{Z_t}=\sum_{i=1}^n\pi^i_t\,\dfrac{dS^i_t}{S^i_t}+
\bl(1-\sum_{i=1}^n\pi^i_t\br)\,\dfrac{dS^0_t}{S^0_t}\,.
\end{equation*}
Assuming that $Z_0=Y_0$ and letting $c(x)=b(x)b(x)^T$\,,
 we have
by It\^o's lemma that, cf. Pham \cite{Pha03},
\begin{multline}
\label{eq:1a}
  L^\pi_t=
\frac{1}{t}\,\int_0^t
\bl(\pi_s^T a(X_s)+(1-\pi_s^T \ind)r(X_s)
-\frac{1}{2}\,\pi_s^T c(X_s)\pi_s-\alpha(X_s)
+\frac{1}{2}\,\abs{\beta(X_s)}^2\br)\,ds\\+
\frac{1}{t}\,\int_0^t\bl(b(X_s)^T\pi_s -\beta(X_s)\br)^T
\,dW_s\,.
\end{multline}
One can see that $L^\pi_t$ is ''of order one'' for $t$
great. Therefore, if one embeds 
the probability of outperformance
$\mathbf P(\ln(Z_t/Y_t)\ge0)$ (respectively, the probability of
underperfomance
$\mathbf P(\ln(Z_t/Y_t)\le0)$) into  the parameterised family of
probabilities $\mathbf P(L^\pi_t\ge q)$ (respectively, 
$\mathbf P(L^\pi_t\le q)$)\,, one will concern themselves with large
deviation probabilities.

Let, for $u\in\R^n$ and $x\in\R^l$\,, 
\begin{subequations}
  \begin{align}
      \label{eq:4}
  M(u,x)&=u^T (a(x)- r(x)\ind )
-\frac{1}{2}\,u^T c(x)u+r(x)-\alpha(x)
+\frac{1}{2}\,\abs{\beta(x)}^2\intertext{ and }
  \label{eq:8}
N(u,x)&= b(x)^Tu-\beta(x)\,.
\end{align}
\end{subequations}
A change of variables brings    (\ref{eq:1a}) to the form
\begin{equation}
    \label{eq:5}
  L^\pi_t= 
\int_0^1 M(\pi_{ts},X_{ts})\,ds
+\frac{1}{\sqrt{t}}\,\int_0^1 N(\pi_{ts},X_{ts})^T\,dW_{s}^t\,,
\end{equation}
where 
$W^t_s=W_{ts}/\sqrt{t}$\,.
We note that $W^t=(W^t_s,\,s\in[0,1])$ is a Wiener process relative to
$\mathbf{F}^t=(\mathcal{F}_{ts},\,s\in[0,1])$\,. The righthand side of
\eqref{eq:5}  can
be viewed as a
diffusion process with a small diffusion coefficient which
lives in ''normal time'' represented by the variable $s$\,, 
whereas in $X_{ts}$ and $\pi_{ts}$ ''time'' is
 accelerated by a factor of $t$\,.
Furthermore,
on introducing $\pi^t_s=\pi_{ts}$\,, $X^t_s=X_{ts}$\,, assuming
 that, for suitable function $u(\cdot)$\,,
 $\pi^t_s=u(X^t_s)$\,,
defining \begin{equation}
  \label{eq:37}
  \Psi^t_s=\int_0^s M( u(X_{\tilde s}^t),X^t_{\tilde s})\,d\tilde s
+\frac{1}{\sqrt{t}}\,\int_0^s N(
 u(X_{\tilde s}^t),X_{\tilde s}^t)^T\,dW_{\tilde s}^t\,,
\end{equation}
so that $L^\pi_t=\Psi^t_1$\,,
and  writing \eqref{eq:14} as
\begin{equation}
  \label{eq:5'}
  X^t_s=X^t_0+t\int_0^s\theta(X^t_{\tilde s})\,d\tilde s+\sqrt{t}
\int_0^s\sigma(X^t_{\tilde s})\,dW^t_{\tilde s}\,,  
\end{equation}
one can see that
\eqref{eq:37} and \eqref{eq:5'} make up a similar system of equations
to those studied in Liptser \cite{Lip96} and in Puhalskii \cite{Puh16}.
The following heuristic derivation which is based on the Large Deviation
Principle in Theorem 2.1 in Puhalskii \cite{Puh16} provides insight
into our results below.

Let us introduce additional pieces of notation first.
 Let
  $\mathbb{C}^2$
represent the set of real--valued 
twice continuously differentiable
 functions on $\R^l$\,.
For $f\in\mathbb C^2$\,, we let $\nabla f(x)$ represent the gradient
of $f$ at $x$ which is regarded as a column $l$--vector and we let
$\nabla^2f(x)$ represent the $l\times l$--Hessian matrix of $f$ at $x$\,.
 Let $\mathbb C^1_0$ and $\mathbb C^2_0$
represent the sets of functions of compact support on $\R^l$ that are
 once and twice continuously differentiable, respectively.
Let  $\mathbb P$ denote
 the set of probability densities $m=(m(x)\,,x\in\R^l)$ on
$\R^l$ such that $\int_{\R^l}\abs{x}^2\,m(x)\,dx<\infty$ and let
 $\hat{\mathbb{P}}$ denote the set of probability densities
$m$ from  $\mathbb{P}$ such that
 $m\in\mathbb{W}^{1,1}_{\text{loc}}(\R^l)$ 
 and  $\sqrt{m}\in\mathbb{W}^{1,2}(\R^l)$\,, where $\mathbb W$
 is used for denoting a Sobolev space, see, e.g., Adams and Fournier
 \cite{MR56:9247}. 
Let $\mathbb C([0,1],\R)$ represent the set of continuous real--valued
functions on $[0,1]$ being endowed with the uniform topology and let 
$\mathbb C_\uparrow([0,1],\mathbb M(\R^l))$ represent the set of
 functions $\mu_t$ on $[0,1]$ with values in the set
$\mathbb M(\R^l)$ of (nonnegative)
 measures on $\R^l$ such that $\mu_t(\R^l)=t$ and $\mu_t-\mu_s$ is a
nonnegative measure when $t\ge s$\,. The space 
$\mathbb M(\R^l)$ is assumed to be equipped with the weak topology
and the space $\mathbb C_\uparrow([0,1],\mathbb M(\R^l))$\,, with
  the uniform topology.
Let  the empirical process of
$X^t=(X_{s}^t\,, s\in[0,1])$\,, which is denoted by
$\mu^t=(\mu^t(ds,dx))$\,, be 
defined by the equation
\begin{equation*}
  \mu^t([0,s],\Gamma)=\int_0^s\chi_{\Gamma}(X_{t\tilde s})\,d\tilde s\,,
\end{equation*}
with $\Gamma$ denoting a Borel subset of $\R^l$
and with
 $\chi_{\Gamma}(x)$ 
representing the indicator function of  $\Gamma$\,.
We note that both $X^t$ and $\pi^t=(\pi^t_s,s\in[0,1])$ are
$\mathbf{F}^t$-adapted.

If one were to apply  to the processes $\Psi^t=(\Psi^t_s\,, s\in[0,1])$
and $\mu^t$  Theorem 2.1 in Puhalskii \cite{Puh16}, then
 the pair
$(\Psi^t,\mu^t)$ would satisfy the Large Deviation Principle in 
$\mathbb C([0,1])\times \mathbb C_\uparrow([0,1],\mathbb
M_1(\R^l))$\,, as $t\to\infty$\,,
 with the  deviation function (usually referred to as a rate function)
\begin{multline}
  \label{eq:41''}
  \mathbf{ J}(\Psi,\mu)=
\int_0^1
\sup_{\lambda\in\R}\Bl(
\lambda\bl(\dot{\Psi}_s-
\int_{\R^l}M(u(x),x)\,m_s(x)\,dx\br)
-\frac{1}{2}\,\lambda^2 \int_{\R^l}\abs{N(u(x),x)}^2\,m_s(x)\,dx\\
+\sup_{f\in \mathbb{C}_0^1}\int_{\R^l}\Bl( \nabla  f(x)^T \bl(
\frac{1}{2}\,\text{div}\,\bl(\sigma(x)\sigma(x)^Tm_s(x)\br)-\bl(
\theta(x)+\lambda\sigma(x)^T N(u(x),x)\br)m_s(x)
\br)
\\-\frac{1}{2}\,
\abs{\sigma(x)^T\nabla   f(x)}^2\,m_s(x)\Br)\,dx\Br)\,ds\,,
\end{multline}
provided the  function 
$\Psi=(\Psi_s,\,s\in[0,1])$ 
 is absolutely continuous w.r.t.  Lebesgue measure on $\R$  and
the function $\mu=(\mu_s(\Gamma))$\,,
when considered as a measure on $[0,1]\times\R^l$\,, is
absolutely continuous w.r.t.  Lebesgue measure on $\R\times \R^l$\,, i.e.,
 $\mu(ds,dx)=m_s(x)\,dx\,ds$\,, where
   $m_s(x)$\,,   as a function of $x$\,, belongs to 
$\hat{\mathbb{P}}$
  for almost all $s$\,.
If those conditions do not hold, then 
$  \mathbf{ J}(\Psi,\mu)=\infty$\,. 
(We assume that the divergence of a square matrix 
is evaluated rowwise.)

Integration by parts yields an alternative form:
\begin{multline}
  \label{eq:116}
  \mathbf{ J}(\Psi,\mu)=
\int_0^1
\sup_{\lambda\in\R}\Bl(
\lambda\bl(\dot{\Psi}_s-
\int_{\R^l}M(u(x),x)\,m_s(x)\,dx\br)
-\frac{1}{2}\,\lambda^2 \int_{\R^l}\abs{N(u(x),x)}^2\,m_s(x)\,dx\\
+\sup_{f\in \mathbb{C}_0^2}\int_{\R^l}\Bl(
-\,\frac{1}{2}\,\text{tr}\,\bl(\sigma(x)\sigma(x)^T\nabla^2f(x)\br)-
\nabla f(x)^T(\theta(x)+\lambda\sigma(x)^T N(u(x),x))
\\-\frac{1}{2}\,
\abs{\sigma(x)^T\nabla   f(x)}^2\,\Br)m_s(x)\,dx\Br)\,ds\,,
\end{multline}
with $\text{tr}\, \Sigma$ standing for the trace of  square matrix 
$\Sigma$\,.
Since $L^{ \pi}_t=\Psi^t_1$\,,
by projection,  $L^{\pi}_t$ obeys the large deviation
principle in $\R$
 for rate $t$ with the deviation function 
$\mathbf{ I}(L)=\inf\{ \mathbf{ J}(\Psi,\mu):\;
\Psi_1=L\,\}$\,. 
Therefore, 
\begin{equation}
  \label{eq:111}
  \limsup_{t\to\infty}\frac{1}{t}\,\ln \mathbf P(L^\pi_t\ge q)\le
-\inf_{L\geq q}\mathbf I(L)\,.
\end{equation}
The integrand against $ds$ in \eqref{eq:116} being a convex function of
$\dot\Psi_s$ and of $m_s(x)$\,, along with the requirements that
$\int_0^1\dot\Psi_s\,ds=L$ and $\int_{\R^l}m_s(x)\,dx=1$
imply, by Jensen's inequality, that one may assume that $\dot\Psi_s=L$
and that $m_s(x)$ does not depend on $s$ either, so 
that $m_s(x)=m(x)$\,.
Hence,
\begin{multline*}
  \inf_{L\ge q}\mathbf I(L)=\inf_{L\ge q}
\inf_{m\in\hat{\mathbb P}}\sup_{\lambda\in\R}\Bl(
\lambda\bl(L-
\int_{\R^l}M(u(x),x)\,m(x)\,dx\br)
-\frac{1}{2}\,\lambda^2 \int_{\R^l}\abs{N(u(x),x)}^2\,m(x)\,dx\\
+\sup_{f\in \mathbb{C}_0^2}\int_{\R^l}\Bl(
-\,\frac{1}{2}\,\text{tr}\,\bl(\sigma(x)\sigma(x)^T\nabla^2f(x)\br)
-\nabla f(x)^T(\theta(x)+\lambda\sigma(x)^T N(u(x),x))
\\-\frac{1}{2}\,
\abs{\sigma(x)^T\nabla   f(x)}^2\Br)\,m(x)\,dx\Br)\,.
\end{multline*}
On noting that the expression on the righthand side is convex in
$(L,m)$ and is concave in $(\lambda,f)$\,, one  hopes to be able to
apply a minimax theorem to change the order of taking
$\inf $ and $\sup$ so that
\begin{multline}
\label{eq:118} 
   \inf_{L\ge q}\mathbf I(L)=\sup_{\lambda\in\R}\sup_{f\in \mathbb{C}_0^2}\inf_{L\ge q}
\inf_{m\in\hat{\mathbb P}}\Bl(
\lambda\bl(L-
\int_{\R^l}M(u(x),x)\,m(x)\,dx\br)
-\frac{1}{2}\,\lambda^2 \int_{\R^l}\abs{N(u(x),x)}^2\,m(x)\,dx\\
+\int_{\R^l}\Bl(
-\,\frac{1}{2}\,\text{tr}\,\bl(\sigma(x)\sigma(x)^T\nabla^2f(x)\br)
-\nabla f(x)^T(\theta(x)+\lambda\sigma(x)^T N(u(x),x))
\\-\frac{1}{2}\,
\abs{\sigma(x)^T\nabla   f(x)}^2\Br)\,m(x)\,dx\Br)\,.
\end{multline}
If $\lambda<0$\,, then the infimum over $L\ge q$ equals $-\infty$\,.
If $\lambda\ge0$\,, it
is attained at $L=q$  and $\inf_{m\in\hat{\mathbb P}}$ ''is 
attained at a $\delta$--density''
so that \eqref{eq:118} results in
\begin{multline}
  \label{eq:119}
      \inf_{L\ge q}\mathbf I(L)=\sup_{\lambda\in\R_+}\sup_{f\in \mathbb{C}_0^2}
\Bl(
\lambda q-\sup_{x\in\R^l}\bl(
\lambda M(u(x),x)
+\frac{1}{2}\,\lambda^2 \abs{N(u(x),x)}^2\,\\
+\frac{1}{2}\,\text{tr}\,\bl(\sigma(x)\sigma(x)^T\nabla^2f(x)\br)
+\nabla f(x)^T(\theta(x)+\lambda\sigma(x)^T N(u(x),x))
+\frac{1}{2}\,
\abs{\sigma(x)^T\nabla   f(x)}^2\br)\Br)\,.
\end{multline}
For an optimal outperforming portfolio, one wants to maximise the 
righthand side of
\eqref{eq:111} over functions $u(x)$\,, so the righthand side of
\eqref{eq:119} has to be minimised. Assuming one can apply
minimax considerations once again yields
\begin{multline*}
  \inf_{u(\cdot)}\inf_{L\ge q}\mathbf I(L)=
\sup_{\lambda\in\R_+}\sup_{f\in \mathbb{C}_0^2}
\Bl(
\lambda q-\sup_{x\in\R^l}\sup_{u\in\R^n}\bl(
\lambda M(u,x)
+\frac{1}{2}\,\lambda^2 \abs{N(u,x)}^2\,\\
+\nabla f(x)^T(\theta(x)+\lambda\sigma(x)^T N(u,x))\br)
+\frac{1}{2}\,\text{tr}\,\bl(\sigma(x)\sigma(x)^T\nabla^2f(x)\br)
+\frac{1}{2}\,
\abs{\sigma(x)^T\nabla   f(x)}^2\Br)\,.
\end{multline*}
 By \eqref{eq:4} and \eqref{eq:8}, the $\sup_{u\in\R^n}=\infty$ if
$\lambda>1$ so, on recalling \eqref{eq:111}, it is reasonable to
 conjecture that
\begin{multline}
  \label{eq:125}
\sup_{\pi}  \limsup_{t\to\infty}\frac{1}{t}\,\ln \mathbf P(L^\pi_t\ge q)=-
\sup_{\lambda\in[0,1]}\sup_{f\in \mathbb{C}_0^2}
\Bl(
\lambda q-\sup_{x\in\R^l}\sup_{u\in\R^n}\bl(
\lambda M(u,x)
+\frac{1}{2}\,\lambda^2 \abs{N(u,x)}^2\,\\
+\nabla f(x)^T(\theta(x)+\lambda\sigma(x)^T N(u,x))\br)
+\frac{1}{2}\,\text{tr}\,\bl(\sigma(x)\sigma(x)^T\nabla^2f(x)\br)
+\frac{1}{2}\,
\abs{\sigma(x)^T\nabla   f(x)}^2\Br)
\end{multline}
and an optimal portfolio is of the form $u(X_t)$\,, with $u(x)$
attaining the supremum with respect to $u$ on the righthand side of
 \eqref{eq:125} for $\lambda$ and $f$ that
deliver their respective suprema. Similar arguments may be applied  to
finding $\inf_{\pi}  \liminf_{t\to\infty}(1/t)\,\ln \mathbf P(L^\pi_t<
q)$\,. Unfortunately, we are unable to fill in the gaps in the above
deduction, e.g., in order for the results of Puhalskii \cite{Puh16} to
apply, the function $u(x)$ has to be bounded in $x$, while the optimal
portfolio  typically is not.
Besides, it is not at all obvious that the optimal portfolio should be
expressed as a function of the economic factor.
Nevertheless, the above line of reasoning is essentially correct,
as our main results  show. Besides,  there is 
a special case which we analyse at the final stages of our proofs that
allows a  direct application of Theorem 2.1 in Puhalskii \cite{Puh16}. 
We now proceed to stating the results.
That requires introducing more pieces of notation and 
providing background information.

The following nondegeneracy
condition is needed. (It was introduced in Puhalskii and Stutzer
\cite{PuhStu16}.)  Let $I_k$ denote the $k\times
  k$--identity matrix and let 
  \begin{equation*}
Q_1(x)=I_k-b(x)^Tc(x)^{-1}b(x)\,.
  \end{equation*}
The matrix $Q_1(x)$ represents the orthogonal 
projection operator onto the null space      of
$b(x)$ in $\R^k$\,.
 We will assume that
\begin{itemize}
\item[(N)]
  \begin{enumerate}
  \item 
 The matrix $\sigma(x)Q_1(x)\sigma(x)^T$ is uniformly
  positive definite.
\item The quantity
  $
  \beta(x)^TQ_2(x)\beta(x)
  $ is bounded away from zero,
where\\
\begin{equation}
  \label{eq:73}
  Q_2(x)=Q_1(x)
\bl(I_k-\sigma(x)^T(\sigma(x)Q_1(x)\sigma(x)^T)^{-1}\sigma(x)\br)
Q_1(x)\,.
\end{equation}

  \end{enumerate}
\end{itemize}
Condition (N) admits the following geometric interpretation.
\begin{lemma}
  \label{le:angle}
The matrix $\sigma(x)Q_1(x)\sigma(x)^T$ 
is uniformly positive definite 
 if and only if
 arbitrary nonzero
vectors from the ranges of $\sigma(x)^T$ and  $b(x)^T$\,, respectively,
 are
 at angles bounded away from zero
if and only if 
the matrix
$c(x)-b(x)\sigma(x)^T(\sigma(x)\sigma(x)^T)^{-1}\sigma(x)  b(x)^T$
is uniformly positive definite. 
Also,      $\beta(x)^TQ_2(x)\beta(x)$  is bounded away from zero
if and only if the projection of $\beta(x)$ onto the null space
 of $b(x)$ is of  length bounded away from zero 
and is at angles bounded away from zero to all 
projections onto that null space of nonzero vectors from
 the range of $\sigma(x)^T$\,.
\end{lemma}
The proof of the lemma is provided in the appendix.
Under part 1 of condition (N), we have that $k\ge n+l$
and the rows of the matrices $\sigma(x)$ and $b(x)$ are linearly
independent. 
Part 2 of condition (N) implies that $\beta(x)$ does not belong to the
sum of the ranges of $b(x)^T$ and of $\sigma(x)^T$\,. (Indeed, if that
were the case, then $Q_1(x)\beta(x)$\,, which is the projection of
$\beta(x)$ onto the null space of $b(x)$\,,  would also be the projection of a
vector from the range of $\sigma(x)^T$ onto the null space of $b(x)$\,.)
Thus, $k>n+l$\,.

The righthand side of \eqref{eq:125} motivates the
following definitions.
Let, given $x\in\R^l$\,, $\lambda\in\R$\,, and $p\in\R^l$\,, 
\begin{equation}
  \label{eq:40}
\breve  H(x;\lambda,p)= \lambda\sup_{u\in\R^n}\bl(  M(u,x)
+\frac{1}{2}\,\lambda
\abs{N(u,x)}^2+ p^T\sigma(x) N(u,x)\br)+
p^T\theta(x)+\frac{1}{2}\,\abs{{\sigma(x)}^Tp}^2\,.
  \end{equation}
By \eqref{eq:4} and \eqref{eq:8}, 
the latter righthand side is  finite if
$\lambda<1$\,, 
with the supremum being attained at
\begin{equation}
  \label{eq:39}
           u(x)=\frac{1}{1-\lambda}\,c(x)^{-1}\bl(a(x)-r(x)\mathbf1
-\lambda b(x)\beta(x)+b(x)\sigma(x)^Tp\br)\,.
\end{equation}
Furthermore,  
 \begin{multline}
    \label{eq:64}
\sup_{u\in\R^n}\bl( M(u,x)
+\frac{1}{2}\,\lambda
\abs{N(u,x)}^2+p^T\sigma(x) N(u,x)\br)
\\=    \frac{1}{2}\,\frac{1}{1-\lambda}\,
\norm{a(x)-r(x)\mathbf1-\lambda
  b(x)\beta(x)+b(x)\sigma(x)^Tp}^2_{c(x)^{-1}}
\\+\frac{1}{2}\,\lambda\abs{\beta(x)}^2+
r(x)-\alpha(x)+\frac{1}{2}\,\abs{\beta(x)}^2-\beta(x)^T\sigma(x)^Tp\,,
  \end{multline}
where, for $y\in\R^n$ and positive definite symmetric 
$n\times n$--matrix 
$\Sigma$\,, we denote $\norm{y}^2_\Sigma=y^T\Sigma y$\,.

Therefore,
on introducing
\begin{subequations}
  \begin{align}
  \label{eq:84}
  T_\lambda(x)&=\sigma(x)\sigma(x)^T+\frac{\lambda}{1-\lambda}\,\sigma(x)
  b(x)^Tc(x)^{-1}b(x)\sigma(x)^T,\\
\label{eq:84a}
S_\lambda(x)&=\frac{\lambda}{1-\lambda}\,(a(x)-r(x)\mathbf1-
\lambda
b(x)\beta(x))^Tc(x)^{-1}b(x)\sigma(x)^T-\lambda\beta(x)^T\sigma(x)^T+
\theta(x)^T\,,
\intertext{and}
\label{eq:84b}R_\lambda(x)&=\frac{\lambda}{2(1-\lambda)}\,\norm{a(x)-r(x)\ind-
\lambda b(x)\beta(x)}^2_{c(x)^{-1}}
+\lambda(r(x)-\alpha(x)+\frac{1}{2}\,\abs{\beta(x)}^2)\notag
\\&+\frac{1}{2}\,\lambda^2\abs{\beta(x)}^2\,,
\end{align}
\end{subequations}
we have that 
\begin{equation}
  \label{eq:80}
\breve   H(x;\lambda,p)=
\frac{1}{2}\,p^TT_\lambda(x)p
+S_\lambda(x)p+
R_\lambda(x)\,.
\end{equation}
Let us note that, by condition (N), $T_\lambda(x)$ is a uniformly
positive definite matrix.

If $\lambda=1$\,, then, on noting that
\begin{multline}
  \label{eq:96}
    M(u,x)
+\frac{1}{2}\,
\abs{N(u,x)}^2+ p^T\sigma(x) N(u,x)=
u^T(a(x)-r(x)\mathbf 1-b(x)\beta(x)+b(x)\sigma(x)^Tp)\\+
r(x)-\alpha(x)+\abs{\beta(x)}^2-p^T\sigma(x)\beta(x)\,,
\end{multline}
we have that $\breve H(x;1,p)<\infty$ if and only if 
\begin{equation}
  \label{eq:135}
a(x)-r(x)\mathbf 1-b(x)\beta(x)+b(x)\sigma(x)^Tp=0\,,
\end{equation}
in which case
\begin{equation}
  \label{eq:61}
  \breve H(x;1,p)=r(x)-\alpha(x)+\abs{\beta(x)}^2-p^T\sigma(x)\beta(x)+
p^T\theta(x)+\frac{1}{2}\,\abs{{\sigma(x)}^Tp}^2\,.\end{equation}
As mentioned, if  $\lambda>1$\,, then the righthand side of \eqref{eq:40} equals infinity.
Consequently, $\breve H(x;\lambda,p)$ is a lower semicontinuous
function of $(\lambda,p)$ with values in $\R\cup\{+\infty\}$\,.
By Lemma \ref{le:conc} below, $\breve H(x;\lambda,p)$ is convex in
$(\lambda,p)$\,.

 We define, given  $f\in\mathbb C^2$\,,
\begin{equation}
  \label{eq:59} 
H(x;\lambda,f)= \breve H(x;\lambda,\nabla f(x))
+\frac{1}{2}\, \text{tr}\,\bl({\sigma(x)}{\sigma(x)}^T\nabla^2 f(x)\br)\,.
\end{equation}
By the convexity of $\breve H$\,, the function
$H(x;\lambda,f)$ is convex in $(\lambda,f)$\,. 

Let 
\begin{equation}
  \label{eq:29}
  F(\lambda)=
\inf_{f\in\mathbb C^2}
\sup_{x\in\R^l}H(x;\lambda,f)\text{ if }\lambda<1\,,
\end{equation}
$F(1)=\lim_{\lambda\uparrow 1}F(\lambda)$\,, $F(\lambda)=\infty$ if $\lambda>1$\,,
and
\begin{equation*}
\overline\lambda=\sup\{\lambda\in\R:\,F(\lambda)<\infty\}\,.
\end{equation*}
 By $H(x;\lambda,f)$ being
 convex in $(\lambda,f)$\,, $F(\lambda)$ is convex for $\lambda<1$\,, 
so $F(1)$ is well
 defined, see, e.g., Theorem 7.5 on p.57 in Rockafellar \cite{Rock}.
The function  $F(\lambda)$ is  seen to be convex as a
function on $\R$\,. It is finite
 when $\lambda<\lambda_0$\,, for
 some $\lambda_0\in(0,1]$\,,
which is obtained by taking $f(x)=\kappa\abs{x}^2$\,, $\kappa>0$ being
small enough (see Lemma \ref{le:sup-comp} for more detail). 
Therefore
$\overline\lambda\in(0,1]$\,. 
  Lemma \ref{le:minmax} below establishes
that   $F(0)=0$\,, that $F(\lambda)$ is  lower semicontinuous on $\R$
and that  if $F(\lambda)$ is finite, with $\lambda<1$\,,
then the infimum in \eqref{eq:29} is attained at
 function $f^\lambda$ which
satisfies  the  equation 
\begin{equation}
  \label{eq:86}
H(x;\lambda,f^\lambda)=F(\lambda)\,, \text{ for all }
x\in\R^l\,.  
\end{equation}
Furthermore, $f^\lambda\in\mathbb C_\ell^1$\,, 
with  $\mathbb{C}^1_\ell$   representing 
 the set of real--valued 
 continuously differentiable
 functions on $\R^l$ whose gradients satisfy the linear growth
 condition. Thus, the infimum in \eqref{eq:29} can be taken over
 $\mathbb C^2\cap \mathbb C^1_\ell$ when $\lambda<1$\,.
Equation \eqref{eq:86} is dubbed  an ergodic Bellman equation,
see, e.g., Fleming and Sheu \cite{FleShe02}, Kaise and Sheu \cite{KaiShe06},
Hata, Nagai, and Sheu \cite{Hat10},
Ichihara \cite{Ich11}.

Let $\mathcal{P}$ represent the set of probability measures $\nu$
on $\R^l$ such that $\int_{\R^l}\abs{x}^2\,\nu(dx)<\infty$\,.
For $\nu\in\mathcal{P}$\,,
we let $\mathbb L^{2}(\R^l,\R^l,\nu(dx))$ 
represent the Hilbert space (of the equivalence classes)
 of $\R^l$-valued functions $h(x)$ on
$\R^l$ that are square integrable with respect to $\nu(dx)$ equipped
with the norm $\bl(\int_{\R^l}\abs{h(x)}^2\,\nu(dx)\br)^{1/2}$ and we let
 $\mathbb L^{1,2}_0(\R^l,\R^l,\nu(dx))$ 
represent the closure in 
$\mathbb L^{2}(\R^l,\R^l,\nu(dx))$ 
of the set of gradients of $\mathbb
C_0^1$-functions.
 We will retain the notation 
$\nabla f$ for the elements of $\mathbb
L^{1,2}_0(\R^l,\R^l,\nu(dx))$\,, although those functions might not
be proper gradients.
Let 
$\mathcal{U}_\lambda$ denote the  set of functions
$f\in \mathbb C^2\cap \mathbb C^1_\ell$ such
that
$\sup_{x\in\R^l}H(x;\lambda,f)<\infty$\,.
The set $\mathcal{U}_\lambda$ is nonempty if and only if $F(\lambda)<\infty$\,. 
It is convenient to write \eqref{eq:29} in the form, cf. \eqref{eq:118},
\begin{equation}
  \label{eq:136}
    F(\lambda)=\inf_{f\in\mathcal{U}_\lambda}
\sup_{\nu\in\mathcal{P}}\int_{\R^l}H(x;\lambda,f)\,\nu(dx)\,,\quad\text{if }\lambda<1,
\end{equation}
the latter integral possibly being equal to $-\infty$\,.
We  adopt the convention that  $\inf_\emptyset=\infty$\,, so that
\eqref{eq:136} holds when $\mathcal{U}_\lambda=\emptyset$ too.
Let  $\mathbb C_b^2$ represent the subset of $\mathbb C^2$ 
of functions with bounded second derivatives.
Let, for  $f\in\mathbb  C_b^2$ and $m\in\mathbb P$\,, 
\begin{align}
  \label{eq:62}
  G(\lambda,f,m)=&\int_{\R^l}H(x;\lambda,f)\,m(x)\,dx\,.
\end{align}
This function is well defined, 
is convex in $(\lambda,f)$ and is concave in $m$\,.
By Lemma \ref{le:conc} and Lemma \ref{le:saddle_3} below,
for $\lambda<\overline\lambda$\,, 
$F(\lambda)=\sup_{m\in\hat{\mathbb P}}
\inf_{f\in\mathbb C_0^2}G(\lambda,f,m)$\,. One can  replace
$\hat{\mathbb P}$ with $\mathbb P$ in the preceding $\sup$ and replace
$\mathbb C_0^2$ with $\mathbb C_b^2$ in the preceding $\inf$.
If  $m\in\hat{\mathbb P}$\,, then  integration by parts in
\eqref{eq:62}    obtains that, for $f\in\mathbb C_b^2$\,,
\begin{equation}
  \label{eq:12}
 G(\lambda,f,m)=\breve G(\lambda,\nabla f,m)\,,
\end{equation}
where
\begin{align}
  \label{eq:11}
  \breve G(\lambda,\nabla f,m)
=&\int_{\R^l}\Bl(\breve H(x;\lambda,\nabla f(x))
-\frac{1}{2}\, \nabla f(x)^T
\,\frac{\text{div}\,({\sigma(x)}{\sigma(x)}^T\,
 m(x))}{m(x)}\,
\Br)\,m(x)\,dx\,.
\end{align}
(Unless specifically mentioned otherwise,
it is assumed throughout that $0/0=0$\,. More detail 
on  the integration by parts is given
  in the proof of Lemma \ref{le:minmax}.)
The function $\breve G(\lambda,\nabla f,m)$ is convex in
$(\lambda,f)$ and is concave in $m$\,. 
The righthand side of \eqref{eq:11} being well defined 
for $\nabla f\in\mathbb L_0^{1,2}(\R^l,\R^l,m(x)\,dx)$\,, we adopt \eqref{eq:11}
as the definition of $\breve G(\lambda,\nabla f,m)$ for 
$(\lambda,\nabla f,m)\in \R\times \mathbb
L_0^{1,2}(\R^l,\R^l,m(x)\,dx)\times
\hat{\mathbb P}$\,.


Let, for $m\in\hat{\mathbb P}$\,,
\begin{equation}
  \label{eq:13}
  \breve F(\lambda,m)=\inf_{\nabla f\in\mathbb L^{1,2}_0(\R^l,\R^l,m(x)\,dx)}
\breve G(\lambda,\nabla f,m)\,,
\end{equation}
when $\lambda\le1$ and let $\breve F(\lambda,m)=\infty$\,, for
$\lambda>1$\,. By Lemma \ref{le:conc} below,
the infimum in \eqref{eq:13} is 
attained uniquely, if finite, the latter always being the case for
$\lambda<1$\,.  Furthermore,  if $\lambda<1$\,, then
$\breve F(\lambda,m)
=\inf_{f\in\mathbb C_0^2}
G(\lambda,f,m)\,$.
By \eqref{eq:11}, the function $\breve F(\lambda,m)$
is convex in
$\lambda$ and is concave in $m$\,. It is lower semicontinuous in
$\lambda$ and is strictly convex on $(-\infty,1)$
  by Lemma \ref{le:conc},  so, by convexity, see
Corollary 7.5.1 on p.57 in Rockafellar \cite{Rock},
$  \breve F(1,m)=\lim_{\lambda\uparrow1}\inf_{f\in\mathbb C_0^2}
G(\lambda,f,m)$\,.
By Lemma \ref{le:saddle_3} below,
$\lambda q-\breve F(\lambda,m)$  has saddle point 
$(\hat\lambda,\hat m)$ 
in $(-\infty,\overline\lambda]\times\hat{\mathbb P}$\,, with $\hat\lambda$ being specified
uniquely, and
the supremum of $\lambda q-F(\lambda)$ over $\R$ is attained at
$\hat\lambda$\,.

If $\hat\lambda<1$\,, which is ''the regular case'', then
$\hat m$ is specified uniquely and there exists $\hat f\in \mathbb
C^2\cap \mathbb C^1_\ell$ such that 
$(\hat\lambda,\hat f,\hat m)$ 
is  a  saddle point
of the function $\lambda q-\breve G(\lambda,\nabla f,m)$
in $\R\times
(\mathbb C^2\cap\mathbb C^1_\ell)\times \hat{\mathbb{P}}$\,,
with $\nabla \hat f$  being specified uniquely.
As a matter of fact, $\hat f=f^{\hat \lambda}$\,, 
so the function $\hat f$  satisfies the
ergodic Bellman equation
\begin{equation}
  \label{eq:103'}
  H(x;\hat\lambda,\hat f)=F(\hat\lambda)\,,
\text{ for all $x\in\R^l$\,.}
\end{equation}
The density $\hat m$ is the invariant density of a diffusion
process in that
  \begin{multline}
  \label{eq:104'}
                 \int_{\R^l}
\bl(\nabla  h(x)^T(\hat\lambda \sigma(x) N(\hat u(x),x)+\theta(x)
+\sigma(x)\sigma(x)^T\nabla\hat  f(x))
+\frac{1}{2}\, \text{tr}\,(\sigma(x)\sigma(x)^T
\,\nabla^2  h(x))
\br)\\ \hat m(x)\,dx=0\,,
\end{multline}
for all $h\in\mathbb C_0^2$\,.
Essentially, equations \eqref{eq:103'} and \eqref{eq:104'} represent
Euler--Lagrange equations for $\breve G(\hat\lambda,\nabla f,m)$
at $(\hat f,\hat m)$\,. They specify  $\nabla\hat f$ and $\hat m$ 
uniquely and  imply that $(\hat f,\hat m)$ is a
saddle point of $\breve G(\hat\lambda,\nabla f,m)$\,, cf., Proposition 1.6 on
p.169 in Ekeland and Temam \cite{EkeTem76}.
We  define  $\hat u(x)$ 
as the $u$ that attains supremum in
\eqref{eq:40} for $\lambda=\hat\lambda$ and $p=\nabla\hat f(x)$ 
so that, by \eqref{eq:39},
\begin{equation}
  \label{eq:69}
        \hat u(x)=\frac{1}{1-\hat\lambda}\,c(x)^{-1}\bl(a(x)-r(x)\mathbf1
-\hat\lambda b(x)\beta(x)+b(x)\sigma(x)^T\nabla\hat
  f(x)\br)\,.
\end{equation}

Suppose that   $\hat\lambda=1$\,, which is ''the degenerate case''.
Necessarily, $\overline\lambda=1$\,,
so, the infimum on the righthand side of
 \eqref{eq:13} for $\lambda=1$ and $m=\hat m$ is finite and  is
 attained at unique $\nabla \hat f$\,(see Lemma \ref{le:conc}).
Consequently, $F(1)<\infty$\,.
According to Lemma \ref{le:saddle_3} below, cf., \eqref{eq:135} and \eqref{eq:104'},
\begin{equation}
  \label{eq:134}
a(x)-r(x)\mathbf 1-b(x)\beta(x)+b(x)\sigma(x)^T\nabla\hat f(x)=0
\quad\hat
m(x)dx\text{--a.e.}  \end{equation}
  and
  \begin{equation}
    \label{eq:137}
        \int_{\R^l}\bl(
\nabla h(x)^T\bl(
-\sigma(x)\beta(x)+\theta(x)+\sigma(x)\sigma(x)^T\nabla\hat f(x)\br)
+\frac{1}{2}\,
\text{tr}\,\bl(\sigma(x)\sigma(x)^T\nabla^2 h(x)\br)\br)\hat m(x)\,dx=0\,,
\end{equation}
provided  that 
$h\in\mathbb C_0^2$ and $b(x)\sigma(x)^T\nabla h(x)=0$
$\hat m(x)\,dx$--a.e.
By \eqref{eq:96}, the value of 
the expression in the supremum in \eqref{eq:40} 
 does not depend on the
choice of $ u$ when $\lambda=1$ and $p=\nabla \hat f(x)$\,, so,
there is some leeway as to the choice of  an optimal control.
As the concave function $\lambda q-
   \breve F(\lambda,\hat m)$ 
attains maximum at $\lambda=1$\,, 
 $d/d\lambda\,\breve F(\lambda,\hat m)\Big|_{1-}\le q$\,, 
with 
$d/d\lambda\,\breve F(\lambda,\hat m)\Big|_{1-}$ standing for the lefthand derivative of 
$\breve F(\lambda,\hat m)$
 at $\lambda=1$\,.
Hence,
there exists  bounded continuous function
 $\hat v(x)$  with values in
 the range
of $b(x)^T$  such that
$\abs{\hat v(x)}^2/2=q-d/d\lambda\,\breve F(\lambda,\hat m)\Big|_{1-}$\,.
(For instance, one can take
$  \hat v(x)=b(x)^Tc(x)^{-1/2}\,z\,
\sqrt{2(q-d/d\lambda\,\breve F(\lambda,\hat m)\Big|_{1-})}\,,
$ where $z$ represents an   element of $\R^n$ of  length one.)
We  let
$\hat u(x)=
c(x)^{-1}b(x)(\beta(x)+\hat v(x))$\,.

In either case, we define 
$  \hat\pi_t=\hat u(X_t)$ and, given $\rho>0$\,,
$  \hat\pi^\rho_t=\hat u^\rho(X_t)$\,, where 
$\hat u^\rho(x)=\hat u (x)\chi_{[0,\rho]}(\abs{x})$\,.
We introduce, given $\lambda\in\R$\,, 
$f\in \mathbb C^2$\,, 
 $m\in \mathbb P$\,, and  measurable $\R^n$--valued 
function $v=(v(x)\,,x\in\R^l)$\,,
\begin{multline}
  \label{eq:53}
    \overline H(x;\lambda,f,v)= 
\lambda  M(v(x),x)
+\frac{1}{2}\,\abs{\lambda
N(v(x),x)+{\sigma(x)}^T\nabla f(x)}^2
+\nabla f(x)^T\,\theta(x)\\
+\frac{1}{2}\, \text{tr}\,\bl({\sigma(x)}{\sigma(x)}^T\nabla^2 f(x)\br)\,.
\end{multline}
By \eqref{eq:40},  \eqref{eq:59},
\eqref{eq:62},  \eqref{eq:69}, and  \eqref{eq:53}, if
$\hat\lambda<1$\,, then
\begin{equation}
  \label{eq:47}
F(\hat\lambda)= H(x;\hat\lambda,\hat f)=\overline
H(x;\hat\lambda,\hat f,\hat u)=\inf_{f\in\mathbb C^2}\sup_{x\in\R^l}\overline
H(x;\hat\lambda, f,\hat u)\,.
\end{equation}
Let
\begin{subequations}
  \begin{align}
  \label{eq:30}
  J_q=&\sup_{\lambda\le1}(\lambda q-F(\lambda))\,,\\
  \label{eq:38}
   J^{\text{o}}_q=& 
 \sup_{\lambda\in[0,1]}
(\lambda q-F(\lambda))\,,
\intertext{and}
  \label{eq:36}
   J^{\text{s}}_q=&  \sup_{\lambda\le0}(\lambda q-F(\lambda))\,.
\end{align}
\end{subequations}
 It is  noteworthy that if
$\hat\lambda<0$\,, then $J^{\text{s}}_q>0$
and $J^{\text{0}}_q=0$\,, while if $\hat\lambda>0$\,,
  then $J^{\text{o}}_q>0$ and $J^{\text{s}}_q=0$\,.

We are in a position to state the first limit theorem.
\begin{theorem}
  \label{the:bounds}
  \begin{enumerate}
  \item 
For arbitrary portfolio $\pi=(\pi_t,\,t\ge0)$\,,
\begin{equation}
  \label{eq:58}
    \liminf_{t\to\infty}\frac{1}{t}\ln
\mathbf{P}(L^\pi_t< q)\ge -J^{\text{s}}_{q}\,.
\end{equation}
If, in addition, $\abs{X_0}$ is bounded and $ f^\lambda(x)$ is bounded below
by an affine function of $x$ when
$0<\lambda<\overline\lambda$\,, then
 \begin{equation}
   \label{eq:60}
     \limsup_{t\to\infty}\frac{1}{t}\ln
\mathbf{P}(L^\pi_t\ge q)\le - J^{\text{o}}_{q}\,.
\end{equation}
\item
 The following asymptotic bound holds:
\begin{equation}
        \label{eq:27}
      \liminf_{t\to\infty}\frac{1}{t}\ln
\mathbf{P}(L^{\hat\pi}_t> q)\ge -J^{\text{o}}_{q}\,.
\end{equation}
If, in addition,
\begin{equation}
  \label{eq:97}
  \limsup_{\rho\to\infty}
\inf_{f\in \mathbb C^2}\sup_{x\in \R^l}
\overline H(x;\hat\lambda,f,\hat u^\rho)\le F(\hat\lambda)
\end{equation}
when $\hat\lambda<0$\,,
then
\begin{equation}
    \label{eq:9}
\limsup_{\rho\to\infty}\limsup_{t\to\infty}\frac{1}{t}\ln
\mathbf{P}(L^{\hat\pi^\rho}_t\le q)\le -J^{\text{s}}_{q}\,.
\end{equation}
  \end{enumerate}
\end{theorem}
\begin{remark}
  The upper bounds in \eqref{eq:60} and in \eqref{eq:9} are of
  interest only if $\hat\lambda>0$ and $\hat\lambda<0$\,, respectively.
\end{remark}
\begin{remark}
   The assertions of
  Theorem
 \ref{the:bounds}  hold in the case where
$\beta(x)=0$ too, provided $\inf_{x\in\R^l}r(x)<q$\,.
If $\inf_{x\in\R^l}r(x)\ge q$\,, then investing  in the safe security only
is obviously optimal.
\end{remark}
\begin{remark}
  The requirement that $ f^\lambda(x)$ be bounded below 
by an affine function
when $0<\lambda<\overline\lambda$  is fulfilled 
for the Gaussian model, as we discuss below.
\end{remark}
A sufficient condition for  \eqref{eq:97}  to hold is given by the
next lemma which features a condition introduced by Nagai
\cite{Nag12}, see also Puhalskii and Stutzer \cite{PuhStu16}.
The proof is relegated to the appendix.
 \begin{lemma}
  \label{le:condition}
Suppose that 
there exist $\varrho>0$\,, $C_1>0$ and $C_2>0$ such that, for all $x\in\R^l$\,,
\begin{equation}
  \label{eq:31} (1+\varrho)   \norm{b(x)\sigma(x)^T\nabla\hat
  f(x)}^2_{c(x)^{-1}}
-\norm{a(x)-r(x)\mathbf1}^2
_{c(x)^{-1}}\le C_1\abs{x}+C_2\,.
\end{equation}
Then \eqref{eq:97} holds for $\hat\lambda<0$\,.
\end{lemma}
\begin{remark}
As the proof  shows, an upper bound on  the righthand side of \eqref{eq:31} can be
  allowed to grow at a subquadratic rate.
\end{remark}
\begin{remark}
  The inequality in \eqref{eq:31} holds provided
  \begin{equation}
    \label{eq:92}
        \limsup_{\abs{x}\to\infty}
\frac{1}{\abs{x}^2}\bl(\norm{b(x)\sigma(x)^T\nabla\hat
  f(x)}^2_{c(x)^{-1}}
-\norm{a(x)-r(x)\mathbf1}^2
_{c(x)^{-1}}\br)<0\,.
  \end{equation}
It holds also if $b(x)\sigma(x)^T=0$ which means that the Wiener
processes effectively driving the security prices and the economic
factor process are independent.
\end{remark}

The following theorem shows that the portfolio 
$\hat\pi=(\hat\pi_t,\,t\ge0)$ is 
 risk--sensitive optimal for suitable $q$\,.
If 
 $F$ is subdifferentiable at $\lambda$\,, we let
$u^\lambda(x)$ represent the function $\hat u(x)$
for a value of
 $q$ that is a subgradient of $F$ at $\lambda$\,.
We also let
$u^{\lambda,\rho}(x)=u^\lambda(x)\chi_{[0,\rho]}(\abs{x})$\,,
$\pi^\lambda_t=u^\lambda(X_t)$\,, $\pi^{\lambda,\rho}_t=u^{\lambda,\rho}(X_t)$\,,
 $\pi^\lambda=(\pi^\lambda_t,\,t\ge0)$\,,
and $\pi^{\lambda,\rho}=(\pi^{\lambda,\rho}_t,\,t\ge0)$\,.
The function $F$ is subdifferentiable at $\lambda<\overline\lambda$\,.
It might not be subdifferentiable at
$\overline\lambda$\,.

\begin{theorem}
  \label{the:risk-sens}
  \begin{enumerate}
\item 
If  $0<\lambda<\overline\lambda$\,, 
if the function
$f^{\lambda(1+\epsilon)}(x)$  is bounded below by an affine function of $x$ when 
$\epsilon$ is small enough,
 and if $\abs{X_0}$ is bounded, 
 then,
  for any portfolio $\pi=(\pi_t\,,t\ge0)$\,, 
\begin{equation*}
  \limsup_{t\to\infty}\frac{1}{t}\,
\ln \mathbf Ee^{\lambda t L^\pi_t}\le  F(\lambda)\,.
\end{equation*}
If either  $0<\lambda<\overline\lambda$
or $\lambda=\overline\lambda$ and $F$ is subdifferentiable at
$\overline\lambda$\,, then 
\begin{equation*}
  \liminf_{t\to\infty}\frac{1}{t}\,
\ln \mathbf Ee^{\lambda t L^{\pi^\lambda}_t}\ge F(\lambda)\,.
\end{equation*}
If either $\lambda=\overline\lambda$ and
  $F$ is not subdifferentiable at
$\overline\lambda$ or $\lambda>\overline\lambda$\,, 
then there exists portfolio $\pi^\lambda$ such that 
\begin{equation*}
  \liminf_{t\to\infty}\frac{1}{t}\,
\ln \mathbf Ee^{\lambda t L^{\pi^\lambda}_t}\ge F(\lambda)\,.
\end{equation*}
  \item 
If  $\lambda<0$\,, then, for any portfolio $\pi=(\pi_t\,,t\in\R_+)$\,,
\begin{equation*}
  \liminf_{t\to\infty}\frac{1}{t}\,
\ln \mathbf Ee^{\lambda t L^\pi_t}\ge F(\lambda)
\end{equation*}
and, provided \eqref{eq:97} holds with $\hat\lambda=\lambda$ and
$\hat u^\rho=u^{\lambda,\rho}$ and $\abs{X_0}$ is bounded,
\begin{equation*}
\lim_{\rho\to\infty}  \liminf_{t\to\infty}\frac{1}{t}\,
\ln \mathbf Ee^{\lambda t L^{\pi^{\lambda,\rho}}_t}=
\lim_{\rho\to\infty}  \limsup_{t\to\infty}\frac{1}{t}\,
\ln \mathbf Ee^{\lambda t L^{\pi^{\lambda,\rho}}_t}= F(\lambda)\,.
\end{equation*}
  \end{enumerate}
\end{theorem}
\begin{remark}
  We recall that $F(\lambda)=\infty$ if $\lambda>\overline\lambda$\,. 
For a one--dimensional model, $\overline\lambda$ is found explicitly in
  Pham \cite{Pha03}, also, see the appendix below. We conjecture  that $F$ is 
  differentiable and strictly convex 
for $\lambda<\overline \lambda$\,, which would imply
  that $\pi^\lambda$ is specified uniquely. This is provably the case
  for the model of Pham \cite{Pha03} and provided $\lambda<0$\,, see
  Pham \cite{Pha03} and Puhalskii and Stutzer \cite{PuhStu16}, respectively.
\end{remark}

If we   assume that
the functions $a(x)$\,, $r(x)$\,,  
$\alpha(x)$\, and $\theta(x)$ are affine functions of $x$ and that the diffusion
coefficients are constant, then fairly  explicit formulas are available. 
More specifically, let
\begin{subequations}
  \begin{align}
    \label{eq:85}
a(x)=A_1x+a_2\,, \\
 \label{eq:85a}
r(x)=r_1^Tx+r_2\,,\\\label{eq:85b}
\alpha(x)=\alpha_1^Tx+\alpha_2\,, \\
\label{eq:85d}
\theta(x)=\Theta_1 x+\theta_2\,,
\intertext{and}
\label{eq:85c}
b(x)=b,\;\beta(x)=\beta,\;\sigma(x)=\sigma\,,
  \end{align}
\end{subequations}
where $A_1\in \R^{n\times l}$\,, $a_2\in\R^n$\,, 
$r_1\in\R^l$\,, $r_2\in \R$\,, $\alpha_1\in \R^l$\,,
$\alpha_2\in\R$\,, $\Theta_1$ is a negative definite $l\times
l$-matrix,
 $\theta_1\in \R^l$\,, $b$ is an $n\times k$-matrix such that
the matrix $bb^T$ is positive definite, $\beta$ is a non-zero
$k$-vector,
 and
$\sigma$ is an $l\times k$-matrix such that the matrix $\sigma\sigma^T$ is
positive definite.
Condition (N)
expresses the requirement 
 that the ranges of $\sigma^T$ and $b^T$ have the trivial
 intersection  and that $\beta$ is not an element of the sum of
 those ranges. 

Finding the 
optimal portfolio
$\hat\pi_t$ may be reduced to
solving an algebraic Riccati equation.
We introduce, for $\lambda<1$\,,
\begin{align*}
  A(\lambda)&=\Theta_1+\frac{\lambda}{1-\lambda}\,\sigma b^Tc^{-1}(A_1-\mathbf 1
  r_1^T),\\
B(\lambda)&=T_\lambda(x)=\sigma\sigma^T+\frac{\lambda}{1-\lambda}\,\sigma
  b^Tc^{-1}b\sigma^T\,, \intertext{and}
C&=
\norm{A_1-\mathbf1 r_1^T}^2_{c^{-1}}\,.
\end{align*}
Let us  suppose that there exists
symmetric  $l\times l$--matrix $ P_1(\lambda)$ 
that satisfies the algebraic Riccati  equation
\begin{equation}
  \label{eq:78}
   P_1(\lambda)B(\lambda) P_1(\lambda)+A(\lambda)^T P_1(\lambda)
+ P_1(\lambda)A(\lambda)+
    \frac{\lambda}{1-\lambda}\, C=0\,.
\end{equation}
  Conditions
for the existence of solutions can be found
in Fleming and Sheu \cite{FleShe02}, see also
  Willems \cite{MR0308890}
 and  Wonham \cite{MR0239161}.
 According to Lemma 3.3 in Fleming and Sheu \cite{FleShe02},
provided that $\lambda< 0$\,,  there exists 
unique $ P_1(\lambda)$ solving \eqref{eq:78} such that $ P_1(\lambda)$ is
 negative semidefinite.
Furthermore, 
  the matrix 
  \begin{equation}
    \label{eq:76}
    D(\lambda)=A(\lambda)+B(\lambda)P_1(\lambda)
\end{equation}
is stable.
If $0<\lambda<1$ and $F(\lambda)<\infty$\,, 
then, by Lemma 4.3 in Fleming and Sheu
\cite{FleShe02}, there exists  unique $ P_1(\lambda)$
solving \eqref{eq:78} such that $ P_1(\lambda)$ is 
positive semidefinite and $D(\lambda)$ is semistable.
By Theorem 4.6
in Fleming and Sheu \cite{FleShe02}, 
the matrix $D(\lambda)$ is stable if $\lambda$ is small enough.

With $D(\lambda)$ being stable, the equation
\begin{equation}
  \label{eq:79}
D(\lambda)^Tp_2(\lambda)
+ 
 E(\lambda)=0
\end{equation}
has a unique solution for $ p_2(\lambda)$\,, where
\begin{equation}
  \label{eq:139a}
E(\lambda)=   \frac{ \lambda}{1-\lambda}\,
(A_1-\mathbf
1r_1^T+b\sigma^T P_1(\lambda))^Tc^{-1}(a_2-r_2\mathbf1-\lambda
b\beta)
\\+\lambda(r_1-\alpha_1- P_1(\lambda)\sigma\beta)+
 P_1(\lambda)\theta_2\,.
\end{equation}
 Substitution shows that 
$H(x;\lambda,\tilde f^\lambda)$\,, with $\tilde f^\lambda(x)=  x^T P_1(\lambda)x/2+ p_2(\lambda)^T x$\,, does not depend on $x$\,.
Let $m^\lambda$ denote the invariant distribution of
 the linear diffusion
\begin{multline}
  \label{eq:75}
    dY_t=D(\lambda)Y_t\,dt
+\bl(\frac{\lambda}{1-\lambda}\,\sigma b^Tc^{-1}
(a_2-r_2\mathbf1-\lambda b\beta
+b\sigma^T p_2(\lambda))-\lambda\sigma\beta
+\sigma\sigma^T p_2(\lambda)+\theta_2\br)\,dt\\+\sigma\,dW_t\,.
\end{multline}
Then the pair $(\tilde f^\lambda,m^\lambda)$ is a saddle point of 
$\breve G(\lambda,\nabla f,m)$ as well as of $G(\lambda,f,m)$
 considered as  functions of $(f,m)\in\mathcal{U}_\lambda
\times\hat{ \mathbb P}$\,.
Hence, 
\begin{multline*}
    H(x;\lambda,f^\lambda)=\breve G(\lambda,\nabla f^\lambda,m^\lambda)=
\inf_{f\in\mathcal{U}_\lambda}\sup_{m\in\hat{\mathbb
    P}}\breve G(\lambda,\nabla f,m)
=\inf_{f\in\mathcal{U}_\lambda}\sup_{m\in\mathbb     P}
G(\lambda,f,m)\\=\inf_{f\in\mathcal{U}_\lambda}\sup_{x\in\R^l}
H(x;\lambda,f)=F(\lambda)\,,
\end{multline*}
so $\tilde f^\lambda$ satisfies the Bellman equation \eqref{eq:86}.
As a result, under the hypotheses of Fleming and Sheu
\cite{FleShe02},
 $\tilde f^\lambda$ is bounded below by an affine function
when $\hat\lambda\in(0,1)$\,. Condition \eqref{eq:92} is implied by
the condition that the matrix
$(b\sigma^TP_1(\hat\lambda))^Tc^{-1}b\sigma^TP_1(\hat\lambda)-
(A_1-\mathbf 1r_1^T)^Tc^{-1}(A_1-\mathbf 1r_1^T)$ is negative definite.

Furthermore, one can see that
\begin{multline}
  \label{eq:155}
    F(\lambda)=\frac{1}{2}\,\norm{p_2(\lambda)}^2_{\sigma\sigma^T}
+\frac{1}{2}\,
\frac{\lambda}{1-\lambda}\,
\norm{a_2-r_2\mathbf 1-\lambda
  b\beta+b\sigma^T p_2(\lambda)
}^2_{c^{-1}}\\+
(
-\lambda\beta^T\sigma^T+\theta_2^T)p_2(\lambda)
+\lambda(r_2-\alpha_2+\frac{1}{2}\,\abs{\beta}^2)+
\frac{1}{2}\,\lambda^2\abs{\beta}^2+
\frac{1}{2}\,\text{tr}\,(\sigma\sigma^TP_1(\lambda))\,.
\end{multline}
If $\hat\lambda<1$\,, equation \eqref{eq:69}  is as follows
\begin{equation*}
     \hat    u(x)=\frac{1}{1-\hat\lambda}\,c^{-1}\bl(A_1-\mathbf 1r_1^T
+b\sigma^T  P_1(\hat\lambda))x+
\frac{1}{1-\hat\lambda}\,c^{-1}\bl(a_2-r_2\mathbf1-
\hat\lambda b\beta+b\sigma^T p_2(\hat\lambda)\br)\,.
\end{equation*}
and 
$J_q=F(\hat\lambda)$\,.
If $\hat\lambda=1$\,, then one may
   look, once again, for $\hat f(x)=x^TP_1(1)x/2+p_2(1)^Tx$\,.  
Substitution in \eqref{eq:134} yields
\begin{subequations}
  \begin{align}
  \label{eq:130}
A_1-\mathbf 1 r_1^T+b\sigma^T P_1(1)=0\,,\\
  \label{eq:133}
a_2-r_2\mathbf 1-b\beta+b\sigma^T p_2(1)=0\,.
\end{align}
\end{subequations}
(One can also obtain \eqref{eq:130} by multiplying \eqref{eq:78}
through with $1-\lambda$ and taking a formal limit as
$\lambda\uparrow1$.)
If those conditions hold, choosing $\hat f(x)$ quadratic is justified. 
An optimal control is $\hat u(x)=c^{-1}(b\beta+\hat v)$\,, with $\hat
v$ coming from the range of $b^T$ and with $\abs{\hat v}^2/2=
q-d/d\lambda\, \breve F(\lambda,\hat m)\Big|_{1-}$\,. 

With  $\tilde\lambda$ representing the supremum of $\lambda$ such that
$P_1(\lambda)$ exists and  $D(\lambda)$ is stable, one  has that 
$\tilde \lambda \le\overline\lambda$\,. 
 Pham \cite{Pha03} shows that,
in the one--dimensional case, under broad assumptions,
$\tilde\lambda=\overline\lambda$ and
$F(\lambda)$ is differentiable on $(-\infty,\overline\lambda)$\,, 
 both cases that $\overline\lambda<1$ and
$\overline \lambda=1$ being realisable. 
The hypotheses in Pham \cite{Pha03}, however, rule out the possibility that
$\hat\lambda=1$\,.
In the appendix, we complete  the analysis of Pham \cite{Pha03} so that
 the case where $\hat\lambda=1$ is realised too.

Bounds \eqref{eq:58}
  and \eqref{eq:9}    of Theorem \ref{the:bounds}
are available in Puhalskii and Stutzer \cite{PuhStu16} who use a
different definition of $H(x;\lambda,f)$\,. They also assume a
more general stability condition than in \eqref{eq:45} for \eqref{eq:58} and
provide more detail on the relation to earlier results for the
underperformance probability optimisation.
Theorem \ref{the:bounds}  improves on the results in Puhalskii
\cite{Puh11} by doing away with a certain growth requirement on 
$\abs{\pi_t}$ (see (2.12) in Puhalskii \cite{Puh11}).
 Maximising the probability of
outperformance for a one-dimensional model
 is studied in Pham \cite{Pha03}, who, however, stops short of proving
  the asymptotic optimality of
 $\hat\pi$ and produces nearly optimal portfolios instead. 
Besides, the requirements  in Pham \cite{Pha03}  amount to
 $F(\lambda)$ being essentially smooth,  the
portfolio's wealth   growing no 
faster  than linearly with the economic factor
(see 
condition in (2.5) in Pham \cite{Pha03}) and 
 $\theta_2=0$\,. On the other hand, 
it is not assumed  in  Pham \cite{Pha03} that
$\beta$ does not belong to the sum of the ranges of $b^T$ and $\sigma^T$\,,
which property is required by our  condition (N).

Most of the results on the  risk--sensitive optimisation concern the
case of a negative Hara parameter.  
Theorem 4.1 in Nagai \cite{Nag03}  obtains asymptotic optimality of
$\pi(\lambda)$\,, rather than asymptotic $\epsilon$--optimality,
 for a nonbenchmarked setup under a number of additional
conditions, e.g.,   the interest rate is bounded and the following
version of \eqref{eq:31} is required:
$\norm{b(x)\sigma(x)^T\nabla\hat
  f(x)}^2_{c(x)^{-1}}
-\norm{a(x)-r(x)\mathbf1}^2
_{c(x)^{-1}}\to -\infty$\,, as $\abs{x}\to\infty$\,. (Unfortunately,
there are pieces of undefined notation such as $u(0,x;T)$\,.)
Affine models are considered in
  Bielecki and Pliska \cite{BiePli99},
\cite{BiePli04}, 
Kuroda and Nagai \cite{KurNag02}, for the nonbenchmarked case,
 and  Davis and Lleo \cite{DavLle08}, for
the benchmarked case.
 Fleming and Sheu  \cite{FleShe00}, \cite{FleShe02}
allow $\lambda$ to assume either sign. Although the latter authors correctly
 identify the limit quantity in Theorem
\ref{the:risk-sens}  as the righthand side of
an ergodic Bellman equation,
they  prove  neither that $F(\lambda)$ is  the limit of $(1/t)\ln \mathbf
Ee^{\lambda t L^{\pi^\lambda}_t}$ nor that $F(\lambda)$ is an asymptotic
bound for an arbitrary portfolio. Rather, they prove that
$F(\lambda)$  can be obtained as the limit of the 
optimal growth rates associated with bounded portfolios as the bound
constraint is being relaxed. They also require that $\lambda$ be
sufficiently small, if positive.
The assertion of part 1 of  Theorem
\ref{the:risk-sens} has not been available in this generality even for
the affine model, Theorem 4.1 in Pham
\cite{Pha03}  
 tackling a  case of one security.

There is another notable distinction of our results.
 It concerns the stability condition \eqref{eq:45}
 on the economic factor process. In some of the literature, 
similar conditions involve both the parameters of the factor process
and of the security price process.
For the general model in Nagai \cite{Nag12},
it is      of the form
  $\limsup_{\abs{x}\to\infty}\,
\bl(\theta(x)-\sigma(x)b(x)^Tc(x)^{-1}(a(x)-r(x)\mathbf1
)\br)^T/\abs{x}^2<0$\,, 
 for the
Gaussian model in Hata, Nagai, and Sheu \cite{Hat10},
 it is required that
that the matrix $\Theta_1-\sigma b^Tc^{-1}A_1$ be stable.
 It appears as though that imposing a stability condition on the
factor process only is more in line with the logic of the model.
A similar form of the stability condition to ours appears in Fleming
and Sheu \cite{FleShe02}. 
\section{Technical preliminaries}
\label{sec:prelim}
In this section, we lay the groundwork for the proofs of the main results.
Drawing on  Bonnans and Shapiro \cite{BonSha00}
 (see  p.14 there), we will say that 
function $h:\,\mathbb T\to\R$\,, with $\mathbb T$ representing a topological space, is 
$\inf$--compact (respectively, $\sup$--compact) if the sets
$\{x\in \mathbb T:\,h(x)\le \delta\}$ (respectively, 
the sets $\{x\in \mathbb T:\,h(x)\ge \delta\}$) are compact for all
$\delta\in\R$\,. (It is worth noting that
Aubin \cite{Aub93} and  Aubin and
Ekeland \cite{AubEke84} adopt a slightly different terminology by
requiring only that the sets $\{x\in \mathbb T:\,h(x)\le \delta\}$ be
relatively compact in order for $h$ to be  $\inf$--compact.
Both definitions are  equivalent if $h$ is, in addition, lower semicontinuous.)

We endow 
the set $\mathcal{P}$  of probability measures $\nu$ on $\R^l$
such that $\int_{\R^l}\abs{x}^2\,\nu(dx)<\infty$
 with the Kantorovich--Rubinstein distance 
\begin{equation*}
  d_1(\mu,\nu)=\sup\{\abs{\int_{\R^l}g(x)\,\mu(dx)-\int_{\R^l}g(x)\,\nu(dx)}:\;
\frac{\abs{g(x)-g(y)}}{\abs{x-y}}\le 1\text{ for all }x\not=y\}\,.
\end{equation*}
Convergence with respect to $d_1$ is equivalent to weak
convergence coupled with  convergence of first moments, see, e.g.,
Villani \cite{Vil09}. 
For $\kappa>0$\,, let
  $f_\kappa(x)=\kappa\abs{x}^2/2$\,, where $\kappa>0$ and
$x\in\R^l$\,, and 
 let  $\mathcal{A}_\kappa$ represent the convex hull of
$\mathbb C_0^2$ and of the function 
${f}_\kappa$\,. 
\begin{lemma}
  \label{le:sup-comp}
There exist $\kappa_0>0$ and $\lambda_0>0$ such that
if $\kappa\le\kappa_0$ and $\lambda\le\lambda_0$\,, then
 the functions
$\int_{\R^l}H(x;\lambda, f_\kappa)\,\nu(dx)$ 
and 
$\inf_{f\in\mathcal{A}_\kappa}\int_{\R^l}H(x;\lambda, f)\,\nu(dx)$
are $\sup$--compact 
in $\nu\in \mathcal{P}$ for the Kantorovich--Rubinstein distance $d_1$\,.
\end{lemma}
\begin{proof}
By \eqref{eq:80} and \eqref{eq:59}, for $\lambda<1$\,,
\begin{equation*}
    H(x;\lambda,  f_\kappa)=
\frac{\kappa^2}{2}\,x^T
T_\lambda(x) x
+\kappa S_\lambda(x)x+R_\lambda(x)+\text{tr}(\sigma(x)\sigma(x)^T)\,.
%
\end{equation*}
By \eqref{eq:45}, \eqref{eq:84}, \eqref{eq:84a}, and
\eqref{eq:84b},  as $\abs{x}\to\infty$\,, if $\kappa$ is small, then
the dominating term in $(\kappa^2/2)\,x^T
T_\lambda(x) x$ is of order $\kappa^2\abs{x}^2$\,,
 the dominating terms in 
$\kappa S_\lambda(x)x$ are of orders
$(\lambda/(1-\lambda))\,\kappa\abs{x}^2$ and $-\kappa\abs{x}^2$\,, and the
dominating term in $R_\lambda(x)$ is of order 
$(\lambda/(1-\lambda))\,\abs{x}^2$\,. If $\kappa$ is small
enough, then $-\kappa\abs{x}^2$ dominates  $\kappa^2\abs{x}^2$\,.
For those $\kappa$\,, $(\lambda/(1-\lambda))\,\abs{x}^2$ is dominated
by $-\kappa\abs{x}^2$ if $\lambda$ is small relative to $\kappa$\,.
We conclude that, provided $\kappa$ is small enough, 
there exist $\lambda_0>0$\,,
$K_1$\,, and $K_2>0$\,, such that
\begin{equation}
  \label{eq:23a}
 H(x;\lambda,f_\kappa)\le K_1-K_2\abs{x}^2\,,
\end{equation}
 for all
$\lambda\le \lambda_0$\,.
Therefore,  given $\delta\in\R$\,,
$\sup_{\nu\in\Gamma_\delta}\int_{\R^l}\abs{x}^2\,\nu(dx)<\infty$\,,
where 
$\Gamma_\delta=  \big\{\nu:\,
 \int_{\R^l} H(x;\lambda,f_\kappa)\,\nu(dx)\ge \delta\big\}
$\,. 
In addition, by $H(x;\lambda, f_\kappa)$ being  continuous
in $x$ and Fatou's lemma,
$\int_{\R^l}H(x;\lambda, f_\kappa)\,\nu(dx)$ is an upper semicontinuous
function of $\nu$\,, so $\Gamma_\delta$ is a closed set.
Thus, by Prohorov's theorem,
  $\Gamma_\delta$   is compact.
If $f\in\mathcal{A}_\kappa$\,, then, in view of Fatou's lemma,
\eqref{eq:80}, \eqref{eq:59},
 and \eqref{eq:23a},
  the function $\int_{\R^l}H(x;\lambda, f)\,\nu(dx)$
is upper semicontinuous in $\nu$\,.
Since $ f_\kappa\in\mathcal{A}_\kappa$\,,
we obtain that $\inf_{f\in\mathcal{A}_\kappa}\int_{\R^l}H(x;\lambda, f)\,\nu(dx)$
is $\sup$--compact. 

\end{proof}

\begin{lemma}
  \label{le:minmax}
If $\lambda<1$ and  $F(\lambda)<\infty$\,,  then the infimum 
in \eqref{eq:29} is attained
at  $\mathbb C^2$--function $f^\lambda$ that  satisfies
the Bellman equation \eqref{eq:86} and belongs to $\mathbb C^1_\ell$\,. 
In addition,
the function $F(\lambda)$ is  lower semicontinuous 
and  $F(0)=0$\,.
\end{lemma}
\begin{proof}
Let us assume that 
 $F(\lambda)>-\infty$\,. Applying the reasoning on
 pp.289--294 in Kaise and Sheu \cite{KaiShe06}, one can see that, for
 arbitrary $\epsilon>0$\,,
 there exists $\mathbb C^2$--function $ f_\epsilon$ 
such that, for all $x\in\R^l$\,,
$    H(x;{\lambda},  f_\epsilon)
=F(\lambda)+\epsilon$\,.
Considering that some details are omitted 
in Kaise and Sheu \cite{KaiShe06}, we give 
an outline of the proof, following the lead of Ichihara \cite{Ich11}.
As $F(\lambda)<\infty$\,, by \eqref{eq:29},
 there exists function $ f^{(1)}_\epsilon\in \mathbb C^2$ such that
$ H(x;\lambda,  f^{(1)}_\epsilon)<
F(\lambda)+\epsilon$ 
for all $x$\,. 
Given  open ball $S$\,, centred at the origin,
 by Theorem 6.14 on p.107 in Gilbarg and
Trudinger \cite{GilTru83}, there exists 
$\mathbb C^2$--solution $ f^{(2)}_\epsilon$ to the linear
elliptic boundary value problem
$ H(x;\lambda, f) -(1/2)\nabla f(x)^TT_\lambda(x)\nabla f(x)
=F(\lambda)+2\epsilon$ when $x\in S$ and
  $f(x)= f_\kappa(x)$
when $x\in\partial S$\,, with $\partial S$ standing for
the boundary of $S$\,. 
Therefore, 
$ H(x;\lambda, f_\epsilon^{(2)})>
F(\lambda)+\epsilon$ in $S$\,.
By Theorem 8.4 on p.302 of Chapter 4 in 
Ladyzhenskaya and Uraltseva \cite{LadUra68}, for any ball $S'$
contained in $S$ and centred at the origin, there exists 
$\mathbb C^{2}$--solution $f^{(3)}_{\epsilon,S'}$ to
the boundary value problem $ H(x;\lambda, f)=
F(\lambda)+\epsilon$ in $S'$ and $f(x)= f_\kappa(x)$ on $\partial S'$\,.
Since $f_{\epsilon,S'}^{(3)}$ solves the boundary value problem 
$(1/2)\text{tr}\,(\sigma(x)\sigma(x)^T\nabla^2f(x))=
-\breve H(x;\lambda,\nabla f^{(3)}_{\epsilon,S'}(x))+F(\lambda)+\epsilon$ when $x\in
S'$ and
$f(x)=f_\kappa(x)$ when $x\in \partial S'$\,, we have by Theorem 6.17 on
p.109 of Gilbarg and Trudinger \cite{GilTru83} that 
$f_{\epsilon,S'}^{(3)}(x)$ is thrice continuously differentiable.
Letting the radius of $S'$ (and that of
$S$) go to infinity, we have, by p.294 in Kaise and Sheu
\cite{KaiShe06}, see also Proposition 3.2 in Ichihara \cite{Ich11},
  that the $f^{(3)}_{\epsilon,S'}$ converge locally uniformly and 
 in $\mathbb
W^{1,2}_{\text{loc}}(\R^l)$ to $f_\epsilon$ which is a weak solution to
$H(x;\lambda,f)=F(\lambda)+\epsilon$\,. 
Furthermore, by Lemma 2.4 in Kaise and Sheu \cite{KaiShe06}, the 
$\mathbb W^{1,\infty}(S'')$--norms of the $f^{(3)}_{\epsilon,S'}$ are uniformly bounded over
balls $S'$ for any fixed ball $S''$ contained in the
$S'$\,. Therefore, $f_\epsilon$ belongs to $\mathbb
W^{1,\infty}_{\text{loc}}(\R^l)$\,.
By Theorem 6.4 on p.284 in Ladyzhenskaya and Uraltseva
\cite{LadUra68},
$f_\epsilon$ is thrice continuously differentiable.

As in Theorem 4.2 in Kaise and Sheu
 \cite{KaiShe06}, by using the gradient bound in Lemma 2.4 there
(which proof does require
 $f_\epsilon$ to be thrice continuously differentiable), we
 have that the  $f_\epsilon$ converge along a
 subsequence  uniformly on compact
 sets  as $\epsilon\to0$ to
 a $\mathbb C^2$--solution of $    H(x;\lambda,  f)
=F(\lambda)$\,.
 That solution, which we denote by $f^\lambda$\,, 
delivers the infimum in
 \eqref{eq:29}
and satisfies the Bellman equation, with 
 $\nabla f^\lambda(x)$ obeying the linear growth condition,
 see Remark 2.5 in Kaise and Sheu \cite{KaiShe06}.
If we assume that $F(\lambda)=-\infty$\,, then the above reasoning
shows that there exists a solution to $    H(x;\lambda,  f)
=-K$\,, for all great enough  $K$ which leads to a
contradiction by the argument of the proof of Theorem 2.6 in Kaise and
Sheu \cite{KaiShe06}.


 We prove that $F$  is a lower semicontinuous  function.
   Let
$\lambda_i\to\lambda<1$\,, as $i\to\infty$\,, and 
 let the $F(\lambda_i)$ converge to a finite quantity. 
By the part just proved, there exist $\tilde f_i\in\mathbb C^2$ such that
$H(x;\lambda_i, \tilde f_i)=F(\lambda_i)$\,, for all $x$\,.
Furthermore, by a similar reasoning to the one used above
the sequence $\tilde f_i$ is
relatively compact in $\mathbb L^{\infty}_{\text{loc}}(\R^l)\cap \mathbb
W^{1,2}_{\text{loc}}(\R^l)$\, with limit points being in 
$\mathbb W^{1,\infty}_{\text{loc}}(\R^l)$ as well. Subsequential limit
 $\tilde f$ 
is a $\mathbb C^2$-function such that
$ H(x;\lambda,\tilde  f)=\lim_{i\to\infty}F(\lambda_i)$\,.
By \eqref{eq:29}, 
$F(\lambda)$
is the smallest $\Lambda$ such that there exists $\mathbb
C^2$--function $f$
 that satisfies
 the equation 
$ H(x;\lambda, f)=\Lambda  
$\,, for all $x\in\R^l$\,.
Hence,
$\lim_{i\to\infty}F(\lambda_i)\ge F(\lambda)$\,. 
The function $F(\lambda)$ is lower semicontinuous at $\lambda=1$ by definition.

We prove that $F(0)=0$\,. 
Taking $f(x)=0$ in \eqref{eq:29} yields $F(0)\le0$\,. 
Suppose that $F(0)<0$ and let $f\in\mathbb C^2\cap \mathbb C^1_\ell$
be such that, for all $x\in\R^l$\,,
\begin{equation}
  \label{eq:7}
  \nabla f(x)^T\,\theta(x)
+\frac{1}{2}\,\abs{\sigma(x)^T\nabla f(x)}^2
+\frac{1}{2}\, \text{tr}\bl({\sigma(x)}{\sigma(x)}^T\nabla^2 f(x)<0\,.
\end{equation}
By
\eqref{eq:45}, there exists density $m\in\hat{\mathbb P}$ such that
\begin{equation}
  \label{eq:83}
    \int_{\R^l}\bl(\nabla h(x)^T\,\theta(x)
+\frac{1}{2}\, \text{tr}\,\bl({\sigma(x)}{\sigma(x)}^T\nabla^2
h(x)\br)\br)
\,m(x)\,dx=0\,,
\end{equation}
for all $h\in\mathbb C_0^2$\,, see, e.g., Corollary 1.4.2 in
 Bogachev, Krylov, and  R\"eckner \cite{BogKryRoc}. 
By \eqref{eq:7}, $\int_{\R^l}\bl(\nabla f(x)^T\,\theta(x)
+(1/2) \text{tr}\,\bl({\sigma(x)}{\sigma(x)}^T\nabla^2
f(x)\br)\br)
\,m(x)\,dx$ is well defined, being possibly equal to $-\infty$
 and, by monotone convergence,
 \begin{multline*}
   \int_{\R^l}\bl(\nabla f(x)^T\,\theta(x)
+\frac{1}{2}\, \text{tr}\,\bl({\sigma(x)}{\sigma(x)}^T\nabla^2
f(x)\br)\br)
\,m(x)\,dx\\
=\lim_{R\to\infty}
\int_{x\in\R^l:\,\abs{x}\le R}\bl(\nabla f(x)^T\,\theta(x)
+\frac{1}{2}\, \text{tr}\,\bl({\sigma(x)}{\sigma(x)}^T\nabla^2
f(x)\br)\br)
\,m(x)\,dx\,.
 \end{multline*}
By integration by parts,
\begin{multline*}
  \int_{x\in\R^l:\,\abs{x}\le R}\bl(\nabla f(x)^T\,\theta(x)
+\frac{1}{2}\, \text{tr}\,\bl({\sigma(x)}{\sigma(x)}^T\nabla^2
f(x)\br)\br)
\,m(x)\,dx\\=
\int_{x\in\R^l:\,\abs{x}\le R}\bl(\nabla f(x)^T\,\theta(x)
-\frac{1}{2}\,\nabla f(x)^T\,
\frac{ \text{div}\,\bl({\sigma(x)}{\sigma(x)}^Tm(x)\br)}{m(x)}
\br)
\,m(x)\,dx\\+\frac{1}{2}
\,\int_{x\in\R^l:\,\abs{x}=R}\nabla f(x)^T
{\sigma(x)}{\sigma(x)}^Td(x)m(x)\,d\tau,
\end{multline*}
with $d(x)$ denoting the unit outward normal to the sphere
$\{x\in\R^l:\,\abs{x}=R\}$ at point $x$ and with the latter integral being a
surface integral.
As $\int_{\R^l}\abs{\nabla f(x)}m(x)\,dx<\infty$\,, 
\begin{equation*}
  \liminf_{R\to\infty}\int_{x\in\R^l:\,\abs{x}=R}\abs{\nabla f(x)^T
{\sigma(x)}{\sigma(x)}^Td(x)}m(x)\,d\tau=0\,,
\end{equation*}
so letting $R\to\infty$ appropriately yields the identity
\begin{multline}
  \label{eq:10}
      \int_{\R^l}\bl(\nabla f(x)^T\,\theta(x)
+\frac{1}{2}\, \text{tr}\,\bl({\sigma(x)}{\sigma(x)}^T\nabla^2
f(x)\br)\br)
\,m(x)\,dx\\=
\int_{\R^l}\bl(\nabla f(x)^T\,\theta(x)
-\frac{1}{2}\,\nabla f(x)^T\,
\frac{ \text{div}\,\bl({\sigma(x)}{\sigma(x)}^Tm(x)\br)}{m(x)}
\br)
\,m(x)\,dx\,,
\end{multline}
implying that the lefthand side is finite.
A similar integration by parts in \eqref{eq:83} yields
\begin{equation*}
\int_{\R^l}\bl(\nabla h(x)^T\,\theta(x)
-\frac{1}{2}\,\nabla h(x)^T\,
\frac{ \text{div}\,\bl({\sigma(x)}{\sigma(x)}^Tm(x)\br)}{m(x)}
\br)
\,m(x)\,dx=0\,.
\end{equation*}
Since $m\in\hat{\mathbb P}$\,,
this identity extends to $h\in \mathbb C^2\cap \mathbb C^1_\ell$\,, so
  the righthand side of  \eqref{eq:10} equals zero, which
contradicts \eqref{eq:7}. Thus, $F(0)=0$\,.

\end{proof}
\begin{remark} 
  As a byproduct of the proof, for $\lambda<1$\,,
  \begin{equation*}
    \inf_{f\in\mathbb C^2}
\sup_{x\in\R^l}H(x;\lambda,f)=
\inf_{f\in\mathbb C^2\cap \mathbb C^1_\ell}
\sup_{x\in\R^l}H(x;\lambda,f)\,.
  \end{equation*}
\end{remark}

\begin{lemma}
  \label{le:approx}
If $\lambda<1$ and
 $\mathcal{U}_\lambda\not=\emptyset$\,, then, for $\nu\in\mathcal{P}$\,,
\begin{equation}
  \label{eq:94}
    \inf_{f\in\mathcal{U}_\lambda}
\int_{\R^l} H(x;\lambda,f)\,\nu(dx)=
\inf_{f\in\mathbb C^2_0}
\int_{\R^l} H(x;\lambda,f)\,\nu(dx)\,.
\end{equation}
\end{lemma}
\begin{proof}
Let  $\eta$ be a cut--off function, i.e., a
$[0,1]$--valued smooth nonincreasing function on $\R_+$ such that
$\eta(y)=1$ when $y\in[0,1]$ and $\eta(y)=0$ when $y\ge 2$\,. Let us
assume, in addition, that the derivative $\eta'$ does
 not exceed $2$ in absolute value  and let $R>0$\,.
 Let $\eta_R(x)=\eta(\abs{x}/R)$\,.
Given  $\psi\in\mathbb C_0^2$ and
$\varphi \in\mathcal{U}_\lambda$\,, 
by \eqref{eq:80} and \eqref{eq:59},
\begin{multline}
  \label{eq:100}
      H(x;\lambda,\eta_R\psi+(1-\eta_R)\varphi)=
\frac{1}{2}\,
\nabla \psi(x)^TT_\lambda(x)\nabla \psi(x)\,\eta_R(x)^2
+S_\lambda(x)\nabla
\psi(x)\,\eta_R(x)\\
+\frac{1}{2}\, \text{tr}\,\bl({\sigma(x)}{\sigma(x)}^T\nabla^2
\psi(x)\br)
\eta_R(x)
+
\frac{1}{2}\,
\nabla \varphi(x)^TT_\lambda(x)\nabla \varphi(x)\,(1-\eta_R(x))^2
+S_\lambda(x)\nabla
\varphi(x)\,(1-\eta_R(x))\\
+\frac{1}{2}\, \text{tr}\,\bl({\sigma(x)}{\sigma(x)}^T
\nabla^2\varphi(x)\br)(1-\eta_R(x))
+\epsilon_R(x)
+R_\lambda(x)\,,
\end{multline}
where
\begin{multline}
  \label{eq:103}
  \epsilon_R(x)=
\frac{1}{2}\,
\nabla \eta_R(x)^TT_\lambda(x)\nabla \eta_R(x)\,(\psi(x)-\varphi(x))^2
+
\nabla \psi(x)^TT_\lambda(x)\nabla \eta_R(x)\,
(\psi(x)-\varphi(x))\eta_R(x)\\
+
\nabla \psi(x)^TT_\lambda(x)\nabla \varphi(x)\,
(1-\eta_R(x))\eta_R(x)
+
\nabla \varphi(x)^TT_\lambda(x)\nabla \eta_R(x)\,
(\psi(x)-\varphi(x))(1-\eta_R(x))\\
+S_\lambda(x)(\psi(x)-\varphi(x))\nabla\eta_R(x) 
+\frac{1}{2}\, \text{tr}\,\bl({\sigma(x)}{\sigma(x)}^T
\bl((\psi(x)-\varphi(x))\nabla^2\eta_R(x)\\
+(\nabla\psi(x)-\nabla\varphi(x))\nabla\eta_R(x)^T\br)\br)\,.
\end{multline}
Replacing on the righthand side of \eqref{eq:100}
 $\eta_R(x)^2$ and $(1-\eta_R(x))^2$ with $\eta_R(x)$ and
$1-\eta_R(x)$\,, respectively, obtains that
\begin{multline}
  \label{eq:105}
H(x;\lambda,\eta_R\psi+(1-\eta_R)\varphi)
\le\eta_R(x)H(x;\lambda,\psi)+(1-\eta_R(x))H(x;\lambda,\varphi)
+\epsilon_R(x)\,.
\end{multline}
Therefore, 
\begin{multline*}
  \int_{\R^l}      H(x;\lambda,\eta_R\psi+(1-\eta_R)\varphi)\,\nu(dx) \le
\int_{\R^l}\eta_R(x) H(x;\lambda,\psi)\,\nu(dx)
+\sup_{x\in\R^l}(H(x;\lambda,\varphi)\vee0)\nu(\R^l\setminus B_R)
\\+\int_{\R^l}\epsilon_R(x)\,\nu(dx)\,,
\end{multline*}
where $a\vee b=\max(a,b)$\,.
By dominated convergence, the first integral on the righthand side
converges to $\int_{\R^l}H(x;\lambda,\psi)\,\nu(dx)$\,, as
$R\to\infty$\,. 
Since $\abs{\nabla\eta_R(x)}\le4\chi_{\{\abs{x}\ge
  R\}}(x)/\abs{x}$\,,
 $\abs{\nabla\varphi(x)}$ is of, at most, linear growth,
by $\varphi$ being a member of 
$\mathbb{C}^1_\ell$\,, so that
$\varphi(x)$ grows, at most, quadratically,
and since
$\int_{\R^l}\abs{x}^2\,\nu(dx)<\infty$\,,
by \eqref{eq:103}, one has that
\begin{equation}
  \label{eq:87}
    \lim_{R\to\infty}\int_{\R^l}\epsilon_R(x)\,\nu(dx)=0\,.
\end{equation}
Since $\psi\eta_R+\varphi(1-\eta_R)\in
\mathcal{U}_\lambda$\,, agreeing with $\varphi$ if $\abs{x}>2R$\,, 
  \begin{equation*}
     \inf_{f\in\mathcal{U}_\lambda}
\int_{\R^l} H(x;\lambda,f)\,\nu(dx)
\le\inf_{f\in\mathbb C^2_0}
\int_{\R^l} H(x;\lambda,f)\,\nu(dx)\,.
\end{equation*}
Conversely, let $\varphi \in\mathcal{U}_\lambda$ and
$\psi_R(x)=\eta_R(x)\varphi(x)$\,. One can see that $\psi_R$ is a $\mathbb
C^2_0$--function. 
By \eqref{eq:12}, in analogy with \eqref{eq:105} and \eqref{eq:87},
\begin{equation*}
\int_{\R^l} H(x;\lambda,\psi_R)\,\nu(dx)
\le
  \int_{\R^l}\bl(\eta_R(x)H(x;\lambda,\varphi)+
(1-\eta_R(x))H(x;\lambda,\mathbf 0)\br)\,\nu(dx)+
\hat\epsilon_R\,,
\end{equation*}
where 
$  \lim_{R\to\infty}\hat\epsilon_R=0\,,$
with $\mathbf 0$ representing the function that is equal to zero
identically.
By Fatou's lemma, $H(x;\lambda,\varphi)$ being bounded from above,
\begin{equation}
  \label{eq:108}
    \limsup_{R\to\infty}\int_{\R^l}\eta_R(x)H(x;\lambda,\varphi)\,\nu(dx)\le
\int_{\R^l}H(x;\lambda,\varphi)\,\nu(dx)\,.
\end{equation}
By dominated convergence,
\begin{equation*}
\lim_{R\to\infty}
  \int_{\R^l}
(1-\eta_R(x))H(x;\lambda,\mathbf 0)\,\nu(dx)=0\,.
\end{equation*} 
Hence,
\begin{equation*}
\inf_{f\in\mathbb C^2_0}
\int_{\R^l} H(x;\lambda,f)\,\nu(dx)\le
  \inf_{f\in\mathcal{U}_\lambda}
\int_{\R^l} H(x;\lambda,f)\,\nu(dx)\,,
\end{equation*}
which concludes the proof of \eqref{eq:94}.

\end{proof}
\begin{remark}
\label{re:inf}  Similarly,  it can be shown that, if $\lambda<1$\,, then
\begin{equation*}
      \inf_{f\in\mathbb C_b^2}
\int_{\R^l} H(x;\lambda,f)\,\nu(dx)=
\inf_{f\in\mathbb C^2_0}
\int_{\R^l} H(x;\lambda,f)\,\nu(dx)\,.
\end{equation*}
(The analogue of \eqref{eq:108} holds with equality by bounded convergence.)
\end{remark}
The following lemma appears in Puhalskii and Stutzer \cite{PuhStu16}.
 \begin{lemma}
   \label{le:density_differ}
If, given $\lambda<1$\,, probability measure $\nu$ on $\R^l$ is such that
the integrals $\int_{\R^l} H(x;\lambda,f)\,\nu(dx)$ are bounded below
uniformly over 
$f\in\mathbb C_0^2$\,,
then $\nu$ admits density  which belongs to
$\hat{\mathbb{P}}$\,.
 \end{lemma}
 \begin{proof}
   The reasoning follows that of Puhalskii \cite{Puh16}, cf. Lemma
   6.1, Lemma 6.4, and Theorem 6.1 there.
If there exists $\kappa\in\R$ such that
$\int_{\R^l} H(x;\lambda,f)\,\nu(dx)\ge\kappa$\, 
for all $f\in\mathbb C_0^2$\,,
then
by \eqref{eq:59}, for arbitrary $\delta>0$\,,
\begin{equation*}
  \delta\int_{\R^l}\frac{1}{2}\, \text{tr}\,
\bl({\sigma(x)}{\sigma(x)}^T\nabla^2 f(x)\br)\,
\nu(dx)\ge \kappa-\int_{\R^l}\breve H(x;\lambda,\delta\nabla f(x))
\,\nu(dx)\,.
\end{equation*}
On letting 
\begin{equation*}
  \delta=\kappa^{1/2}
\Bl(\int_{\R^l}\nabla f(x)^T
T_\lambda(x)\nabla f(x)\,\nu(dx)\Br)^{-1/2}\,,
\end{equation*}
we obtain with the aid of \eqref{eq:80} and the Cauchy--Schwarz
inequality  that there exists constant $K_1>0$
such that, for all $f\in\mathbb C_0^2$\,,
\begin{equation*}
    \int_{\R^l} \text{tr}\,
\bl({\sigma(x)}{\sigma(x)}^T\nabla^2 f(x)\br)\,
\,\nu(dx)\le K_1\Bl(\int_{\R^l}\abs{\nabla f(x)}^2\,\nu(dx)\Br)^{1/2}\,.
\end{equation*}
It follows that the lefthand side extends to a linear functional on
$\mathbb L^{1,2}_0(\R^l,\R^l,\nu(dx))$\,, hence, by the Riesz
representation theorem,
there exists $\nabla h\in\mathbb L^{1,2}_0(\R^l,\R^l,\nu(dx))$ such that
\begin{equation}  \label{eq:46}
  \int_{\R^l} \text{tr}\,
\bl({\sigma(x)}{\sigma(x)}^T\nabla^2 f(x)\br)\,
\,\nu(dx)=\int_{\R^l}\nabla h(x)^T\nabla f(x)\,\nu(dx)
\end{equation}
and 
$  \int_{\R^l}\abs{\nabla h(x)}^2\nu(dx)\le K_1\,.
$
Theorem 2.1 in Bogachev, Krylov, and R\"ockner \cite{BogKryRoc01}
 implies that the
measure $\nu(dx)$ has  density $m(x)$ with respect to  Lebesgue measure
which belongs to  $L_{\text{loc}}^\xi(\R^l)$ 
for all  $\xi\in(1,l/(l-1))$\,. 
It follows  that, for 
arbitrary open ball $S$ in $\R^l$\,,
  there exists  $K_2>0$ such that
for all  
$ f\in \mathbb{C}_0^2$ with support in  $S$\,,
\begin{equation*}
  \abs{\int_{S} \text{tr}\,\bl(\sigma(x)\sigma(x)^T\nabla^2f(x)
\br)\,m(x)\,dx}\le K_2\bl(\int_S \abs{\nabla f(x)}^{2\xi/(\xi-1)}
\,dx\br)^{(\xi-1)/(2\xi)}\,.
\end{equation*}

\noindent 
By Theorem 6.1 in Agmon \cite{Agm59}, 
the  density $m$  belongs to
$\mathbb{W}_{\text{loc}}^{1,\zeta}(S)$ for all $\zeta\in(1,2l/(2l-1))$.
 Furthermore, $\nabla h(x)=-\nabla m(x)/m(x)$ so that
$\sqrt{m}\in \mathbb W^{1,2}(\R^l)$\,.
 \end{proof}
\begin{remark}
  Essentially, \eqref{eq:46} signifies that one can integrate by parts on
  the lefthand side, so $m(x)$ needs to be differentiable.
\end{remark}

\begin{lemma}
  \label{le:conc}
  \begin{enumerate}
  \item 
The function $\breve H(x,\lambda,p)$ is strictly convex in
$(\lambda,p)$ on
$(-\infty,1)\times\R^l$  and is  convex on $\R\times \R^l$\,.
The function $H(x;\lambda,f)$ is convex in $(\lambda,f)$ on 
$\R\times \mathbb C^2$\,. For $m\in\mathbb P$\,,  the function
$G(\lambda,f,m)$ is
  convex  in $(\lambda,f)$ on $\R\times 
\mathbb C_b^2$\,.
 \item
 Let $m\in
\hat{\mathbb P}$\,.
Then the function $\breve G(\lambda,\nabla f,m)$ is
 convex and lower semicontinuous  in $(\lambda,\nabla f)$ on $\R\times {\mathbb
  L}^{1,2}_0(\R^l,\R^l,m(x)\,dx)$ and is strictly convex on 
$(-\infty,1)\times {\mathbb
  L}^{1,2}_0(\R^l,\R^l,m(x)\,dx)$\,.
If  $\lambda<1$\,, then the infimum in
   \eqref{eq:13} is attained at unique $\nabla f$\,. If $\lambda=1$
   and the infimum in \eqref{eq:13}
 is finite, then it is  attained at unique
   $\nabla f$ too. 
The function $
\breve F(\lambda,m)$ is
  convex and lower semicontinuous with respect to $\lambda$\,,
it is strictly
convex  on $(-\infty,1)$\,,
 and tends to $\infty$
superlinearly, as
$\lambda\to-\infty$\,.
If $\lambda<1$\,, then
\begin{equation}
  \label{eq:15}
   \breve F(\lambda,m)=\inf_{f\in\mathbb C^2\cap \mathbb C^1_\ell}
\breve G(\lambda,\nabla f,m)=
\inf_{f\in\mathbb
  C_0^2}G(\lambda, f,m)\,.
\end{equation}
If $\lambda<1$ and $\mathcal{U}_\lambda\not=\emptyset$\,, then
\begin{equation}
  \label{eq:52}
\breve F(\lambda,m)  =
\inf_{f\in\mathcal{U}_\lambda}\breve G(\lambda,\nabla f,m)=
\inf_{f\in\mathcal{U}_\lambda} G(\lambda, f,m)\,.
\end{equation}

If $f\in \mathbb L^{1,2}_0(\R^l,\R^l,m(x)\,dx)$\,, then
$  \breve G(\lambda,\nabla f,m)$  is differentiable in 
$\lambda\in(-\infty,1)$ and
  \begin{multline}
    \label{eq:10a}
       \frac{d}{d\lambda}\,   \breve   G(\lambda,
  \nabla f,m)=
 \int_{\R^l}\bl( M( u^{\lambda,\nabla f}(x),x)
+\lambda
\abs{N(u^{\lambda,\nabla f}(x),x)}^2\\+ \nabla f(x)^T\sigma(x)
N(u^{\lambda,\nabla f}(x),x)\br)
 m(x)\,dx\,,
  \end{multline}
where $ u^{\lambda,\nabla f}(x)$ is defined by \eqref{eq:39} with $\nabla f(x)$
as $p$\,.
Furthermore, $\breve F(\lambda,m)$  is differentiable with respect to
$\lambda$ and 
\begin{equation}
  \label{eq:222}
    \frac{d}{d\lambda}\,   \breve   F(\lambda,m)=
\frac{d}{d\lambda}\,   \breve   G(\lambda,
  \nabla f^{\lambda,m},m)\,,
\end{equation}
with $\nabla f^{\lambda,m}$ attaining the infimum 
on the righthand side of \eqref{eq:13}.
In addition, if $\breve F(1,m)<\infty$\,, then the lefthand
derivatives at 1 equal each other as well:
\begin{equation}
  \label{eq:163}
   \frac{d}{d\lambda}\,\breve F(\lambda, m)\big|_{1-}=
\frac{d}{d\lambda}\,
\breve G(\lambda,\nabla  f^{1,m}, m)\big|_{1-}\,.
\end{equation}
\item
The function $F(\lambda)$ is convex, is continuous for
$\lambda<\overline\lambda$\,,
and   $F(\lambda)\to\infty$
superlinearly, as
$\lambda\to-\infty$\,.
The functions $J_q$\,, $J_q^{\text{o}}$\,, and $J_q^{\text{s}}$ are
continuous.
  \end{enumerate}
\end{lemma}
\begin{proof}
If $\lambda<1$\,, then, by \eqref{eq:40} and \eqref{eq:64},
the Hessian matrix of $\breve H(x;\lambda,p)$
with respect to $(\lambda,p)$ is given by
\begin{align*}
\breve H_{pp}(x;\lambda,p)&=\frac{1}{1-\lambda}\,
\sigma(x) b(x)^Tc(x)^{-1}b(x)\sigma(x)^T
+\sigma(x)Q_1(x)\sigma(x)^T\,,\\
  \breve H_{\lambda\lambda}(x;\lambda,p)&=\frac{1}{(1-\lambda)^3}\,
\norm{a(x)-r(x)\mathbf1+b(x)\sigma(x)^Tp-b(x)\beta(x)}^2_{c(x)^{-1}}
+\beta(x)^TQ_1(x)\beta(x)\,,\\
\breve H_{\lambda p}(x;\lambda,p)&=-\frac{1}{(1-\lambda)^2}\,
\bl(a(x)-r(x)\mathbf1+b(x)\sigma(x)^Tp-b(x)\beta(x)\br)^Tc(x)^{-1}b(x)\sigma(x)^T\\&+\beta(x)^TQ_1(x)\sigma(x)^T\,.
\end{align*}
We show that it is positive definite. More specifically, we prove that 
for all $\tau\in\R$ and
$y\in\R^l$ such that $\tau^2+\abs{y}^2\not=0$\,,
\begin{equation*}
  \tau^2\breve H_{\lambda\lambda}(x;\lambda,p)
+y^T\breve H_{pp}(x;\lambda,p) y+
2\tau\breve H_{\lambda p}(x;\lambda,p)y>0\,.
\end{equation*}
Since $\breve H_{pp}(x;\lambda,p)$ is a   positive definite matrix by
condition (N),
the latter inequality  holds when $\tau=0$\,. Assuming $\tau\not=0$\,, we
need to show that
\begin{equation}
  \label{eq:104}
  \breve H_{\lambda\lambda}(x;\lambda,p)
+y^T\breve H_{pp}(x;\lambda,p) y+
2\breve H_{\lambda p}(x;\lambda,p)y>0\,.
\end{equation}
Let, for   $d_1=(v_1(x),w_1(x))$ 
and $d_2=(v_2(x),w_2(x))$\,,
where $v_1(x)\in\R^n\,,w_1(x)\in\R^k\,,v_2(x)\in\R^n\,,
w_2(x)\in\R^k$\,, and $x\in\R^l$\,,  the
inner product be defined by
$d_1\cdot d_2=v_1(x)^Tc(x)^{-1}v_2(x)+
w_1(x)^Tw_2(x)$\,.
By the Cauchy--Schwarz inequality, applied to
$d_1=\bl((1-\lambda)^{-3/2}(a(x)-r(x)\mathbf1+b(x)\sigma(x)^Tp-b(x)\beta(x)),
Q_1(x)\beta(x)\br)$ and
$d_2=((1-\lambda)^{-1/2}b(x)\sigma(x)^Ty,Q_1(x)\sigma(x)^Ty)$\,,
we have that $
  (\breve H_{\lambda p}(x;\lambda,p)y)^2
<y^T\breve H_{pp}(x;\lambda,p) y
\breve H_{\lambda\lambda}(x;\lambda,p)\,,
$ with the inequality being strict because, by part 2 of condition (N), 
 $Q_1(x)\beta(x)$ is not a scalar multiple of 
$Q_1(x)\sigma(x)^Ty$\,.
Thus, \eqref{eq:104} holds, so 
the function $\breve H(x;\lambda,p)$ is strictly convex
in $(\lambda,p)$ on 
$(-\infty,1)\times \R^l$\,, for all $x\in\R^l$\,.

Since by \eqref{eq:40} and \eqref{eq:64}, $\breve H(x;\lambda_n,p_n)\to 
\breve H(x;1,p)\le\infty$ as $\lambda_n\uparrow 1$
and $p_n\to p$\,, and $\breve H(x;\lambda,p)=\infty$ if 
$\lambda>1$\,, the function $\breve H(x;\lambda,p)$ is convex
in $(\lambda,p)$ on 
$\R\times \R^l$\,.
By  \eqref{eq:59},
the function $H(x;\lambda,f)$ 
is  convex in $(\lambda,f)$ on $\R\times \mathbb C^2$\,.
By \eqref{eq:62}, for any $m\in\mathbb P$\,,
 $G(\lambda,f,m)$ is convex in $(\lambda,f)$ on $\R\times \mathbb
C^2_b$\,.

Let  $m\in\hat{\mathbb P}$\,. 
By \eqref{eq:11} and the 
strict convexity of $\breve H$\,, $\breve G(\lambda,\nabla f,m)$ is strictly convex in
$(\lambda,\nabla f)\in
(-\infty,1)\times \mathbb L^{1,2}_0(\R^l,\R^l,m(x)\,dx)$\,.
Let us note that, by \eqref{eq:64}, for $\epsilon>0$\,,
\begin{multline}
  \label{eq:81}
\breve H(x;\lambda,p)\ge
-\frac{1}{2}\,
\norm{a(x)-r(x)\mathbf1-\lambda
  b(x)\beta(x)+b(x)\sigma(x)^Tp}^2_{c(x)^{-1}}
+\frac{1}{2}\,\lambda^2\abs{\beta(x)}^2\\+
\lambda(
r(x)-\alpha(x)+\frac{1}{2}\,\abs{\beta(x)}^2-\beta(x)^T\sigma(x)^Tp)
+
p^T\theta(x)+\frac{1}{2}\,\abs{{\sigma(x)}^Tp}^2\\\ge
-\frac{1}{2}\,\Bl((1+\epsilon)\norm{b(x)\sigma(x)^Tp}^2_{c(x)^{-1}}+
\bl(1+\frac{1}{\epsilon}\br)\norm{a(x)-r(x)\mathbf1-\lambda
  b(x)\beta(x)}^2_{c(x)^{-1}}\Br)
\\+\frac{1}{2}\,\lambda^2\abs{\beta(x)}^2+
\lambda(r(x)-\alpha(x)+\frac{1}{2}\,\abs{\beta(x)}^2)+
p^T(\theta(x)-\lambda\sigma(x)\beta(x))
+\frac{1}{2}\,\abs{{\sigma(x)}^Tp}^2
\\=\frac{1}{2}\,\norm{p}^2_{Q_{1,\epsilon}(x)}+
\frac{1}{2}\,\bl(1+\frac{1}{\epsilon}\br)\norm{a(x)-r(x)\mathbf1-\lambda
  b(x)\beta(x)}^2_{c(x)^{-1}}
\\+\frac{1}{2}\,\lambda^2\abs{\beta(x)}^2+
\lambda(r(x)-\alpha(x)+\frac{1}{2}\,\abs{\beta(x)}^2)
+p^T(\theta(x)-\lambda\sigma(x)\beta(x))\,,
\end{multline}
where $Q_{1,\epsilon}(x)=Q_1(x)-\epsilon
\sigma(x)b(x)^Tc(x)^{-1}b(x)\sigma(x)^T$\,.
Since $Q_1(x)$ is uniformly positive definite, so is
$Q_{1,\epsilon}(x)$\,, provided $\epsilon$ is small enough.
By \eqref{eq:81}, \eqref{eq:11},
and  by the facts that $\int_{\R^l}\abs{x}^2m(x)\,dx<\infty$ and
$\int_{\R^l}\abs{\nabla m(x)}^2/m(x)\,dx<\infty$\,, 
$\breve G(\lambda,\nabla f,m)$ tends
to infinity as the $\mathbb L^2(\R^l,\R^l,m(x)\,dx)$--norm 
of $\nabla f$ tends to infinity, locally uniformly over $\lambda$\,.
Since, in addition, $\breve G(\lambda,\nabla f,m)$ is strictly convex in
$(\lambda,\nabla f)$\,, 
  the infimum  on the righthand side of
 \eqref{eq:13} is attained at  unique $\nabla f$\,, if finite,
 see, e.g.,
Proposition 1.2 on p.35 in Ekeland and Temam \cite{EkeTem76}.
(If $\lambda<1$\,, then $\breve G(\lambda,\nabla f,m)<\infty$\,, for
all $\nabla f\in\mathbb L^{1,2}_0(\R^l,\R^l,m(x)\,dx)$\,, by \eqref{eq:80} and \eqref{eq:11}.)
Hence, the righthand side of \eqref{eq:13} is strictly convex
in $\lambda$ on $(-\infty,1)$\,.
(For, let $\inf_{\nabla f\in\mathbb L^{1,2}_0(\R^l,\R^l,m(x)\,dx)}
\breve G(\lambda_i,\nabla f,m)=\breve G(\lambda_i,\nabla 
f_i,m)$\,, for $i=1,2$\,.
Then $\inf_{\nabla f\in\mathbb L^{1,2}_0(\R^l,\R^l,m(x)\,dx)}
\breve G((\lambda_1+\lambda_2)/2,\nabla f,m)
\le \breve G((\lambda_1+\lambda_2)/2,(\nabla f_1+\nabla f_2)/2,m)
<(\breve G(\lambda_1,\nabla f_1,m)
+\breve G(\lambda_2,\nabla f_2,m))/2
=(\inf_{\nabla f\in\mathbb L^{1,2}_0(\R^l,\R^l,m(x)\,dx)}
\breve G(\lambda_1,\nabla f,m)+\inf_{\nabla f\in\mathbb L^{1,2}_0(\R^l,\R^l,m(x)\,dx)}
\breve G(\lambda_2,\nabla f,m))/2$\,.)

By \eqref{eq:81}, by
 $\breve H(x;\lambda,p)$ being a lower semicontinuous
function of $(\lambda,p)$ with values in $\R\cup\{+\infty\}$\,,
 by \eqref{eq:11} and Fatou's
 lemma,  
 $\breve G(\lambda,\nabla f,m)$ is lower semicontinuous in 
$(\lambda,\nabla f)$ on $\R\times \mathbb
L^{1,2}_0(\R^l,\R^l,m(x)\,dx)$\,.
By a similar argument to that in 
Proposition 1.7 on p.14 in Aubin \cite{Aub93} or Proposition 5
on p.12 in Aubin and Ekeland \cite{AubEke84},
 the function $\breve F(\lambda,m)$
 is  lower semicontinuous in $\lambda$\,.
More specifically, let $\lambda_i\to\lambda$ and let 
$K_1=\liminf_{i\to\infty}
\breve F(\lambda_i,m)$\,. Assuming that $K_1<\infty$\,,
by \eqref{eq:13},
for all $i$ great enough,
\begin{equation*}
  \breve F(\lambda_i,m)=
\inf_{\nabla f\in\mathbb L^{1,2}_0(\R^l,\R^l,m(x)\,dx):\,\breve G(\lambda_i,\nabla f,m)\le K_1+1}\breve G(\lambda_i,\nabla f,m)\,.
\end{equation*}
By \eqref{eq:11} and \eqref{eq:81}, there exists $K_2$ such that, for
all $i$\,,
if $\breve G(\lambda_i,\nabla f,m)\le K_1+1$\,, then $\int_{\R^l}\abs{\nabla
  f(x)}^2\,m(x)\,dx\le K_2$\,. 
The  set of the latter 
$\nabla{f}$  being weakly compact in $\mathbb
L^{1,2}_0(\R^l,\R^l,m(x)\,dx)$ and the function $\breve G(\lambda,\nabla
f,m)$ being convex and lower semicontinuous in $\nabla f$\,, 
there exist $\nabla f_i$ such that 
$\breve F(\lambda_i,m)=
\breve G(\lambda_i,\nabla f_i,m)$
\,.
Extracting a suitable  subsequence of $\nabla f_i$ that weakly
converges to some $\nabla \tilde f$ and invoking the lower
semicontinuity of $ \breve G(\lambda,\nabla f,m)$ in $(\lambda,\nabla f)$ yields
\begin{multline*}
  \liminf_{i\to\infty}
\breve F(\lambda_i,m)=
\liminf_{i\to\infty}\inf_{\nabla f\in\mathbb
  L^{1,2}_0(\R^l,\R^l,m(x)\,dx):\,
\breve G(\lambda_i,\nabla f,m)\le K_1+1}\breve G(\lambda_i,\nabla f,m)\\\ge
\liminf_{i\to\infty}\inf_{\nabla f\in\mathbb L^{1,2}_0(\R^l,\R^l,m(x)\,dx):\,\int_{\R^l}\abs{\nabla
  f(x)}^2\,m(x)\,dx\le K_2}\breve G(\lambda_i,\nabla f,m)\\=
\liminf_{i\to\infty}\breve G(\lambda_i,\nabla f_i,m)
\ge \breve G(\lambda,\nabla \tilde f,m)
\ge \breve F(\lambda,m)\,.
\end{multline*}
We have proved that 
the function $
\breve F(\lambda,m)$ is lower semicontinuous in $\lambda$\,. 
It follows that   the function $\sup_{m\in\hat{\mathbb P}}
\breve F(\lambda,m)$
 is  lower semicontinuous.

Let us show that the gradients of functions from $\mathbb C^2\cap
\mathbb C^1_\ell$ make up a dense subset of $\mathbb
L^{1,2}_0(\R^l,\R^l,\hat m(x)\,dx)$\,. 
  Let $f\in\mathbb C^1_\ell$ and
 let  $\eta(y)$ represent a cut--off function, i.e., a
$[0,1]$--valued smooth nonincreasing function on $\R_+$ such that
$\eta(y)=1$ when $y\in[0,1]$ and $\eta(y)=0$ when $y\ge 2$\,. Let $R>0$\,.
The function $f(x)\eta(\abs{x}/R)$ belongs to $\mathbb C^1_0$\,. In
addition,
\begin{multline*}
  \int_{\R^l}\abs{\nabla f(x)-\nabla
\bl(f(x)\eta\bl(\frac{\abs{x}}{R}\br)\br)}^2m(x)\,dx
\le 2\int_{\R^l}\abs{\nabla
  f(x)}^2\bl(1-\eta\bl(\frac{\abs{x}}{R}\br)\br)^2m(x)\,dx\\
+\frac{2}{R^2}\,\int_{\R^l} f(x)^2\eta'\bl(\frac{\abs{x}}{R}\br)^2m(x)\,dx\,,
\end{multline*}
where $\eta'$ stands for the derivative of $\eta$\,.
Since $\int_{\R^l}\abs{x}^2\,m(x)\,dx$ converges, the righthand side of
the latter inequality tends to $0$ as $R\to\infty$\,.
Hence, $\nabla f\in\mathbb L^{1,2}_0(\R^l,\R^l,\hat m(x)\,dx)$\,. 
On the other hand, the
gradients of $\mathbb C^1_0$--functions can be approximated with the
gradients of $\mathbb C^2\cap
\mathbb C^1_\ell$--functions in $\mathbb L^{1,2}_0(\R^l,\R^l,\hat
m(x)\,dx)$\,, which ends the proof. 

On recalling \eqref{eq:13}, we obtain the leftmost equality in \eqref{eq:15}.
Similarly, since $G(\lambda,f,m)=\breve G(\lambda,\nabla f,m)$
when $f\in\mathbb C_0^2$ 
and the gradients of $\mathbb C_0^2$--functions are dense in $\mathbb
L^{1,2}_0(\R^l,\R^l,m(x)\,dx)$\,, the rightmost side of \eqref{eq:15}
equals the leftmost side.
For \eqref{eq:52}, we recall Lemma \ref{le:approx} and note that, as 
 the proof of Lemma \ref{le:minmax} shows,
$G(\lambda,f,m)=\breve G(\lambda,\nabla f,m)$ when
$f\in\mathcal{U}_\lambda$ and $\lambda<1$\,.  

By   \eqref{eq:80} and \eqref{eq:73}, as $\lambda\to-\infty$\,,
\begin{equation*}
  \lim_{\lambda\to-\infty}\frac{1}{\lambda^2}\,\inf_{p\in\R^l}
\bl( \breve H(x;\lambda,p)-\frac{1}{2}\,p^T\sigma(x)\sigma(x)^T
\,\frac{\nabla m(x)}{m(x)}\br)=
\frac{1}{2}\,\norm{\beta(x)}^2_{Q_2(x)}\,.
\end{equation*}
The latter quantity  being positive
 by the second part of condition (N)
implies, by \eqref{eq:13}, that
$  \liminf_{\lambda\to-\infty}(1/\lambda^2)
\breve F(\lambda,m)>0$\,, so,
$
  \liminf_{\lambda\to-\infty}(1/\lambda^2)\inf_{f\in\mathbb C_0^2} 
G(\lambda,f,m)>0\,.
$ By  \eqref{eq:29}, \eqref{eq:136}, and \eqref{eq:62}, 
$F(\lambda)\ge \inf_{f\in \mathbb C_0^2}G(\lambda,f,m)$\,, so, 
$  \liminf_{\lambda\to-\infty}F(\lambda)/\lambda^2>0$\,.
Therefore, for all $q$ from a bounded set,
the supremum in \eqref{eq:30} can be taken over $\lambda$
from the same compact set, which implies that $J_q$ is continuous.
With $J_q^{\text{o}}$ and $J_q^{\text{s}}$\,, a similar reasoning applies.
Since $\sup_{x\in\R^l}H(x;\lambda,f)$
is a convex function of $(\lambda,f)$\,, by \eqref{eq:29}, 
$F(\lambda)$ is convex. Being finite, it is continuous for
$\lambda<\overline\lambda$\,.

We prove the differentiability properties.
The  assertion in \eqref{eq:10a} follows
 by
  Theorem 4.13 on p.273
 in Bonnans and Shapiro
 \cite{BonSha00}   and dominated convergence, once we recall
\eqref{eq:80}
and \eqref{eq:11}. Equation \eqref{eq:222} is obtained similarly,
with $\breve G(\cdot,\cdot, m)$ as 
$f(\cdot,\cdot)$\,, with $\lambda$ as $u$\,, and with
$\nabla f$ as $x$\,, respectively, in the hypotheses of   Theorem 4.13 on p.273
 in Bonnans and Shapiro
 \cite{BonSha00}. In
some more detail, $\breve G(\lambda,\nabla f,m)$ and $d\breve
G(\lambda,\nabla f,m)/d\lambda$ are continuous functions of
$(\lambda,\nabla f)$ by \eqref{eq:40},  \eqref{eq:39}, and
\eqref{eq:11}.
The $\inf$--compactness condition on p.272 in Bonnans and Shapiro
\cite{BonSha00} holds because, as it has been shown in the proof of
the lower semicontinuity of $\breve F(\lambda,m)$\,, the infimum on
the righthand side of \eqref{eq:13} can be taken over the same weakly
compact subset of $\mathbb L^{1,2}_0(\R^l,\R^l,m(x)\,dx)$ for all
$\lambda$ from a compact subset of $(-\infty,1)$\,.
For \eqref{eq:163},
one can also  apply
 the reasoning of the proof of Theorem 4.13 on p.273 in Bonnans
and Shapiro \cite{BonSha00}. Although the hypotheses of the theorem
are not satisfied, the proof on pp.274,275  goes through, the key being that 
the function $\breve G(\lambda,\nabla f, m)$ tends to infinity uniformly
over $\lambda$ close enough to $1$ on the left, as the
$\mathbb L^2(\R^l,\R^l, m(x)\,dx)$--norm of $\nabla f$ tends to
infinity.

\end{proof}
\begin{remark}
  If condition (N) is not assumed, then strict convexity in the
  statement has to be replaced with convexity.
\end{remark}
\begin{remark}
  If $\beta(x)=0$\,, then
   $F(\lambda)/\lambda^2$ tends to
  zero as $\lambda\to-\infty$\,.
Furthermore,
\begin{equation*}
\liminf_{\lambda\to-\infty}\frac{1}{\abs{\lambda}}\,\inf_{f\in\mathbb C_0^2} 
G(\lambda,f,m)\ge -\int_{\R^l}r(x)m(x)\,dx\,,
\end{equation*}
so that
\begin{equation*}
    \liminf_{\lambda\to-\infty}\frac{F(\lambda)}{\abs{\lambda}}\ge
-\inf_{x\in\R^l}r(x)\,.
\end{equation*}
Consequently, 
if $\inf_{x\in\R^l}r(x)<q$\,, then $\lambda q-F(\lambda)$ tends to
$-\infty$ as $\lambda\to-\infty$, so $\sup_{\lambda\in\R}(\lambda
q-F(\lambda))$ is attained. That might not be the case if 
$\inf_{x\in\R^l}r(x)\ge q$\,. For instance, if the functions 
$a(x)$\,, $r(x)$\,, $b(x)$\,, and $\sigma(x)$ are constant and $q$ is
small enough, then the derivative of
$\lambda q-F(\lambda)$  is positive for all $\lambda<0$\,.
In particular, $J_q$\,, $J_q^{\text{s}}$\,, or $J_q^{\text{o}}$ 
might not be continuous at
$\inf_{x\in\R^l}r(x)$\,, $J_q^{\text{s}}$
being rightcontinuous and $J_q^{\text{o}}$
being leftcontinuous regardless.\end{remark}

\begin{lemma}
  \label{le:saddle_3}
  \begin{enumerate}
  \item 
The function
$\lambda q-\breve F(\lambda,m)$ has saddle point
 $(\hat\lambda,\hat m)$ in
$(-\infty,\overline\lambda]\times\hat{\mathbb P}$\,, with $\hat
\lambda$ being specified uniquely.
In addition,
$\hat
\lambda q-F(\hat\lambda)=\sup_{\lambda\in\R}(\lambda q-F(\lambda))$\,.
If  $\lambda\le\overline\lambda$\,, then
$    F(\lambda)=\sup_{m\in\hat{\mathbb   P}}
\breve F(\lambda,m)$\,.
 \item
Suppose that $\hat\lambda<1$\,.
Then the function  $\lambda q-\breve G(\lambda,\nabla f,m)$\,,
being  concave in $(\lambda,f)$ and
convex in $m$\,,  has  saddle point
$(\hat\lambda,\hat f,\hat m)$ in $(-\infty,\overline\lambda]\times
(\mathbb C^2\cap\mathbb C^1_\ell)\times \hat{\mathbb{P}}$\,, with
   $\nabla \hat
f$ and $\hat m$ being specified uniquely.
Equations 
  \eqref{eq:103'}  and  \eqref{eq:104'}  hold.
 \item Suppose that $\hat\lambda=1$\,. Then there exists
unique $\nabla\hat
  f\in\mathbb L^{1,2}_0(\R^l,\R^l,\hat m(x)\,dx)$ such that 
$\breve F(1,\hat m)=\breve G(1,\nabla\hat f,\hat m)$\,, 
$a(x)-r(x)\mathbf 1-b(x)\beta(x)+b(x)\sigma(x)^T\nabla\hat f(x)=0$ $\hat
m(x)\,dx$--a.e. and
\begin{equation*}
  \int_{\R^l}\bl(
\nabla h(x)^T\bl(
-\sigma(x)\beta(x)+\theta(x)+\sigma(x)\sigma(x)^T\nabla\hat f(x)\br)
+\frac{1}{2}\,
\text{tr}\,\bl(\sigma(x)\sigma(x)^T\nabla^2 h(x)\br)\br)\hat m(x)\,dx=0\,,
\end{equation*}
for all $h\in\mathbb C_0^2$ such that $b(x)\sigma(x)^T\nabla h(x)=0$
$\hat m(x)\,dx$--a.e.
\end{enumerate}\end{lemma}
\begin{proof}

Let $\mathcal{U}=\{(\lambda,f):\, f\in\mathcal{U}_\lambda\}$\,. It is
a convex set by $H(x;\lambda,f)$ being convex in $(\lambda,f)$\,.
Let $\tilde q\in\R$\,.
When  $(\lambda,f)\in\mathcal{U}$ and $\nu\in\mathcal{P}$\,, 
 the function $\lambda\tilde q-
\int_{\R^l} H(x;\lambda,f)\,\nu(dx)$ is well defined,
being possibly equal to $+\infty$\,, is concave
 in  $(\lambda,f)$\,,
 is convex and lower semicontinuous in $\nu$\,, and is
$\inf$--compact in $\nu$\,, provided $\lambda<0$\,, the latter
property holding by Lemma \ref{le:sup-comp}.
Theorem
7 on p.319 in Aubin and Ekeland \cite{AubEke84}, whose proof applies
to the case of the function $f(x,y)$ in the statement of the theorem
taking  values in $\R\cup \{+\infty\}$
  yields the identity
\begin{equation}
  \label{eq:74}
\inf_{\nu\in\mathcal{P}}\sup_{(\lambda,f)\in\mathcal{U}}\bl(\lambda
\tilde q-\int_{\R^l} H(x;\lambda,f)\,\nu(dx)\br)=
\sup_{(\lambda,f)\in\mathcal{U}}\inf_{\nu\in\mathcal{P}}\bl(\lambda
\tilde q-\int_{\R^l} H(x;\lambda,f)\,\nu(dx)\br)\,,
\end{equation} 
with the infimum on the lefthand side being attained, at  $\hat
\nu$\,.  If $\nu$ has no density with
respect to Lebesgue measure that belongs to $\hat{ \mathbb P}$\,,
then, by Lemma \ref{le:density_differ}, the supremum on the lefthand side
equals $+\infty$\,. Hence, the infimum on the lefthand side may be
taken over $\nu$ with densities from $\hat{\mathbb P}$\,, in
particular, it may be assumed that
 $\hat\nu(dx)=\hat m(x)\,dx$\,, where
$\hat m\in\hat{\mathbb P}$\,.
 We thus have that
\begin{equation}
  \label{eq:2}
  \inf_{m\in\hat{\mathbb P}}\sup_{\lambda\in\R}(\lambda \tilde q-\inf_{f\in\mathcal{U}_\lambda}
G(\lambda,f, m))=\sup_{\lambda\in\R}(\lambda
\tilde q-\inf_{f\in\mathbb C^2\cap \mathbb C^1_\ell}\sup_{x\in\R^l}
H(x;\lambda,f))\,.
\end{equation}
(We recall that if $\mathcal{U}_\lambda=\emptyset$\, then 
$\inf_{f\in\mathcal{U}_\lambda}=\infty$\,.)
By part 2 of Lemma \ref{le:conc},
$\inf_{f\in\mathcal{U}_\lambda}
G(\lambda,f, m)\to\infty$ superlinearly, as $\lambda\to-\infty$\,,
which, when combined with \eqref{eq:81},
 implies that both sides of \eqref{eq:2} are finite.
We have that
\begin{equation*}
  \inf_{m\in\hat{\mathbb P}}
\sup_{\lambda\in\R}(\lambda \tilde q-\inf_{f\in\mathcal{U}_\lambda}
G(\lambda,f, m))
\ge \sup_{\lambda\in\R}\inf_{m\in\hat{\mathbb P}}
(\lambda \tilde q-\inf_{f\in\mathcal{U}_\lambda}
G(\lambda,f, m)) 
\ge \sup_{\lambda\in\R}
(\lambda \tilde q-\inf_{f\in\mathcal{U}_\lambda}\sup_{m\in\hat{\mathbb P}}
G(\lambda,f, m))\,.
\end{equation*}
The latter rightmost side being equal to the rightmost side of
\eqref{eq:2} and the definition of $F(\lambda)$ in \eqref{eq:29} imply that
\begin{equation}
  \label{eq:32}
  \sup_{\lambda\in\R}
(\lambda \tilde q-\sup_{m\in\hat{\mathbb P}}\inf_{f\in\mathcal{U}_\lambda}
G(\lambda,f, m))=\sup_{\lambda\in\R}(\lambda
\tilde q-\inf_{f\in\mathbb C^2\cap \mathbb C^1_\ell}\sup_{x\in\R^l}
H(x;\lambda,f))=
\sup_{\lambda\in\R}(\lambda
\tilde q-F(\lambda))\,.
\end{equation}
Therefore, for arbitrary $\lambda\in\R$ and $\tilde q\in\R$\,,
\begin{equation}
  \label{eq:34}
    \sup_{m\in\hat{\mathbb P}}\inf_{f\in\mathcal{U}_{\lambda}}
G(\lambda,f, m)\ge \lambda \tilde q-\sup_{\tilde\lambda\in\R}(\tilde\lambda
\tilde q-F(\tilde\lambda))\,.
\end{equation}
Since $F$ is a lower semicontinuous and convex function, it equals its
bidual, so,  taking supremum over $\tilde q$ in \eqref{eq:34} yields
the inequality $      \sup_{m\in\hat{\mathbb P}}\inf_{f\in\mathcal{U}_{\lambda}}
G(\lambda,f, m)\ge F(\lambda)\,.$
The opposite inequality being true by the definition of $F(\lambda)$
(see \eqref{eq:29})
implies that
\begin{equation}
  \label{eq:42}
F(\lambda)=  \sup_{m\in\hat{\mathbb P}}\inf_{f\in\mathcal{U}_{\lambda}}
G(\lambda,f, m)\,.
\end{equation}
In addition, owing to Lemma \ref{le:conc}, if
$\lambda<\overline\lambda$\,, then
\begin{equation}\label{eq:71}
  F(\lambda)=\sup_{m\in\hat{\mathbb   P}}\inf_{f\in
\mathbb C^2\cap \mathbb C^1_\ell}\breve G(\lambda,\nabla f,m)
=\sup_{m\in\hat{\mathbb   P}}\breve F(\lambda,m)\,.
\end{equation}
By convexity and lower semicontinuity, the latter equality extends to $\lambda=\overline\lambda$\,.

Since the infimum on the lefthand side of \eqref{eq:2} is attained at
$\hat m$\,, by \eqref{eq:42}, \begin{multline}
    \label{eq:3}
\sup_{\lambda\in\R}\bl(\lambda
q-\inf_{f\in\mathcal{U}_\lambda}G(\lambda,f,\hat m)\br)=
\inf_{m\in\hat{\mathbb P}}\sup_{\lambda\in\R}\bl(\lambda
q-\inf_{f\in\mathcal{U}_\lambda}G(\lambda,f,m)\br)\\
= 
\sup_{\lambda\in\R}\inf_{m\in\hat{\mathbb P}}\bl(\lambda
q-\inf_{f\in\mathcal{U}_\lambda}G(\lambda,f,m)\br)\,.
\end{multline}
By convexity of $\inf_{f\in\mathcal{U}_\lambda}
G(\lambda,f, \hat m)$ and of $\breve F(\lambda,\hat m)$ in $\lambda$\,,
we have that $\inf_{f\in\mathcal{U}_{\overline{\lambda}}}
G(\overline\lambda,f, \hat m)$ 
and  $\breve F(\overline\lambda,\hat m)$ are greater than
or equal to their respective lefthand limits at $\overline\lambda$\,,
so,
 by the fact that $\mathcal{U}_\lambda=\emptyset$ if
$\lambda>\overline\lambda$ and part 2 of Lemma \ref{le:conc},
\begin{equation*}
  \sup_{\lambda\in\R}(\lambda q-\inf_{f\in\mathcal{U}_\lambda}
G(\lambda,f, \hat m))=
\sup_{\lambda<\overline \lambda}(\lambda q-\inf_{f\in\mathcal{U}_\lambda}
G(\lambda,f, \hat m))=
\sup_{\lambda<\overline \lambda}(\lambda q-\breve F(\lambda, \hat m))
=\sup_{\lambda\le\overline \lambda}(\lambda q-\breve F(\lambda, \hat m))\,.
\end{equation*}
Similarly,
\begin{equation*}
  \inf_{m\in\hat{\mathbb P}}\sup_{\lambda\in\R}\bl(\lambda
q-\inf_{f\in\mathcal{U}_\lambda}G(\lambda,f,m)\br)=
\inf_{m\in\hat{\mathbb P}}\sup_{\lambda\le\overline\lambda}\bl(\lambda
q-\breve F(\lambda,m)\br)
\end{equation*}
and
\begin{equation*}
\sup_{\lambda\in\R}\inf_{m\in\hat{\mathbb P}}\bl(\lambda
q-\inf_{f\in\mathcal{U}_\lambda}G(\lambda,f,m)\br)
=  \sup_{\lambda\le\overline\lambda}\inf_{m\in\hat{\mathbb P}}\bl(\lambda
q-\breve F(\lambda, m)\br)\,,
\end{equation*}
so, by \eqref{eq:3}, 
\begin{equation*}
\sup_{\lambda\le\overline \lambda}(\lambda q-\breve F(\lambda, \hat m))=
    \inf_{m\in\hat{\mathbb P}}\sup_{\lambda\le\overline\lambda}\bl(\lambda
q-\breve F(\lambda, m)\br)=
\sup_{\lambda\le\overline\lambda}\inf_{m\in\hat{\mathbb P}}\bl(\lambda
q-\breve F(\lambda, m)\br)\,.
\end{equation*}
 Since, by Lemma \ref{le:conc}, $\breve F(\lambda,\hat m)$ is a lower
 semicontinuous function of $\lambda$ and
 $\breve F(\lambda,\hat m)\to\infty$ superlinearly as $\lambda\to
-\infty$\,,  the supremum on the leftmost side
  is attained at some $\hat \lambda$\,.
It follows  that
 $(\hat\lambda,\hat
m)$ is a saddle point of 
$\lambda q-\breve F(\lambda,m)$
in $(-\infty,\overline\lambda]\times\hat{\mathbb P}$\,. 
By Lemma \ref{le:conc},
 $\lambda q-\breve F(\lambda,m)$ is a strictly
concave function of $\lambda$ on $(-\infty,1)$ for all $m$\,, so $\hat \lambda$ is
specified uniquely,  see 
 Proposition 1.5 on p.169 in Ekeland and Temam
\cite{EkeTem76}. 

We obtain that
\begin{multline*}
  \sup_{\lambda\in\R}(\lambda q-F(\lambda))=
  \sup_{\lambda\le\overline\lambda}(\lambda q-F(\lambda))=
\sup_{\lambda\le\overline\lambda
}(\lambda q-\sup_{m\in\hat{\mathbb P}} \breve F(\lambda, m))=
\hat\lambda q-  
\breve F(\hat\lambda,\hat m)\\
=\hat\lambda q-  
\sup_{m\in\hat{\mathbb P}}\breve F(\hat\lambda, m)
=\hat\lambda q-F(\hat\lambda)\,.
\end{multline*}
Part 1 has been proved.

Suppose that $\hat\lambda<1$ and 
let $\hat f=f^{\hat\lambda}$\,, where $f^\lambda$ is introduced in
Lemma \ref{le:minmax}. Since $H(x;\hat\lambda,\hat f)=F(\hat\lambda)$
for all $x\in\R^l$\,, we have that $F(\hat\lambda)=G(\hat\lambda,\hat
f, m)=\breve G(\hat\lambda,\nabla\hat f, m)$\,, for all
$m\in\hat{\mathbb P}$\,. 
By \eqref{eq:12},
\begin{equation}
  \label{eq:115}
  \inf_{f\in\mathbb C^2\cap \mathbb C^1_\ell}
\sup_{m\in\hat{\mathbb P}}\breve G(\hat\lambda,\nabla f,m)\le\sup_{m\in\hat{\mathbb P}}
\breve G(\hat\lambda,\nabla\hat f,
  m)=F(\hat\lambda)=\breve G(\hat\lambda,\nabla\hat f,
  \hat m)\,.
\end{equation}
By \eqref{eq:71}, the 
inequality is actually equality and
$(\hat f,\hat m)$ is a saddle point of $\breve G(\hat\lambda,\nabla f,m)$ in
$(\mathbb C^2\cap\mathbb C^1_\ell)\times \hat{\mathbb P}$\,,
see, e.g., Proposition 2.156 on
p.104 in Bonnans and Shapiro \cite{BonSha00} or Proposition 1.2 on
p.167 in Ekeland and Temam \cite{EkeTem76}.
As a result,
\begin{equation}
  \label{eq:120}
\inf_{\tilde f\in\mathbb C^2\cap \mathbb C^1_\ell}\breve G(\hat\lambda,
  \nabla\tilde f,\hat m)=
\breve G(\hat\lambda,\nabla\hat f,\hat m)\,.
\end{equation}
By \eqref{eq:13} and $\mathbb C^2\cap \mathbb C^1_\ell$ being dense in
$\mathbb L^{1,2}_0(\R^l,\R^l,\hat m(x)\,dx)$\,, the lefthand side of
\eqref{eq:120} equals $\breve F(\lambda,\hat m)$\,, so, the infimum on
the righthand side of \eqref{eq:13} for $m=\hat m$ is attained at the 
gradient of the $\mathbb C^2\cap
\mathbb C^1_\ell$--function $\hat f$\,.

The following reasoning shows that $(\hat\lambda,\hat f,\hat m)$ is
a saddle point of $\lambda q-\breve G(\lambda,\nabla f,m)$
in $(-\infty,\overline\lambda]
\times(\mathbb C^2\cap \mathbb C^1_\ell)\times \hat{\mathbb P}$\,.
Let $\lambda\le\overline\lambda$\,,  
$f\in\mathbb C^2\cap \mathbb C^1_\ell$\,, and 
$m\in\hat{\mathbb P}$\,. 
Since $\breve G(\hat\lambda,\nabla\hat f,\hat m)
\ge \breve G(\hat\lambda,\nabla\hat f,m)$ by 
$(\hat f,\hat m)$ being a saddle point of $\breve G(\hat\lambda,\nabla
f,m)$\,, we
have that
\begin{equation}
  \label{eq:121}
  \hat\lambda q-\breve G(\hat\lambda,\nabla\hat f,\hat m)\le
\hat\lambda q-\breve G(\hat\lambda,\nabla\hat f,m)\,.
\end{equation}
By \eqref{eq:120}, by \eqref{eq:13},
and by $(\hat\lambda,\hat m)$ being a saddle point of 
$\lambda q-\breve F(\lambda,m)$\,,
\begin{equation}
  \label{eq:123}
\hat\lambda q- \breve G(\hat\lambda,\nabla\hat f,\hat m)=\hat\lambda q-
\breve F(  \hat\lambda,\hat m)\ge\lambda q-\hat F(  \lambda,\hat m)
  \ge
  \lambda q- \breve G(\lambda,\nabla f,\hat m)\,.
\end{equation}
Putting together \eqref{eq:121} and \eqref{eq:123} yields the required
property. 

Since $(\hat\lambda,\hat f,\hat m)$ is
a saddle point of $\lambda q-\breve G(\lambda,\nabla f, m)$
in $(-\infty,\overline\lambda]\times
(\mathbb C^2\cap\mathbb C^1_\ell)\times\hat{\mathbb P}$
 and $\lambda q-\breve G(\lambda,\nabla f, m)$
is strictly concave in $(\lambda,\nabla f)$ for all $m$\,, 
the pair $(\hat\lambda,\nabla\hat f)$ is specified uniquely, see
 Proposition 1.5 on p.169 of Ekeland and Temam
\cite{EkeTem76}. 
Equation \eqref{eq:103'} follows by 
Lemma \ref{le:minmax}. Since $\hat f$ is a stationary point of
$\breve G(\hat\lambda,\nabla f,\hat m)$\,,
the directional derivatives of
$\breve G(\hat\lambda,\nabla f,\hat m)$ at $\hat f$ are equal to zero, 
cf. Proposition 1.6
on p.169 in Ekeland and Temam \cite{EkeTem76}. By \eqref{eq:11},
\begin{equation}
  \label{eq:65}
    \int_{\R^l}\Bl(\breve H_p(x;\hat\lambda,\nabla \hat f(x))
-\frac{1}{2}\, \,\frac{
\bl(\text{div}\,({\sigma(x)}{\sigma(x)}^T\,
 \hat m(x))\br)^T}{\hat m(x)}\,
\Br)\nabla h(x)\,\hat m(x)\,dx=0\,,
\end{equation}
for all $h\in \mathbb C_0^2$\,. Integration by parts yields
 \eqref{eq:104'}.
In  more detail, by Theorem 4.17 on p.276 in Bonnans and Shapiro
\cite{BonSha00}, if $\lambda<1$\,, then  the function
 $\sup_{u\in\R^n}\bl(  M(u,x)
+\lambda
\abs{N(u,x)}^2/2+ p^T\sigma(x) N(u,x)\br)$\,, with the supremum
being attained at  unique point $\tilde u(x)$\,,  
has a  derivative  with respect to $p$ given by
$(\sigma(x) N(\tilde u(x),x))^T$\,, which, when combined with
 \eqref{eq:62} and \eqref{eq:65}, 
yields  \eqref{eq:104'}. By Example 1.7.11 (or Example 1.7.14)
 in Bogachev, Krylov, and
R\"ockner \cite{BogKryRoc}, $\hat m$ is specified uniquely by \eqref{eq:104'}.
Part 2 has been proved.

If $\hat \lambda=1$\,, then $\breve F(1,\hat m)<\infty$\,. By Lemma
\ref{le:conc}, $\nabla \hat f$ exists. The other properties in part 3 follow by
\eqref{eq:135} and \eqref{eq:61}. 
\end{proof}
\begin{remark}
  If $\hat\lambda<0$\,, then $H(x;\hat\lambda,f_\kappa)\to-\infty$ as
  $\abs{x}\to\infty$\,, where  $\kappa>0$ and is
  small enough, see Puhalskii and Stutzer \cite{PuhStu16}. In that
  case, the  theory in Keise and Sheu \cite{KaiShe06} and
  Ichihara \cite{Ich11} yields an alternative approach to 
the existence of solution $\hat m$ to \eqref{eq:103'}.
 If $\hat\lambda>0$\,, however, those results do not
  seem to apply.
\end{remark}
\begin{remark}
  If the suprema in  \eqref{eq:71} were
  attained, then $F(\lambda)$ would be strictly convex.
\end{remark}

   \begin{lemma}
  \label{le:saddle_2}
 Suppose that $\hat\lambda\le0$\,. Then, for $\kappa>0$
small enough,
\begin{equation*}
  \inf_{f\in\mathcal{A}_\kappa}
\sup_{\nu\in\mathcal{P}}
\int_{\R^l}\overline H(x;\hat\lambda,f, \hat u^\rho) \nu(dx)=
  \sup_{\nu\in\mathcal{P}}\inf_{f\in\mathbb C^2_0}
\int_{\R^l}\breve H(x;\hat\lambda,f,\hat u^\rho)\nu(dx)
=  \inf_{f\in\mathbb C^2}
\sup_{x\in\R^l}
\overline H(x;\hat\lambda,f, \hat u^\rho)
\,.
\end{equation*}
\end{lemma}
\begin{proof}
For $\kappa>0$ small enough, the function 
$\int_{\R^l}\overline H(x;\hat\lambda,f,\hat u^\rho)\,\nu(dx)$
is convex in $f\in\mathcal{A}_\kappa$\,, is concave and
 upper semicontinuous in $\nu\in\mathcal{P}$\,, and is
$\sup$--compact in $\nu$\,, the latter property being shown in analogy
with the proof of Lemma \ref{le:sup-comp}.
Invoking Theorem
7 on p.319 in Aubin and Ekeland \cite{AubEke84},
\begin{multline*}
  \inf_{f\in\mathbb C^2}
\sup_{x\in\R^l}
\overline H(x;\hat\lambda,f, \hat u^\rho)
=\inf_{f\in\mathbb C^2\cap \mathbb C^1_\ell}
  \sup_{\nu\in
    \mathcal{P}}\int_{\R^l}\overline H(x;\hat\lambda,f,\hat
  u^\rho)\,\nu(dx)\\
\le\inf_{f\in\mathcal{A}_\kappa}
  \sup_{\nu\in
    \mathcal{P}}\int_{\R^l}\overline H(x;\hat\lambda,f,\hat
  u^\rho)\,\nu(dx)
=    \sup_{\nu\in
    \mathcal{P}}\inf_{f\in\mathbb C^2\cap \mathbb C^1_\ell}
\int_{\R^l}\overline H(x;\hat\lambda,f,\hat u^\rho)\,\nu(dx)
\\
=    \sup_{\nu\in
    \mathcal{P}}\inf_{f\in\mathbb C^2}
\int_{\R^l}\overline H(x;\hat\lambda,f,\hat u^\rho)\,\nu(dx)
\le   \inf_{f\in\mathbb C^2}
\sup_{x\in\R^l}
\overline H(x;\hat\lambda,f, \hat u^\rho)
\,.
\end{multline*}

\end{proof}
\begin{remark}
  One can also show that, if $\kappa>0$ is small enough, then
  \begin{multline*}
      F(\lambda)=
  \sup_{\nu\in
    \mathcal{P}}\inf_{f\in\mathcal{A}_\kappa}
\int_{\R^l} H(x;\lambda,f)\,\nu(dx)=
  \inf_{f\in\mathcal{A}_\kappa}\sup_{\nu\in
    \mathcal{P}}
\int_{\R^l} H(x;\lambda,f)\,\nu(dx)\\=
\sup_{\nu\in\mathcal{P}}  \inf_{f\in\mathcal{A}_\kappa}
\int_{\R^l}\overline H(x;\hat\lambda,f,\hat u)\,\nu(dx)
\,.
  \end{multline*}

\end{remark}
\section{Proofs of the main results }
\label{sec:proof-bounds}

We prove  Theorem~\ref{the:bounds} by proving,
firstly, the upper bounds and, afterwards, the  lower bounds.

\subsection{The upper bounds}
\label{sec:upper-bounds}
This subsection contains the proofs of \eqref{eq:60} and \eqref{eq:9}. 
Let us note that, by (\ref{eq:5}),
\begin{multline}
  \label{eq:5a}
  L^\pi_t=
\int_0^1 M(\pi_{s}^t,X^t_s)\,ds
+\frac{1}{\sqrt{t}}\,\int_0^1 N(\pi_{s}^t,X_s^t)^T\,dW_{s}^t
\\=\int_0^1\int_{\R^l} M(\pi_{s}^t,x)\,\mu^t(ds,dx)
+\frac{1}{\sqrt{t}}\,\int_0^1 N(\pi_{s}^t,X_s^t)^T\,dW_{s}^t\,.
\end{multline}

\subsubsection{The proof of \eqref{eq:60}.}
\label{sec:proof}
By \eqref{eq:14} and It\^o's lemma, for 
$\mathbb C^2$--function $f$\,,
\begin{multline*}
  f(X_t)=f(X_0)+\int_0^t \nabla f(X_s)^T\theta(X_s)\,ds+
\frac{1}{2}\,\int_0^t \text{tr}\,\bl(\sigma(X_s)\sigma(X_s)^T\nabla^2
f(X_s)\br) \,ds\\
+\int_0^t \nabla f(X_s)^T\sigma(X_s)\,dW_s\,.
\end{multline*}
Since the process $\exp\bl(\int_0^t (\lambda 
N(\pi_s,X_s)+\nabla f(X_s)^T\sigma(X_s))\,dW_s
-(1/2)\int_0^t \abs{\lambda 
N(\pi_s,X_s)+\nabla f(X_s)\sigma(X_s)}^2\,ds\br)$ is a local
martingale, 
where   $\lambda\in\R$\,,
 by (\ref{eq:14}) and (\ref{eq:5a}), 
\begin{multline}\label{eq:1}
\mathbf{E}\exp\bl(t\lambda L^{\pi}_t+f(X_t)-f(X_0)-t
\int_0^1 \lambda M(\pi^t_s,X^t_s)\,ds-
t\int_0^1\nabla f(X^t_s)^T\,\theta(X^t_s)\,ds
\\-\frac{t}{2}\,\int_0^1\text{tr}\,({\sigma(X^t_s)}{\sigma(X^t_s)}^T\,
\nabla^2f(X^t_s))\,ds
-\frac{t}{2}\,\int_0^1\abs{\lambda 
N(\pi^t_s,X^t_s)+{\sigma(X^t_s)}^T\nabla f(X^t_s)}^2\,ds
\br)\le 1\,.
\end{multline}
Let  $\nu^t(dx)=\mu^t([0,1],dx)$\,. By
 \eqref{eq:40} and \eqref{eq:59}, for $\lambda\in[0,1)$\,,
 \begin{equation}
   \label{eq:110}
         \mathbf{E}\exp\bl(t\lambda L^{\pi}_t+
f(X_t)-f(X_0)-
t\int_{\R^l}H(x;\lambda, f)\,\nu^t(dx)\br)\le 1\,.
\end{equation}
Consequently,
\begin{equation*}
  \mathbf{E}\chi_{\{L^\pi_t\ge q\}}\exp\bl(t\lambda L^{\pi}_t
+f(X_t)-f(X_0)-
t\int_{\R^l}H(x;\lambda, f)\,\nu^t(dx)
\br)\le 1
\end{equation*}
Thus,\begin{equation*}
\ln
      \mathbf{E}\chi_{\{L^\pi_t\ge q\}}e^{f(X_t)-f(X_0)}\le
\sup_{\nu\in\mathcal{P}} \bl(-\lambda q t+
t\int_{\R^l}H(x;\lambda,f)\,\nu(dx)\br)
= -\lambda q t+
t\sup_{x\in\R^l}H(x;\lambda,f)\,.
\end{equation*}
By the reverse H\"older inequality, for arbitrary $\epsilon>0$\,,
\begin{equation*}
    \mathbf{E}\chi_{\{L^\pi_t\ge q\}}e^{f(X_t)-f(X_0)}\ge
\mathbf P(L^\pi_t\ge q)^{1+\epsilon}
\bl(\mathbf{E}e^{-
(f(X_t)-f(X_0))/\epsilon}\br)^{-\epsilon}\,,
\end{equation*}
so,
\begin{equation*}
    \frac{1+\epsilon}{t}\,\ln\mathbf P(L^\pi_t\ge q)\le
-\lambda q +
\sup_{x\in\R^l} H(x;\lambda,f)
+\frac{\epsilon}{t}\,\ln\mathbf{E}e^{-
(f(X_t)-f(X_0))/\epsilon}\,.
\end{equation*}
We may assume that $\inf_{f\in\mathbb C^2}
\sup_{x\in\R^l} H(x;\lambda,f)<\infty$\,. 
By Lemma \ref{le:minmax},  the latter infimum is
attained at  $ f^\lambda$\,. Since, by hypotheses,
 $ f^\lambda(x)\ge -C_1\abs{x}-C_2$ for some
positive $C_1$ and $C_2$ and $\abs{X_0}$ is bounded, we have that
\begin{equation*}
  \limsup_{t\to\infty}
    \frac{1+\epsilon}{t}\,\ln\mathbf P(L^\pi_t\ge q)\le
-\lambda q +
\inf_{f\in\mathbb C^2}\sup_{x\in\R^l} H(x;\lambda,f)
+\limsup_{t\to\infty}\frac{\epsilon}{t}\,\ln\mathbf{E}e^{
C_1\abs{X_t}/\epsilon}\,.
\end{equation*}
Consequently, by $\mathbf{E}e^{
C_1\abs{X_t}/\epsilon}$ being bounded in $t$ according to
 Lemma \ref{le:exp_moment} of the appendix  and by $\epsilon$ being
arbitrarily small,
\begin{equation*}
  \limsup_{t\to\infty}    \frac{1}{t}\,\ln\mathbf P(L^\pi_t\ge q)\le
-\bl(\lambda q -
\inf_{f\in\mathbb C^2}\sup_{x\in\R^l} 
H(x;\lambda,f)\br)
\end{equation*}
 yielding \eqref{eq:60}, if one recalls \eqref{eq:38}, \eqref{eq:29},
 and $F$ being convex so that the supremum in \eqref{eq:38} can be
 taken over $[0,1)$\,.

\subsubsection{The proof of \eqref{eq:9}}
\label{sec:proof-1}
Since $J^{\text{s}}_{q}=0$ when $\hat\lambda\ge0$\,,
   we may assume that
 $\hat \lambda<0$\,.
Letting $\pi^t_s=\hat u^\rho(X^t_s)$ in  \eqref{eq:1} yields,
for $f\in\mathbb C^2$\,, 
\begin{equation}
  \label{eq:70}
      \mathbf{E}\exp\bl(
 t\hat\lambda L^{\hat\pi^\rho}_t+f(X_t)- f(X_0)-
t\int_{\R^l}\overline H(x;\hat\lambda,f,\hat u^\rho)\,\nu^t(dx)
\br)\le 1\,.
\end{equation}
Therefore, on recalling that $\hat\lambda<0$\,,
\begin{multline}
  \label{eq:98}
\mathbf{E}\mathbf1_{\{L^{\hat\pi^\rho}\le q\}}\exp\bl(f(X_t)-
f(X_0)\br)\le
e^{-t\hat\lambda q}
        \mathbf{E}\exp\bl(
 t\hat\lambda L^{\hat\pi^\rho}_t+f(X_t)- f(X_0)\br)\\\le
e^{-t\hat\lambda q}\exp\bl(t\sup_{\nu\in\mathcal{P}}
\int_{\R^l}\overline H(x;\hat\lambda,f,\hat u^\rho)\,\nu(dx)\br)
\,.
\end{multline}
By the reverse H\"older inequality, for $\epsilon>0$\,,
\begin{equation}
  \label{eq:101}
  \mathbf{E}\mathbf1_{\{L^{\hat\pi^\rho}\le q\}}\exp\bl(
 f(X_t)- f(X_0)\br)
\ge   \mathbf{P}(L^{\hat\pi^\rho}\le q)^{1+\epsilon}
\mathbf E\exp\bl(e^{-(1/\epsilon)(f(X_t)- f(X_0))}\br)^{-\epsilon}\,.
\end{equation}
Assuming that $f\in\mathcal{A}_\kappa$\,, with $\kappa$ being small
enough as compared with $\epsilon$\,, we have, by \eqref{eq:77}, that
\begin{equation*}
  \limsup_{t\to\infty}\mathbf E\exp\bl(e^{-(1/\epsilon)(f(X_t)- f(X_0))}\br)^{1/t}\le1\,.
\end{equation*}
Therefore,
\begin{equation}
  \label{eq:44}
  \limsup_{t\to\infty}
\frac{1+\epsilon}{t}\,\ln \mathbf P(L^{\hat\pi^\rho}\le q)\le
-\hat\lambda q+\inf_{f\in\mathcal{A}_\kappa}\sup_{\nu\in\mathcal{P}}
\int_{\R^l}\overline H(x;\hat\lambda,f,\hat u^\rho)\,\nu(dx))\,.
\end{equation}
By Lemma \ref{le:saddle_2} and \eqref{eq:97},
\begin{equation*}
\limsup_{\rho\to\infty}    \limsup_{t\to\infty}
\frac{1}{t}\,\ln \mathbf P(L^{\hat\pi^\rho}\le q)\le
F(\hat\lambda)\,.
\end{equation*}

\subsection{The lower bounds}
\label{sec:lower-bounds}
In this subsection, we prove \eqref{eq:58} and \eqref{eq:27}.
Let us assume that  $\hat \lambda<\overline\lambda$\,. 
We prove that, if $q'>q$\,, then
\begin{subequations}
  \begin{align}
  \label{eq:48}
            \liminf_{t\to\infty}\frac{1}{t}\ln
\mathbf{P}(L^{\pi}_t< q')\ge -
\bl(\hat\lambda q-G(\hat\lambda,\hat f,\hat m)
 \br) \intertext{and that, if $q''<q$\,, then}
  \label{eq:39a}
            \liminf_{t\to\infty}\frac{1}{t}\ln
\mathbf{P}(L^{\hat\pi}_t> q'')\ge -
\bl(\hat\lambda q-
G(\hat\lambda,\hat f,\hat m)\br)\,.
\end{align}
\end{subequations}
We begin with showing that
\begin{equation}
  \label{eq:51}
  \hat \lambda q
-G(\hat\lambda,\hat f,\hat m)
=\frac{1}{2}\,\int_{\R^l}\abs{\hat\lambda
N(\hat u(x),x)+
\sigma(x)^T\nabla  \hat f(x)
}^2\hat m(x)\,dx\,.
\end{equation}
Since $(\hat\lambda,\hat f,\hat m)$ is a saddle point of 
$\lambda q-\breve G(\lambda,\nabla f,  m)$ in
$(-\infty,\overline\lambda]\times (\mathbb C^2\cap \mathbb C^1_\ell)
\times \mathbb P$
 by Lemma \ref{le:saddle_3},
  $\hat\lambda$ is the point of the maximum of the concave   function
$\lambda q-\breve G(\lambda,\nabla\hat f, \hat m)$
on $(-\infty,\overline\lambda]$\,. Since
$\hat\lambda<\overline\lambda$ and $\breve G(\lambda,\nabla\hat f,
\hat m)$ is differentiable on $(-\infty,\overline\lambda)$\,, the $\lambda$--derivative of 
$\breve G(\lambda,
\nabla  \hat f,\hat m)$ at $\hat\lambda$ equals zero.
By \eqref{eq:10a} of Lemma \ref{le:conc},
\begin{equation}
  \label{eq:132}
  \frac{d}{d\lambda}\,\breve G(\lambda,
\nabla  \hat f,\hat m)\Big|_{\lambda=\hat\lambda}=
    \int_{\R^l}
\bl(  M(\hat u(x),x)
+\hat\lambda\abs{N(\hat u(x),x)}^2+
\nabla \hat f(x)^T\sigma(x) N(\hat u(x),x)\br)
\hat m(x)\,dx\,,
\end{equation}
so,\begin{equation}
  \label{eq:131}
    \int_{\R^l}
\bl(  M(\hat u(x),x)
+\hat\lambda\abs{N(\hat u(x),x)}^2+
\nabla \hat f(x)^T\sigma(x) N(\hat u(x),x)\br)
\hat m(x)\,dx= q\,.
\end{equation}
Therefore, by \eqref{eq:40}, \eqref{eq:59}, and \eqref{eq:62},
\begin{multline}
  \label{eq:30'}
    \hat \lambda q-G(\hat\lambda,\hat f,\hat m)
=\hat\lambda\int_{\R^l}
\bl(  M(\hat u(x),x)
+\hat\lambda\abs{N(\hat u(x),x)}^2+
\nabla \hat f(x)^T\sigma(x) N(\hat u(x),x)
\br)
\hat m(x)\,dx\\
-\int_{\R^l}
\bl( \hat\lambda M(\hat u(x),x)
+\frac{1}{2}\,\hat\lambda^2\abs{N(\hat u(x),x)}^2+\hat\lambda\,
\nabla \hat f(x)^T\sigma(x) N(\hat u(x),x)
\\
+\frac{1}{2}\,\abs{\sigma(x)^T\nabla \hat f(x)}^2+\nabla \hat
f(x)^T\theta(x)
+\frac{1}{2}\,\text{tr}\,({\sigma(x)}{\sigma(x)}^T\nabla^2 \hat f(x)\,
)\br)
\hat m(x)\,dx\\
=\int_{\R^l}
\frac{1}{2}\,\hat\lambda^2\abs{N(\hat u(x),x)}^2
\hat m(x)\,dx
-\int_{\R^l}\bl(\frac{1}{2}\,\abs{\sigma(x)^T\nabla \hat f(x)}^2+\nabla \hat
f(x)^T\theta(x)
\\+\frac{1}{2}\,\text{tr}\,({\sigma(x)}{\sigma(x)}^T\nabla^2 \hat f(x)\,
)\br)
\hat m(x)\,dx\,.
\end{multline}

Integration by parts in 
\eqref{eq:104'}
combined with 
the facts that $\abs{\nabla \hat f(x)}$ grows at most linearly
 with $\abs{x}$\,, that $\hat u(x)$ is a linear function of $\nabla
\hat f(x)$ by \eqref{eq:69}, 
that $\int_{\R^l}\abs{x}^2\,\hat m(x)\,dx<\infty$\,, and that
$\int_{\R^l}\abs{\nabla \hat m(x)}^2/\hat m(x)\,dx<\infty$\,,
shows that \eqref{eq:104'} holds with $\hat f(x)$ as $h(x)$\,.
Substitution on the rightmost side of \eqref{eq:30'} yields \eqref{eq:51}.

 Let
$\hat W^{t}_s$ for $s\in[0,1]$  and measure $\hat{\mathbf{P}}^{t}$ 
be defined by the respective equations

\begin{equation}
  \label{eq:34'}
\hat W^{t}_s=  W^t_s-\sqrt{t}\int_0^s(
\hat\lambda N(\hat u(X^t_{\hat
s}),X^t_{\hat s})
+\sigma (X^t_{\hat s})^T\nabla  \hat f(X^t_{\tilde s}) )\,d\tilde s
\end{equation}
and 
\begin{multline}
  \label{eq:35'}
  \frac{d\hat{\mathbf{P}}^{t}}{d\mathbf{P}}=
\exp\bl(\sqrt{t}\,\int_0^1
(\hat\lambda N(\hat u(X^t_s),X^t_s)+\sigma(X^t_s)^T\nabla
\hat f(X^t_s)
)^T\, 
d W^t_s\\-
\frac{t}{2}\,\int_0^1\abs{
\hat\lambda N
(\hat u(X^t_s),X^t_s)+{\sigma(X^t_s)}^T\nabla \hat f(X^t_s)}^2\,ds\br)\,.
\end{multline}
A multidimensional extension of
  Theorem 4.7 on p.137 in
 Liptser and Shiryayev \cite{LipShi77}, which is proved similarly, obtains  that, given $t>0$\,,
 there exists $\gamma'>0$ such that
$\sup_{s\le t}
\mathbf Ee^{\gamma'\abs{X_s}^2}<\infty$\,.  By Example 3 on pp.220,221 in
 Liptser and Shiryayev \cite{LipShi77}
 and the linear growth condition 
on $\nabla \hat{f}(x)$\,, the expectation of the righthand side of
\eqref{eq:35'} with respect to $\mathbf P$ equals unity.
 Therefore, $\hat{\mathbf{P}}^{t}$
 is a valid 
probability measure and 
the process $(\hat W^{t}_s,\,s\in[0,1])$ is a standard Wiener process
under $\hat{\mathbf{P}}^{t}$\,, see Lemma 6.4 on p.216 in
Liptser and Shiryayev \cite{LipShi77} and Theorem 5.1 on p.191 in
Karatzas and Shreve \cite{KarShr88}.

By \eqref{eq:8} and  \eqref{eq:69},
\begin{equation*}
  a(x)- r(x)\ind+b(x)
    (\hat\lambda 
N(   \hat u(x),x)+
\sigma(x)^T\nabla  \hat f(x))=c(x)\hat u(x)\,.
\end{equation*}
It follows that
\begin{multline}
  \label{eq:6}
L^\pi_t=    \int_0^1M(\pi^t_s,X^t_s)\,ds+\frac{1}{\sqrt{t}}\,\int_0^1 
N(\pi^t_s,X^t_s)^T\,d W^t_s=
\int_0^1M(\pi^t_s,X^t_s)\,ds\\+
\int_0^1N(\pi^t_s,X^t_s)^T (
\hat\lambda N(\hat u(X^t_s
),X^t_s)
+\sigma (X^t_s)^T\nabla  \hat f(X^t_s) )\,ds+
\frac{1}{\sqrt{t}}\,\int_0^1 
N(\pi^t_s,X^t_s)^T\,d \hat W^{t}_s\\=\frac{1}{t}\,\ln\mathcal{E}_1^t
+
\int_0^1M(\hat u(X^t_s),X^t_s)\,ds+
\int_0^1N(\hat u(X^t_s),X^t_s)^T (
\hat\lambda N(\hat u(X^t_s
),X^t_s)
+\sigma (X^t_s)^T\nabla  \hat f(X^t_s) )\,ds\\+
\frac{1}{\sqrt{t}}\,\int_0^1 
N(\hat u(X^t_s),X^t_s)^T\,d \hat W^{t}_s\,,
\end{multline}
where
 $\mathcal{E}_s^t$ represents the stochastic exponential defined by
\begin{equation*}
  \mathcal{E}_s^t=\exp\bl(
\sqrt{t}\,\int_0^s(\pi^t_{\tilde s}-\hat u(X^t_{\tilde
  s}))^Tb(X^t_{\tilde s})
\,d\hat W^{t}_{\tilde s}\\-
\frac{t}{2}\,\int_0^s
\norm{\pi^t_{\tilde s}-\hat u(X^t_{\tilde s})}_{c(X^t_{\tilde s})}^2 d\tilde s\br)\,.
\end{equation*}
By  \eqref{eq:35'} and \eqref{eq:6}, for $\delta>0$\,,
\begin{multline}
  \label{eq:17}
    \mathbf{P}\bl(L^{\pi}_t< q+3\delta\br)=
\hat{\mathbf{E}}^{t}\chi_{
\displaystyle\{
\int_0^1
M(\pi_{s}^t,X^t_s)\,ds
+\frac{1}{\sqrt{t}}\,\int_0^1 N(\pi_{s}^t,X_s^t)^T\,dW_s^t< q+3\delta
\}}\\
\exp\bl(-\sqrt{t}\int_0^1(\hat\lambda
 N(\hat u(X^t_s),X^t_s)+\sigma( X^t_s)^T\nabla  \hat f(X^t_s)
)^T 
\,d\hat W^{t}_s\\
-\frac{t}{2}\,
\int_0^1\abs{\hat\lambda
 N(\hat u(X^t_s),X^t_s)+\sigma(X^t_s)^T\nabla  \hat f(X^t_s)}^2\,ds\br)\\\ge
\hat{\mathbf{E}}^{t}\chi_{\Big\{\displaystyle
\frac{1}{t}\ln\mathcal{E}^t_1<\delta\Big\}}\,
\chi_{\Big\{\displaystyle
\frac{1}{\sqrt{t}}\,\abs{\int_0^1 
N(\hat u(X^t_s),X^t_s)^T\,d \hat W^{t}_s}<\delta\Big\}}
\chi_{\Big\{\displaystyle
\int_{\R^l}M(\hat u(x),x)\,\nu^t(dx)}\\{+
\int_{\R^l}N(\hat u(x),x)^T (
\hat\lambda N(\hat u(x),x)
+\sigma (x)^T\nabla  \hat f(x) )\,\nu^t(dx)< q+\delta\Big\}}
\\
\chi_{\Big\{\displaystyle
\frac{1}{\sqrt t}\,\abs{\int_0^1(
\hat\lambda N(\hat u(X^t_s),X^t_s)+\sigma(X^t_s)^T\nabla  \hat f(X^t_s)
 )^T\,d\hat W^{t}_s}
< \delta\Big\}}\\
\chi_{\Big\{\displaystyle
\int_{\R^L}\abs{\hat\lambda N(\hat u(x),x)
+\sigma(x)^T\nabla 
  \hat f(x)}^2\,\nu^t(dx)-
\int_{\R^l}\abs{\hat\lambda N(\hat u(x),x)
+\sigma(x)^T\nabla  \hat f(x)}^2
\hat m(x)\,dx<2\delta\Big\}}
\\\exp\bl(-2\delta t
-\frac{t}{2}\,\int_{\R^l}\abs{\hat\lambda N(\hat u(x),x)+
\sigma(x)^T\nabla  \hat f(x)
}^2\hat m(x)\,dx\br)\,.
\end{multline}
We will work with the terms on the righthand side in order.
Since $
\hat{\mathbf E}^{t}\mathcal{E}_1^t\le 1$\,, Markov's inequality yields
the convergence
\begin{equation}
  \label{eq:38'}
  \lim_{t\to\infty}
\hat{\mathbf P}^{t}\bl(  \frac{1}{t}\ln\mathcal{E}^t_1<\delta\br)=1\,.
\end{equation}

By \eqref{eq:14} and \eqref{eq:34'},
\begin{equation*}
        dX^t_s=t\,\theta(X^t_s)\,ds+t\,
\sigma(X^t_s)\,\bl( \hat\lambda N(\hat u(X^t_s),X^t_s)+
\sigma(X^t_s)^T\nabla  \hat f(X^t_{ s})
 \br)
\,ds+\sqrt{t}\sigma(X^t_s) d\hat W^{t}_s\,.
\end{equation*}
Hence, the process $X=(X_s\,,s\ge0)=(X^t_{s/t}\,,s\ge0)$ satisfies the
equation
\begin{equation*}
        d X_s=
\theta(X_s)\,ds+
\sigma( X_s)\,\bl(\hat\lambda N(\hat u( X_s), X_s)+
\sigma( X_s)^T\nabla  \hat f( X_{ s})
 \br)
\,ds
+\sigma(X_s) d\tilde W^t_s\,,
\end{equation*}
$(\tilde W_s^t)$ being a standard Wiener process under $\hat{\mathbf P}^{t}$\,.
We note that by Theorem 10.1.3 on p.251 in Stroock and Varadhan
\cite{StrVar79} the distribution of $X$ under $\hat{\mathbf P}^{t}$ is
specified uniquely. In particular, it does not depend on $t$\,.

We show that if $g(x)$ is a continuous function such that
 $\abs{g(x)}\le K(1+\abs{x}^2)$\,, for all $x\in\R^l$ and some $K>0$\,, then
\begin{equation}
  \label{eq:48'}
  \lim_{t\to\infty}
\hat{\mathbf P}^{t}\Bl(\abs{\int_{\R^l}g(x)
\nu^t(dx)-\int_{\R^l}g(x)\hat m(x)\,dx}>\epsilon\Br)=0\,.
\end{equation}
 Since $\hat m(x)$ is a unique solution to
 \eqref{eq:104'},
by Theorem 1.7.5 
in Bogachev, Krylov, and
R\"ockner \cite{BogKryRoc}, 
  $\hat m(x)\,dx$ is a unique  invariant measure of $X$
under $\hat{\mathbf P}^{t}$\,, see also Proposition 9.2 on p.239 in 
Ethier and Kurtz \cite{EthKur86}.
 It is thus an
ergodic measure.
We recall that $\hat m\in\mathbb{\hat P}$\,, so
$\int_{\R^l}\abs{x}^2\hat m(x)\,dx<\infty$\,. Let $P^\ast$ denote the
probability measure on the space $\mathbb C(\R_+,\R^l)$ of continuous
$\R^l$--valued functions equipped with the locally uniform topology
that is defined by 
$P^\ast(B)=\int_{\R^l}P_x(B)\,\hat m(x)\,dx$\,, where $P_x$ is the
distribution in $\mathbb C(\R_+,\R^l)$ of process $X$ started at
$x$\,.
Since $\hat m(x)\,dx$ is ergodic, so is $P^\ast$, see
Corollary on  p.12 in 
Skorokhod \cite{Sko89}. Hence,
 $P^\ast$--a.s.,
 \begin{equation}
   \label{eq:26}
     \lim_{s\to\infty}\frac{1}{s}\,\int_0^s g(\tilde X_{\tilde s})\,d\tilde s=
\int_{\R^l}g(x)\hat m(x)\,dx\,,
\end{equation} 
see, e.g., Theorem 3 on p.9 in 
Skorokhod \cite{Sko89}, with $\tilde X$ representing a generic element of 
$\mathbb C(\R_+,\R^l)$\,. Let $\mathcal C$ denote the complement of 
the set of elements of $\mathbb C(\R_+,\R^l)$ such that
 \eqref{eq:26} holds. 
By Proposition 1.2.18 in
Bogachev, Krylov, and
R\"ockner \cite{BogKryRoc}, $\hat m(x)$ is  continuous and strictly positive.
Since $P^\ast(\mathcal C)=0$\,, we have that
 $P_x(\mathcal C)=0$ for almost all $x\in\R^l$ with respect to 
Lebesgue measure. It follows that if $X_0$ has an absolutely
continuous distribution $n(x)\,dx$\,, then 
$\int_{\R^l}P_x(\mathcal C)n(x)\,dx=0$\,, which means that \eqref{eq:26} holds
a.s. w.r.t. $\hat{\mathbf P}$\,, the latter symbol denoting the
distribution of $X$ on the space of trajectories.
  If the distribution of $X_0$ is not absolutely continuous,
then the distribution of  $X_1$ is because the transition probability has a
density, see pp. 220--226 in Stroock and Varadhan \cite{StrVar79}.
Hence, \eqref{eq:26} holds $\hat{\mathbf P}$--a.s. for that case too.
We have proved \eqref{eq:48'}.

By  \eqref{eq:69},
  the linear growth condition on $\nabla \hat f(x)$\,, and \eqref{eq:48'},
\begin{multline}
  \label{eq:67}
    \lim_{t\to\infty}\hat{\mathbf P}^{t}
\bl(\big|\int_{\R^l}\abs{\hat\lambda N(\hat u(x),x)+
\sigma(x)^T\nabla 
   \hat f(x)}^2\,\nu^t(dx)\\-
\int_{\R^l}\abs{\hat\lambda N(\hat u(x),x)+\sigma(x)^T\nabla  \hat f(x)}^2
m(x)\,dx\big|<2\delta\br)=1\,.
\end{multline}
Since,  for $\eta>0$\,, by the L\'englart--Rebolledo inequality, see
Theorem 3 on p.66 in
Liptser and Shiryayev \cite{lipshir}, 
\begin{multline*}
 \hat{\mathbf{P}}^{t}\bl(\abs{\frac{1}{\sqrt{t}}\,
\int_0^1(\hat\lambda N(\hat u(X^t_s),X^t_s)+
\sigma(x)^T\nabla  \hat f(X^t_s) )\,d\hat W^{t}_s}\ge
\delta\br)
\\\le 
\frac{\eta
}{\delta^2 }+\hat{\mathbf{P}}^{t}\bl(
\int_0^1\abs{
\hat\lambda N(\hat u(X^t_s),X^t_s)+
\sigma(x)^T\nabla  \hat f(X^t_s)
}^2\,ds\ge\eta t\br)\,,
\end{multline*}
 we conclude that
 \begin{equation}
   \label{eq:33}
   \lim_{t\to\infty}
\hat{\mathbf{P}}^{t}\bl(\frac{1}{\sqrt{t}}\,
\abs{\int_0^1(
\hat\lambda N(\hat u(X^t_s),X^t_s)+
\sigma(X^t_s)^T\nabla  \hat f(X^t_s)
 )\,d\hat W^{t}_s}< \delta\br)
=1\,. 
\end{equation}
Similarly,
\begin{equation}
  \label{eq:35}
  \lim_{t\to\infty}
\hat{\mathbf{P}}^{t}\bl(
\frac{1}{\sqrt{t}}\,\abs{\int_0^1 
N(\hat u(X^t_s),X^t_s)^T\,d \hat W^{t}_s}
<\delta\br)=1\,.
\end{equation}

By  \eqref{eq:131} 
and \eqref{eq:48'},
\begin{equation*}
  \lim_{t\to\infty}\hat{\mathbf P}^{t}\bl(
\int_{\R^l}\bl(M(\hat u(x),x)+N(\hat u(x),x)^T (
\hat\lambda N(\hat u(x
),x)
+\sigma (x)^T\nabla  \hat f(x) )\br)\,\nu^t(dx)< q+\delta\br)=1\,.
\end{equation*}
  Recalling  \eqref{eq:38'} and \eqref{eq:17}  obtains that
  \begin{equation}
    \label{eq:124}
       \liminf_{t\to\infty}
\frac{1}{t}\,\ln     \mathbf{P}\bl(L^{\pi}_t< q'\br)\ge
-\frac{1}{2}\,\int_{\R^l}\abs{\hat\lambda N(\hat u(x),x)+
\sigma(x)^T\nabla  \hat f(x)
}^2m(x)\,dx\,,
  \end{equation}
so, \eqref{eq:48} follows from
  \eqref{eq:51}.

In order to prove \eqref{eq:39a},
we note that if $\pi^t_s=\hat u(X^t_s)$\,, then 
 $\mathcal{E}^t_s=0$ in \eqref{eq:6}, so 
\begin{multline*}
      \int_0^1M(\hat u(X^t_s),X^t_s)\,ds+\frac{1}{\sqrt{t}}\,\int_0^1 
N(\hat u(X^t_s),X^t_s)^T\,d W^t_s=
\int_0^1M(\hat u(X^t_s),X^t_s)\,ds\\+
\int_0^1N(\hat u(X^t_s),X^t_s)^T (
\hat\lambda N(\hat u(X^t_s
),X^t_s)
+\sigma (X^t_s)^T\nabla  \hat f(X^t_s) )\,ds+
\frac{1}{\sqrt{t}}\,\int_0^1 
N(\hat u(X^t_s),X^t_s)^T\,d \hat W^{t}_s\,.
\end{multline*}
On recalling \eqref{eq:5a}, similarly to \eqref{eq:17},
\begin{multline}
  \label{eq:56}
    \mathbf{P}\bl(L^{\hat\pi}_t> q-2\delta\br)=
\hat{\mathbf{E}}^{t}\chi_{
\displaystyle\Big\{
\int_0^1\Bl(
M(\hat u(X^t_s),X^t_s)
+N(\hat u(X^t_s),X^t_s)^T \bl(
\hat\lambda N(\hat u(X^t_s
),X^t_s)}\\{
+\sigma (X^t_s)^T\nabla  \hat f(X^t_s)\br) \Br)\,ds+
\frac{1}{\sqrt{t}}\,\int_0^1 
N(\hat u(X^t_s),X^t_s)^T\,d \hat W^{t}_s>q-2\delta}\Big\}\\
\exp\bl(-\sqrt{t}\int_0^1(\hat\lambda
 N(\hat u(X^t_s),X^t_s)+\sigma( X^t_s)^T\nabla  \hat f(X^t_s)
)^T 
\,d\hat W^{t}_s
+\frac{t}{2}\,
\int_0^1\abs{\hat\lambda
 N(\hat u(X^t_s),X^t_s)+\sigma(X^t_s)^T\nabla  \hat f(X^t_s)}^2\,ds\br)\\\ge
\chi_{\Big\{\displaystyle
\frac{1}{\sqrt{t}}\,\int_0^1 
N(\hat u(X^t_s),X^t_s)^T\,d \hat W^{t}_s>-\delta\Big\}}
\chi_{\Big\{\displaystyle\int_{\R^l}
\Bl(M(\hat u(x),x)
}\\{+N(\hat u(x),x)^T \bl(
\hat\lambda N(\hat u(x
),x)
+\sigma (x)^T\nabla  \hat f(x) \br)\Br)\,\nu^t(dx)\ge q-\delta\Big\}}
\\\chi_{\Big\{\displaystyle
\frac{1}{\sqrt t}\,\int_0^1(
\hat\lambda N(\hat u(X^t_s),X^t_s)+\sigma(X^t_s)^T\nabla  \hat f(X^t_s)
 )^T\,d\hat W^{t}_s
\ge -\delta\Big\}}\\
\chi_{\Big\{\displaystyle
\int_{\R^l}\abs{\hat\lambda N(\hat u(x),x)
+\sigma(x)^T\nabla 
  \hat f(x)}^2\,\nu^t(dx)-
\int_{\R^l}\abs{\hat\lambda N(\hat u(x),x)
+\sigma(x)^T\nabla  \hat f(x)}^2
 \hat m(x)\,dx\le2\delta\Big\}}
\\\exp\bl(-2\delta t
-\frac{t}{2}\,\int_{\R^l}\abs{\hat\lambda N(\hat u(x),x)+
\sigma(x)^T\nabla  \hat f(x)
}^2 \hat m(x)\,dx\br)\,.
\end{multline}

One still has \eqref{eq:67}, \eqref{eq:33}, and \eqref{eq:35}.
By  \eqref{eq:131} 
and \eqref{eq:48'},
\begin{equation*}
  \lim_{t\to\infty}\hat{\mathbf P}^{t}\bl(
\int_{\R^l}\bl(M(\hat u(x),x)+N(\hat u(x),x)^T (
\hat\lambda N(\hat u(x
),x)
+\sigma (x)^T\nabla  \hat f(x) )\br)\,\nu^t(dx)> q-\delta\br)=1\,.
\end{equation*}
  Recalling  \eqref{eq:56}  yields
  \begin{equation}
    \label{eq:128}
       \liminf_{t\to\infty}
\frac{1}{t}\,\ln     \mathbf{P}\bl(L^{\hat\pi}_t> q''\br)\ge
-\frac{1}{2}\,\int_{\R^l}\abs{\hat\lambda N(\hat u(x),x)+
\sigma(x)^T\nabla  \hat f(x)
}^2\hat m(x)\,dx\,,
  \end{equation}
so, \eqref{eq:39a} follows from
  \eqref{eq:51}.

Reversing the roles of $q$ and $q'$ in  \eqref{eq:48}
 and reversing the
roles of $q$ and $q''$ in \eqref{eq:39a} 
 obtain that, if $q'<q$\,, then
\begin{align*}
                \liminf_{t\to\infty}\frac{1}{t}\ln
\mathbf{P}(L^{\pi}_t< q)\ge -J^{\text{s}}_{q'}
 \intertext{and that, if $q''>q$\,, then}
            \liminf_{t\to\infty}\frac{1}{t}\ln
\mathbf{P}(L^{\hat\pi}_t> q)\ge -J^{\text{o}}_{q''}\,.
\end{align*}
Letting $q'\to q$ and $q''\to q$ and using the continuity of
$J_q^{\text{s}}$ and $J_q^{\text{o}}$\,, respectively, which
properties hold by
Lemma \ref{le:conc},  prove
   \eqref{eq:58} 
and \eqref{eq:27}, respectively, provided
 $\hat\lambda<\overline \lambda$\,.

 Suppose that $\hat\lambda=\overline\lambda<1$\,. 
Let $\hat f=f^{\hat\lambda}$ be as in  Lemma \ref{le:minmax}.
Then \eqref{eq:124} and
 \eqref{eq:128} hold by a similar argument to the one above.
Since $\overline\lambda$ maximises $\lambda q-\breve G(\lambda,\hat f,\hat m)$
over $\lambda$  we have that
$  (d/d\lambda)\,\breve G(\lambda,\hat f,\hat
  m)\vert_{\overline\lambda-}\le q\,.
$
By \eqref{eq:132} still holding, 
we have that in \eqref{eq:131}
the $=$ sign has to be
replaced with $\le$\,. By $\overline\lambda$ being positive, the
first $=$ sign in \eqref{eq:30'} needs to be replaced with
$\ge$\,, so does the $=$ sign in \eqref{eq:51}.  By
\eqref{eq:124} and \eqref{eq:128},
 one obtains \eqref{eq:58} and \eqref{eq:27}, respectively.

Suppose that $\hat\lambda=\overline\lambda=1$\,.
Since $\hat\lambda>0$\,, so, $J^{\text{s}}_q=0$ and
$J^{\text{o}}_q>0$\,, \eqref{eq:58} is a consequence of \eqref{eq:60}.
We now work toward \eqref{eq:27}.
Since $1$ maximises $\lambda q-\breve F(\lambda,\hat m)$
over $\lambda$  and the
function $\breve F(\lambda,\hat m)$ 
is a convex function of $\lambda$\,,
 $\breve F(1,\hat m)<\infty $ and
$
d/d\lambda\,\breve F(\lambda,\hat m)\big|_{1-}\le q\,.$
Let $\nabla\hat f$ be defined as in part 3 of Lemma \ref{le:saddle_3},
i.e., let $\inf_{\nabla f\in \mathbb
  L^{1,2}_0(\R^l,\R^l,\hat m(x)\,dx)}\breve G(1,\nabla f,\hat m)$ be
attained  at $\nabla\hat f$\,.
By \eqref{eq:163} of Lemma \ref{le:conc}, 
$  d/d\lambda\,
\breve G(\lambda,\nabla\hat f,\hat m)\big|_{1-}\le q\,.$
By part 3 of Lemma \ref{le:saddle_3},
$\breve G(1,\nabla\hat f,\hat m)$ being finite 
 implies that, $\hat m(x)\,dx$--a.e.,
 \begin{equation}
   \label{eq:25}
b(x)\sigma(x)^T  \nabla  \hat f(x)=
 b(x)\beta(x)-a(x)+r(x)\mathbf 1\,.
 \end{equation}
By \eqref{eq:10a} of Lemma \ref{le:conc},  if $\lambda<1$\,, then
  \begin{equation*}
       \frac{d\breve G(\lambda,
\nabla   \hat f,\hat m)}{d\lambda}\,    =
 \int_{\R^l}\bl( M( u^{\lambda,\nabla \hat f}(x),x)
+\lambda
\abs{N(u^{\lambda,\nabla\hat f}(x),x)}^2+ 
N(u^{\lambda,\nabla \hat f}(x)^T\sigma(x)^T\nabla\hat f(x),x)\br)
\hat m(x)\,dx\,,
  \end{equation*}
where $ u^{\lambda,\nabla\hat f}(x)$ is defined by \eqref{eq:39} with 
$\nabla\hat f(x)$
as $p$\,.
On noting that by \eqref{eq:25} the  limit, as $\lambda\uparrow 1$\,, in
\eqref{eq:39} with $\nabla\hat f(x)$ as $p$ 
  equals  $c(x)^{-1}b(x)\beta$\,,
we have, see Theorem 24.1 on p.227 in Rockafellar \cite{Rock} for the
first equality below, that
\begin{multline*}
\frac{d}{d\lambda}
\,\breve G(\lambda,\nabla\hat f,\hat m)\big|_{1-}=
\lim_{\lambda\uparrow 1}\frac{d}{d\lambda}
\,\breve G(\lambda,\nabla\hat f,\hat m)
     =\int_{\R^l}\bl(  M( c(x)^{-1}b(x)\beta(x),x)\\
+\abs{N(c(x)^{-1}b(x)\beta(x),x)}^2
+ N(c(x)^{-1}b(x)\beta(x),x)^T{\sigma(x)}^T\nabla
 \hat f(x)\br)\,\hat m(x)\,dx\,.
\end{multline*}
We recall that
 $\hat v(x)$ is defined to be a bounded continuous function
with values in
 the range
of $b(x)^T$ such that
$\abs{\hat v(x)}^2/2=q-
d/d\lambda\,\breve F(\lambda,\hat m)\Big|_{1-}$
and
$\hat u(x)=
c(x)^{-1}b(x)(\beta(x)+\hat v(x))$\,.
By Lemma \ref{le:conc},
   $d/d\lambda\,\breve F(\lambda,\hat m)\big|_{1-}=d/d\lambda\,
\breve G(\lambda,\nabla \hat f,\hat m)\big|_{1-}$\,.
Since 
the vectors $b(x)^T c(x)^{-1}b(x)\beta(x)-\beta(x) $ and $b(x)^Tc(x)^{-1}b(x)
\hat v(x)$ 
are orthogonal, with
 the former  being in the null space of $b(x)$ and the latter
 being in the range of $b(x)^T$\,, substitution in
 \eqref{eq:4} and \eqref{eq:8} with the account of \eqref{eq:135} yields
\begin{multline}
  \label{eq:89}
    \int_{\R^l}
\bl(   M(\hat u(x),x)
+\abs{N(\hat u(x),x)}^2
+ N(\hat u(x),x)^T{\sigma(x)}^T\nabla \hat f(x)
\br)\,\hat m(x)\,dx\\
=\frac{d}{d\lambda}
\,\breve G(\lambda,\nabla\hat f,\hat m)\big|_{1-}
+\int_{\R^l}\frac{\abs{\hat v(x)}^2}{2}\,\hat m(x)\,dx=q\,.
\end{multline}
(As a consequence,  \eqref{eq:131} holds in this case too.)

We now invoke results in
 Puhalskii \cite{Puh16}.
Let the process $\hat\Psi_t=(\hat\Psi^t_s\,,s\in[0,1])$ be defined by 
\eqref{eq:37} with $\hat u(x)$ as $u(x)$\,.
Since $\hat u(x)$ is a bounded continuous function,
the random variables $N(\hat u(X^t_{ s}),
X_{ s}^t)$ are uniformly bounded.
Condition 2.2 in 
Puhalskii \cite{Puh16} is fulfilled because
 part 2 of condition (N) implies that
 the length of the projection of 
$N(\hat u(x),x)$ onto the nullspace  of $\sigma(x)$ is 
bounded away from zero  and,
consequently, 
 the quantity $\abs{N(\hat u(x),x)}^2-
N(\hat u(x),x)^T\sigma(x)(\sigma(x)\sigma(x)^T)^{-1}\sigma(x)^TN(\hat
u(x),x)$ is bounded away from zero.
Thus, Theorem 2.1 in Puhalskii \cite{Puh16} applies, so
 the pair
$(\hat\Psi^t,\mu^t)$ satisfies the Large Deviation Principle in 
$\mathbb C([0,1])\times \mathbb C_\uparrow([0,1],\mathbb
M_1(\R^l))$ for rate $t$\,, as $t\to\infty$\,,
 with the  deviation function in \eqref{eq:41''},
provided the  function 
$\Psi=(\Psi_s,\,s\in[0,1])$ 
 is absolutely continuous w.r.t.  Lebesgue measure on $\R_+$ and
the function $\mu=(\mu_s(\Gamma))$\,,
when considered as a measure on $[0,1]\times\R^l$\,, is
absolutely continuous w.r.t.  Lebesgue measure, i.e.,
 $\mu(ds,dx)=m_s(x)\,dx\,ds$\,, where
   $m_s(x)$\,,   as a function of $x$\,, belongs to 
$\hat{\mathbb{P}}$
  for almost all $s$\,.
If those conditions do not hold then 
$  \mathbf{ J}(\Psi,\mu)=\infty$\,. 
Since $L^{\hat \pi}_t=\hat\Psi^t_1$ and $\nu^t(\Gamma)=\mu^t([0,1],\Gamma)$\,,
by projection, the pair $(L^{\hat\pi}_t,\nu^t)$ obeys the Large Deviation
Principle in $\R\times \mathbb M_1(\R^l)$
 for rate $t$ with deviation function 
$\mathbf{I}^{\hat u}$\,,
such that 
$\mathbf{ I}^{\hat u}(L,\nu)=\inf\{ \mathbf{ J}(\Psi,\mu):\;
\Psi_1=L\,, \mu([0,1],\Gamma)=\nu(\Gamma)\}$\,. 
Therefore,
\begin{equation}
  \label{eq:48aa}
     \liminf_{t\to\infty}\frac{1}{t}\,\ln \mathbf{P}\bl(L^{\hat\pi}_t>
 q\br)\ge -\inf_{(L,\nu):\, L>q}
\mathbf{ I}^{\hat u}(L,\nu)\,.
\end{equation}
Calculations show that
\begin{equation*}
  \mathbf{ I}^{\hat u}(L,\nu)=\sup_{\lambda\in\mathbb R}
(\lambda L-\inf_{f\in\mathbb C_0^2}\int_{\R^l}
\overline H(x;\lambda,f,\hat u)\,\nu(dx))\,,
\end{equation*}
if  $\nu(dx)=m(x)\,dx$\,, where $m\in \hat{\mathbb{P}}$\,, and 
$\mathbf{ I}^{\hat u}(L,\nu)=\infty$\,, otherwise. 
By \eqref{eq:53}, the function $\lambda L-\inf_{f\in\mathbb C_0^2}\int_{\R^l}
\overline H(x;\lambda,f,\hat u)\,\hat m(x)\,dx$
 is  concave  in $\lambda$ and
is convex and lower semicontinuous in $L$\,.
It is     $\sup$--compact in $\lambda$
 because 
$\mathbf{ I}^{\hat u}(L,\nu)$ is a  deviation function, i.e., it is
$\inf$--compact. 
(We provide a direct proof of the latter property in the appendix.)
  Therefore,
by Theorem
7 on p.319 in Aubin and Ekeland \cite{AubEke84},
\begin{multline}
\inf_{(L,\nu):\, L>q}
\mathbf{ I}^{\hat u}(L,\nu)\le
\inf_{L>q}\sup_{\lambda\in\mathbb R}
(\lambda L-\inf_{f\in\mathbb C_0^2}\int_{\R^l}
\overline H(x;\lambda,f,\hat u)\,\hat m(x)\,dx)\\=
\sup_{\lambda\in\mathbb R}\inf_{L>q}
(\lambda L-\inf_{f\in\mathbb C_0^2}\int_{\R^l}
\overline H(x;\lambda,f,\hat u)\,\hat m(x)\,dx)=
\sup_{\lambda\ge0}
(\lambda q-\inf_{f\in\mathbb C_0^2}\int_{\R^l}
\overline H(x;\lambda,f,\hat u)\,\hat m(x)\,dx)\,.
  \label{eq:50}
\end{multline}
By integration by parts, if $f\in\mathbb C_0^2$\,, then, see \eqref{eq:53},
 \begin{multline}
   \label{eq:21}
   \int_{\R^l} \overline H(x;\lambda,f,v)\hat m(x)\,dx= 
\int_{\R^l}\bl(\lambda  M(v(x),x)
+\frac{1}{2}\,\abs{\lambda
N(v(x),x)+{\sigma(x)}^T\nabla f(x)}^2
+\nabla f(x)^T\,\theta(x)\\
-\frac{1}{2}\, \nabla f(x)^T
\frac{\text{div}\,\bl({\sigma(x)}{\sigma(x)}^T\hat m(x)\br)}{\hat
  m(x)}
\br)\hat m(x)\,dx\,.
 \end{multline}
As the righthand side depends on $f(x)$ through $\nabla f(x)$ only,
similarly to  developments above,  we use the righthand side of \eqref{eq:21}
   in order to define the lefthand side when 
$\nabla f\in\mathbb L^{1,2}_0(\R^l,\R^l,\hat m(x)\,dx)$\,.
By  the set of the gradients of $\mathbb C_0^2$--functions being dense
in $\mathbb L^{1,2}_0(\R^l,\R^l,\hat m(x)\,dx)$\,,
\begin{equation*}
    \inf_{f\in\mathbb C_0^2}\int_{\R^l}
\overline H(x;\lambda,f,\hat u)\,\hat m(x)\,dx=
\inf_{\nabla f\in\mathbb L^{1,2}_0(\R^l,\R^l,\hat m(x)\,dx)}\int_{\R^l}
\overline H(x;\lambda,f,\hat u)\,\hat m(x)\,dx\,.
\end{equation*}
Since $\overline H(x;1,f,\hat u)=H(x;1,f)$ (see \eqref{eq:96} and
\eqref{eq:25})\,, 
$\int_{\R^l}\overline H(x;1,f,\hat u)\hat m(x)\,dx=\breve G(1,\nabla 
f,\hat m)$\,. By $\nabla \hat f$ minimising $\breve G(1,\nabla 
f,\hat m)$ over $\nabla f\in \mathbb L^{1,2}_0(\R^l,\R^l,\hat
m(x)\,dx)$\,, 
 the function 
$q-\int_{\R^l}\overline H(x;1,f,\hat u)\hat m(x)\,dx$
attains  maximum over $\nabla f$
 in $\mathbb L_0^{1,2}(\R^l,\R^l,\hat m(x)\,dx)$ at 
$\nabla\hat f$\,. Therefore, the partial derivative
with respect to $\nabla f$ of $\lambda q-
\int_{\R^l}\overline H(x;\lambda,f,\hat u)\hat m(x)\,dx$
equals zero at $(1,\nabla\hat f)$\,.  
By  \eqref{eq:21}, we can write \eqref{eq:89} as
$d/d\lambda\,
\int_{\R^l}
\overline H(x;
\lambda,\hat  f,\hat u)\hat m(x)\,dx\Big|_{1}=q$\,, so,
 the partial 
derivative with respect to $\lambda$ of
$\lambda q-
\int_{\R^l}\overline H(x;\lambda,f,\hat u)\hat m(x)\,dx$
at $(1,\nabla\hat f)$  equals zero too.
The function $\lambda q-\int_{\R^l}\overline H(x;\lambda,f,\hat u)\hat
m(x)\,dx$ being concave in $(\lambda,\nabla f)$, 
it therefore attains a global maximum in $\R\times
\mathbb L^{1,2}_0(\R^l,\R^l,\hat m(x)\,dx)$
at  $(1,\nabla\hat f)$\,, cf. Proposition 1.2 on p.36 in Ekeland and
Temam \cite{EkeTem76}.
Hence,
\begin{equation*}
  \sup_{\lambda\ge0} \bl(
\lambda q-\inf_{f\in\mathbb{C}_0^2}
\int_{\R^l}\overline H(x;\lambda,f,\hat u)\hat m(x)\,dx\br)=
q-\breve G(1,\nabla\hat f,\hat m)\,.
\end{equation*}
The latter expression being equal to $J^{\text{o}}_q$\,,
 \eqref{eq:48aa}, and
\eqref{eq:50} imply the required lower bound \eqref{eq:27}.

\section{The proof of Theorem \ref{the:risk-sens}}
\label{sec:risk-sens-contr}
For the first assertion of part 1, let us assume that
$\lambda<\overline\lambda$\,. Let $\epsilon>0$ be such that
$\lambda(1+\epsilon)<\overline\lambda$\,. 
Let $ f_\epsilon$ represent the function $ f^{\lambda(1+\epsilon)}$\,.
By  \eqref{eq:40}, \eqref{eq:59}, \eqref{eq:29}, and \eqref{eq:110},
\begin{equation}
\label{eq:102}  
\limsup_{t\to\infty}\frac{1}{t}\,\ln
\mathbf{E}\exp((1+\epsilon)\lambda tL^{\pi}_t+f_\epsilon(X_t)-
f_\epsilon(X_0))\le
F((1+\epsilon)\lambda)
\,.
\end{equation}
By the reverse H\"older inequality,
\begin{equation*}
  \mathbf{E}\exp((1+\epsilon)\lambda tL^{\pi}_t+f_\epsilon(X_t)-
f_\epsilon(X_0))
\ge \bl(\mathbf{E}\exp(\lambda tL^{\pi}_t)\br)^{1+\epsilon}
\bl(\mathbf{E}\exp(-(1/\epsilon)(f_\epsilon(X_t)-
f_\epsilon(X_0))\br)^{-\epsilon}\,,
\end{equation*}
so, since $f_\epsilon$ is bounded below by an affine function 
and $\abs{X_0}$ is bounded, in analogy with the proof of \eqref{eq:60},
\begin{equation*}
  \limsup_{t\to\infty}\frac{1}{t}\,\ln\mathbf{E}\exp(\lambda tL^{\pi}_t)
\le F(\lambda)
\,.
\end{equation*}
The latter inequality is trivially true if $\lambda>\overline\lambda$\,.

We address the lower bound. Let $0<\lambda<\overline\lambda$\,. Then
  $F$ is subdifferentiable at $\lambda$\,.
Let  $q$ represent a subgradient of $F$ at $\lambda$\,.
Since $\lambda
q-F(\lambda)=J^{\text{o}}_q$\,,  by \eqref{eq:27},
 \begin{multline}
   \label{eq:66}
      \liminf_{t\to\infty}\frac{1}{t}\,
\ln         \mathbf{E}e^{\lambda t L^{\pi^\lambda}_t}\ge
\liminf_{t\to\infty}\frac{1}{t}\,
\ln         \mathbf{E}e^{\lambda t L^{\pi^\lambda}_t}
\chi_{\{L^{\pi^\lambda}_t\ge q\}}
\ge \lambda q+\liminf_{t\to\infty}\frac{1}{t}\,\ln\mathbf 
P(L^{\pi^\lambda}_t\ge
q)\\\ge\lambda
q-J^{\text{o}}_q=F(\lambda)\,.
\end{multline}
If $\lambda=\overline\lambda$ and 
   $F$ is  subdifferentiable at $\overline\lambda$\,, a similar proof
   applies.
Suppose  that  $\lambda=\overline\lambda$ and
  $F$ is not subdifferentiable at $\overline\lambda$\,.
By what has been just proved,
\begin{equation*}
  \liminf_{\check\lambda\uparrow\overline\lambda}
   \liminf_{t\to\infty}\frac{1}{t}\,
\ln         \mathbf{E}e^{\check\lambda t L^{\pi^{\check\lambda}}_t}\ge
\liminf_{\check\lambda\uparrow\overline\lambda}F(\check\lambda)=F(\overline\lambda)
\end{equation*}
and H\"older's inequality yields 
\begin{equation*}
  \liminf_{\check\lambda\uparrow\overline\lambda}
   \liminf_{t\to\infty}\frac{1}{t}\,
\ln         \mathbf{E}e^{\overline\lambda t L^{\pi^{\check\lambda}}_t}
\ge F(\overline\lambda)\,.
\end{equation*}
By requiring $\pi^{\overline\lambda}_t$ to match  $\pi^{\lambda_i}_t$ on certain
intervals $[t_i,t_{i+1})$ where $\lambda_i\uparrow\overline\lambda$
and $t_i\to\infty$  appropriately, we can ensure that
$     \liminf_{t\to\infty}(1/t)\,
\ln         \mathbf{E}e^{\overline\lambda t L^{\pi^{\overline\lambda}}_t}\ge
F(\overline\lambda)
$\,.

Suppose that $\lambda>\overline\lambda$\,. 
If $F$ is subdifferentiable at $\overline\lambda$\,, then, similarly
to \eqref{eq:66}, on choosing $q$ as a subgradient of $F$ at
$\overline\lambda$\,, 
\begin{equation}
  \label{eq:68}
         \liminf_{t\to\infty}\frac{1}{t}\,
\ln         \mathbf{E}e^{\lambda t L^{\pi^{\overline\lambda}}_t}
\ge \lambda q+\liminf_{t\to\infty}\frac{1}{t}\,\ln\mathbf 
P(L^{\pi^{\overline\lambda}}_t\ge
q)\ge\lambda
q-J^{\text{o}}_q=(\lambda-\overline\lambda)q+F(\overline\lambda)\,.
\end{equation}
Since $q$ can be chosen arbitrarily great, 
$     \lim_{t\to\infty}(1/t)\,
\ln         \mathbf{E}e^{\lambda t L^{\pi^{\overline\lambda}}_t}=\infty\,.$
If $F$ is not
subdifferentiable at $\overline\lambda$\,, then we pick $\lambda_i$
and $q_i$
such that $\lambda_i\uparrow \overline\lambda$\,, $q_i$ is a subgradient of
$F$ at $\lambda_i$ and $q_i\uparrow\infty$\,.
Arguing along the lines of \eqref{eq:68} yields
\begin{equation*}
  \liminf_{t\to\infty}\frac{1}{t}\,
\ln         \mathbf{E}e^{\lambda t L^{\pi^{\lambda_i}}_t}
\ge \lambda q_i+\liminf_{t\to\infty}\frac{1}{t}\,\ln\mathbf 
P(L^{\pi^{\lambda_i}}_t\ge q_i)\ge
(\lambda-\overline\lambda)q_i+F(\lambda_i)\,,
\end{equation*}
so there exists $\pi^\lambda$ such that 
$\lim_{t\to\infty}(1/t)\ln \mathbf Ee^{\lambda t L^{\pi^\lambda}_t}=\infty$\,.

We prove now part 2. 
  Since $\mathbf Ee^{\lambda t L^\pi_t}\ge e^{\lambda q t}\mathbf
  P(L^\pi_t\le q)$ provided $\lambda<0$\,, 
the inequality in \eqref{eq:58}  of Theorem \ref{the:bounds} implies that 
  \begin{equation*}
      \liminf_{t\to\infty}\frac{1}{t}\,
\ln \mathbf Ee^{\lambda t L^\pi_t}\ge \sup_{q\in\R}(\lambda q
-J^{\text{s}}_q)
=F(\lambda)\,,
  \end{equation*}
with the latter equality holding because by \eqref{eq:36} 
$J^{\text{s}}_q$ is the Legendre--Fenchel transform of the function
 that equals $F(\lambda)$ when $\lambda\le0$ and
equals $\infty$\,, otherwise.

Since $\lambda<0$\,,  $F$ is differentiable at $\lambda$\,, so
$\pi^\lambda$ is well defined. Let $u^\lambda(x)$ be such that
$\pi^\lambda_t=u^\lambda(X_t)$\,, i.e., $u^\lambda(x)$ is defined as
$\hat u(x)$ when $q =F'(\lambda)$\,.
By \eqref{eq:1}, assuming that $f\in\mathcal{A}_\kappa$\,,
\begin{equation*}
    \mathbf{E}\exp\bl(
 \lambda t L^{\pi^{\lambda,\rho}}_t+f(X_t)- f(X_0)-
t\int_{\R^l}\overline H(x;\lambda,f, u^{\lambda,\rho})\,\nu^t(dx)
\br)\le 1\,.
\end{equation*}
By  Lemma \ref{le:saddle_2}, recalling that $\abs{X_0}$ is bounded,
\begin{equation*}
\limsup_{t\to\infty}
\frac{1}{t}\,\ln  \mathbf{E}\exp(\lambda tL^{\pi^{\lambda,\rho}}_t)
\le\inf_{f\in\mathbb C^2} \sup_{x\in\R^l}
\overline H(x;\lambda,f,u^{\lambda,\rho})\,.
\end{equation*}
We now apply condition \eqref{eq:97}.
\appendix
\section{The scalar case}
We will assume that $l=n=1$\,, so,  in \eqref{eq:85}--\eqref{eq:85c}, 
$\Theta_1$\,, $\theta_2$\,, $A_1$, $a_2$\,, $r_1$\,, $r_2$\,,
$\alpha_1$, and $\alpha_2$ are scalars,  $\Theta_1<0$\,, 
$\sigma$ is a $1\times k$--matrix, $b$ is a
$1\times k$--matrix, and  $\beta$ is a $k$--vector.
Accordingly, $c$\,, $\sigma\sigma^T$\,, $\sigma b^T$\,, $P_1(\lambda)$\,, 
$p_2(\lambda)$\,, $A(\lambda)$\,, $B(\lambda)$\,, and $C$  are scalars.
The equation for $P_1(\lambda)$ is 
\begin{equation}
  \label{eq:106}
      B(\lambda)P_1(\lambda)^2+2A(\lambda)P_1(\lambda)+\frac{\lambda}{1-\lambda}\,
  C=0\,. 
\end{equation}
Let
\begin{equation}
  \label{eq:23}
  \tilde\beta=1+\frac{1}{\Theta_1^2}\,\frac{A_1-r_1}{c}\,
\bl(\sigma\sigma^T(A_1-r_1)-2\Theta_1\sigma b^T\br)\,.
\end{equation}
(The latter piece of notation is modelled on that of Pham \cite{Pha03}.)
We have that
\begin{equation*}
  A(\lambda)^2-B(\lambda)\,\frac{\lambda}{1-\lambda}\,C=
\Theta_1^2\,\frac{1-\lambda\tilde\beta}{1-\lambda}\,.
\end{equation*}
Hence, $P_1(\lambda)$ exists if and only if
  \begin{equation*}
    \lambda\le\frac{1}{\tilde\beta}\wedge 1\,,
  \end{equation*}
so,  $\tilde\lambda=\min(1/\tilde\beta,1)$\,.
(Not unexpectedly, if $\lambda<0$ then \eqref{eq:106} has both a
positive and a negative root, whereas both roots are positive if
$0<\lambda\le\tilde\lambda$\,.) If $\lambda<\tilde\lambda$\,, then
\begin{equation}
  \label{eq:16}
   P_1(\lambda)=\frac{1}{B(\lambda)}\,
\bl(-A(\lambda)-
\abs{\Theta_1}\sqrt{\dfrac{1-\lambda\tilde\beta}{1-\lambda}}\br)
\end{equation}
and $F(\lambda)$ is determined by \eqref{eq:79} and \eqref{eq:155}.
The minus sign in front of the square root is chosen because
$D(\lambda)=A(\lambda)+B(\lambda)P_1(\lambda)$ has to be negative
which is needed in order for the analogue of
 \eqref{eq:75} to have  a stationary
distribution.
Therefore,
\begin{equation}
  \label{eq:22}
  D(\lambda)=
\Theta_1\sqrt{\dfrac{1-\lambda\tilde\beta}{1-\lambda}}\,.
\end{equation}
The functions $D(\lambda)$ and $P_1(\lambda)$ are differentiable for
$\lambda<1\wedge (1/\tilde\beta)$\,. As in Pham \cite{Pha03}, we distinguish between three
cases:
$\tilde\beta>1$\,, $\tilde\beta<1$\,, and $\tilde\beta=1$\,.

Suppose that $\tilde\beta>1$\, so, $\tilde\lambda=1/\tilde\beta$\,.
 Then $P_1(\lambda)$ and $D(\lambda)$ are continuous on
 $[0,1/\tilde\beta]$ and differentiable on
$(0,1/\tilde\beta)$\,. We have that  $P_1(1/\tilde\beta)=
-A(1/\tilde\beta)/B(1/\tilde\beta)$
and $D(1/\tilde\beta)=0$\,.
Also,
$D(\lambda)/\sqrt{1/\tilde\beta-\lambda}\to-\abs{\Theta_1}\sqrt{\tilde\beta}/\sqrt{1-1/\tilde\beta}$ and
 $(P_1(1/\tilde\beta)-P_1(\lambda))/\sqrt{1/\tilde\beta-\lambda}
\to \abs{\Theta_1}\sqrt{\tilde\beta}/(B(1/\tilde\beta)\sqrt{1-1/\tilde\beta})$\,, as
$\lambda\uparrow 1/\tilde\beta$\,.
In addition, by \eqref{eq:79} and \eqref{eq:155}, if
$E(1/\tilde\beta)\not=0$\,,
then 
$\abs{p_2(\lambda)}=\abs{E(\lambda)/D(\lambda)}\to\infty$ and $F(\lambda)\to\infty$\,, so,
$F(\lambda)=\infty$ when $\lambda\ge1/\tilde\beta$\,,
$\overline\lambda=1/\tilde\beta$\,, and $\hat\lambda<\overline\lambda$\,.  
Suppose that $E(1/\tilde\beta)=0$\,. By \eqref{eq:79} and \eqref{eq:139a},
$E(\lambda)=D(\lambda)Z(\lambda)+U(\lambda)$\,, where
\begin{equation*}
Z(\lambda)=
  \frac{ \lambda}{1-\lambda}\,
b\sigma^T c^{-1}(a_2-r_2-\lambda
b\beta)-\lambda\sigma\beta+\theta_2
\end{equation*}
and \begin{equation*}
  U(\lambda)=\frac{\lambda}{1-\lambda}\,(A_1-r_1)c^{-1}(a_2-r_2-\lambda
  b\beta)
+\lambda(r_1-\alpha_1)-\frac{A(\lambda)}{B(\lambda)}\,
Z(\lambda)\,.
\end{equation*}
Therefore,
\begin{equation*}
  p_2(\lambda)=-\frac{Z(\lambda)}{B(\lambda)}\,
-\,\frac{U(\lambda)}{D(\lambda)}\,,
\end{equation*}
Since $E(1/\tilde\beta)=D(1/\tilde\beta)=0$\,,
 $U(1/\tilde\beta)=0$\,.
By $U(\lambda)$ being linear in a neighbourhood of $1/\tilde\beta$\,,
 $p_2(\lambda)$ is continuous at $1/\tilde\beta$\,,
$p_2(1/\tilde\beta)=-Z(1/\tilde\beta)/B(1/\tilde\beta)$\,, and
$F(1/\tilde\beta)$ is finite. Let us look at the derivative at $1/\tilde\beta$\,.
We have that  
$(p_2(1/\tilde\beta)-p_2(\lambda))/\sqrt{1/\tilde\beta-\lambda}
\to U'(1/\tilde\beta)\sqrt{1-1/\tilde\beta}/(\Theta_1\sqrt{\tilde\beta})$\,,
as $\lambda\uparrow1/\tilde\beta$\,.
By \eqref{eq:155}, $(F(1/\tilde\beta)-F(\lambda))/\sqrt{1/\tilde\beta-\lambda}
\to (1/2)\,\sigma\sigma^T
\abs{\Theta_1}\sqrt{\tilde\beta}/(B(1/\tilde\beta)\,
\sqrt{1-1/\tilde\beta})$\,.
Therefore, $F'(1/\tilde\beta-)=\infty$\,, so, $\overline\lambda=1/\tilde\beta$
and $\hat\lambda<\overline\lambda$\,.

Suppose that $\tilde\beta<1$\,. By \eqref{eq:23}, $b\sigma^T\not=0$\,.
Also, $\tilde\lambda=\overline\lambda=1$\,. 
By  \eqref{eq:16} and \eqref{eq:22}, $P_1(\lambda)$  has limit $P_1(1)$  when
$\lambda\uparrow 1$ and $(P_1(\lambda)-P_1(1))/\sqrt{1-\lambda}\to
\Theta_1\sqrt{1-\tilde\beta}/((b\sigma^T)^2c^{-1})$ as $\lambda\uparrow1$\,.
In fact, $P_1(1)=-(A_1-r_1)/(b\sigma^T)$\,.
By \eqref{eq:22}, \eqref{eq:79}, \eqref{eq:139a}, and \eqref{eq:155},
 $p_2(\lambda)\to -(a_2-r_2-b\beta)/b\sigma^T$\,, as
 $\lambda\uparrow1$\,,  which quantity we denote
by $p_2(1)$\,.
By \eqref{eq:79} and \eqref{eq:139a},
on noting that $A_1-r_1+b\sigma^TP_1(1)=0$\,.
\begin{equation}
\label{eq:28}\lim_{\lambda\uparrow1}  \frac{p_2(1)-p_2(\lambda))}{\sqrt{1-\lambda}}
=K_1\,,
\end{equation}
where
\begin{equation*}
    K_1=\frac{1}{\Theta_1\sqrt{1-\tilde\beta}}\,\bl(\bl(\Theta_1 
  -\frac{\sigma\sigma^T (A_1-r_1)}{b\sigma^T}\br)\,p_2(1)+r_1-\alpha_1
+P_1(1)(\theta_2-\sigma\beta)\br)\,.
\end{equation*}
Since $a_2-r_2-b\beta+b\sigma^Tp_2(1)=0$\,,
\begin{equation*}
  \lim_{\lambda\uparrow1}\frac{a_2-r_2-\lambda
b\beta+b\sigma^Tp_2(\lambda)}{\sqrt{1-\lambda}}
=\lim_{\lambda\uparrow1}\frac{b\sigma^T(p_2(\lambda)-p_2(1))}{\sqrt{1-\lambda}}
=b\sigma^T K_1\,.
\end{equation*}
By \eqref{eq:139a}, $F(1-)<\infty$\,.
Let us look at the derivative $F'(1-)$\,. 
One needs to improve on \eqref{eq:28}. More specifically,
by \eqref{eq:79}, \eqref{eq:139a}, \eqref{eq:16} and \eqref{eq:22},
one can expand as follows (either by hand or by the use of Mathematica): 
as $\lambda\uparrow1$\,,
\begin{equation*}
  p_2(\lambda)=p_2(1)-K_1 \sqrt{1-\lambda}-K_2 (1-\lambda)+o(1-\lambda)\,,
\end{equation*}
where
\begin{equation*}
K_2=\frac{\sigma\sigma^T}{(b\sigma^T)^2c^{-1}}\,p_2(1)
+\frac{b\beta}{b\sigma^T}\,+\frac{\theta_2-\sigma\beta}{(b\sigma^T)^2c^{-1}}\,.
\end{equation*}
By \eqref{eq:155},
\begin{equation*}
  \lim_{\lambda\uparrow1}\frac{F(\lambda)-F(1)}{\sqrt{1-\lambda}}
=-\sigma\sigma^Tp_2(1)K_1-b\sigma^T K_1(b\beta-b\sigma^TK_2)c^{-1}+
(\sigma\beta-\theta_2)K_1+
\frac{1}{2}\,\sigma\sigma^T\,\frac{\Theta_1\sqrt{1-\tilde\beta}}{(b\sigma^T)^2c^{-1}}\,,
\end{equation*}
which simplifies to
\begin{equation*}
  \lim_{\lambda\uparrow1}\frac{F(1)-F(\lambda)}{\sqrt{1-\lambda}}=
\frac{\abs{\Theta_1}\sqrt{1-\tilde\beta}\,\sigma\sigma^T}{2(b\sigma^T)^2c^{-1}}\,,
\end{equation*}
implying that $F'(1-)=\infty$\,, so, $\hat\lambda<\overline\lambda$\,.

Let us consider the  case that $\tilde\beta=1$\,, 
so, 
$(A_1-r_1)
\bl(\sigma\sigma^T(A_1-r_1)-2\Theta_1\sigma b^T\br)=0\,.
$ One has that $\tilde\lambda=\overline\lambda=1$\,,
$  D(\lambda)=\Theta_1$\,,
 $ P_1(\lambda)=
(-\sigma
  b^Tc^{-1}(A_1-r_1))/\bl((1-\lambda)/\lambda\,\sigma\sigma^T 
+\sigma b^Tc^{-1}b\sigma^T\br)$
and $  p_2(\lambda)=-E(\lambda)/\Theta_1\,.$
Thus,  if $b\sigma^T=0$\,, then $A_1-r_1=0$ and  $P_1(\lambda)=0$\,.
If $b\sigma^T\not=0$\,, then
  $  P_1(1)=
    -(A_1-r_1)/(b\sigma^T)$\,,
  $P_1'(1)=      -\sigma\sigma^T(A_1-r_1)/((b\sigma^T)^3c^{-1})$\,,
  and
  $  P_1''(1)=      2\sigma\sigma^T(A_1-r_1)/\bl((b\sigma^T)^3c^{-1}\br)
\bl(1-\sigma\sigma^T/\bl((b\sigma^T)^2c^{-1}\br)\br)$\,.
Since  
\begin{equation}
  \label{eq:54}
A_1-r_1+b\sigma^TP_1(1)=0\,,  
\end{equation}
  $E(\lambda)$ is
continuous on $[0,1]$ and is
differentiable on $(0,1)$\,, see \eqref{eq:139a}, so is
$p_2(\lambda)$\,. 
By \eqref{eq:155}, 
if $a_2-r_2-b\beta+b\sigma^Tp_2(1)\not=0$\,, then
$F(\lambda)\to\infty$\,, as $\lambda\to\infty$\,, so
$\hat\lambda<\overline\lambda$\,.
If 
\begin{equation}
  \label{eq:57}
a_2-r_2-b\beta+b\sigma^Tp_2(1)=0\,,  
\end{equation}
 then 
\begin{equation*}
      F(1)=\frac{1}{2}\,
\sigma\sigma^T p_2(1)^2+
(-\sigma\beta+\theta_2)p_2(1)
+r_2-\alpha_2+\abs{\beta}^2+\frac{1}{2}\,\sigma\sigma^TP_1(1)
\end{equation*}
and
  \begin{multline*}
      F'(1-)=
\sigma\sigma^Tp_2'(1-) p_2(1)
+\frac{1}{2c}\,(b\sigma^T p_2'(1-)-b\beta)^2-\beta^T\sigma^Tp_2(1)+
(-\sigma\beta+\theta_2)p_2'(1-)\\
+r_2-\alpha_2+\frac{3}{2}\,\abs{\beta}^2+
\frac{1}{2}\,\sigma\sigma^TP_1'(1-)\,.
\end{multline*}
As one can see, $F(\lambda)$ is not essentially smooth.
We obtain that $\hat\lambda<\overline\lambda$ if and only if
$F'(1-)>q$\,, otherwise $\hat\lambda=1$\,.
It is noteworthy that \eqref{eq:54} and \eqref{eq:57} represent
conditions \eqref{eq:130} and \eqref{eq:133}, respectively.

The cases  where $\tilde\beta\ge 1$ and
 $F(\lambda)\to \infty$ as $\lambda\uparrow
1/\tilde\beta$ and where $\tilde\beta<1$ have been 
analysed by Pham \cite{Pha03}. 
\section{Proof  of Lemma \ref{le:angle}}
Suppose that  the matrix $\sigma(x)Q_1(x)\sigma(x)^T$ is uniformly
  positive definite. Then $\abs{Q_1(x)\sigma(x)^Ty}\ge k_1\abs{y}$\,, for some
  $k_1>0$\,, all $x\in\R^l$ and all $y\in\R^k$\,.
Since
$\abs{\sigma(x)^Ty}^2=y^T\sigma(x)\sigma(x)^Ty\le k_2\abs{y}^2$\,, for
some $k_2\ge k_1$\,,
we have that
\begin{equation*}
  \frac{\abs{(I_k-Q_1(x))\sigma(x)^Ty}}{\abs{\sigma(x)^Ty}}
\le\frac{\sqrt{\abs{\sigma(x)^Ty}^2-
k_1^2\abs{y}^2}}{\abs{\sigma(x)^Ty}}
\le\sqrt{1-
\frac{k_1^2}{k_2^2}}\,.
\end{equation*}
Therefore, since $I_k-Q_1(x)$ is the operator of the orthogonal
projection on the range of $b(x)^T$\,, given $z\in \R^n$\,,
\begin{equation*}
  (\sigma(x)^Ty)^Tb(x)^Tz\le 
\sqrt{1-
\frac{k_1^2}{k_2^2}}\,\abs{\sigma(x)^Ty}\abs{b(x)^Tz}\,,
\end{equation*}
so nonzero vectors from the ranges of $\sigma(x)^T$ and of $b(x)^T$ are at
angles uniformly bounded away from zero.
Conversely, if $(\sigma(x)^Ty)^Tb(x)^Tz\le 
\rho_1\,\abs{\sigma(x)^Ty}\abs{b(x)^Tz}$, for some $\rho_1\in(0,1)$\,,
then
$\abs{(I_k-Q_1(x))\sigma(x)^Ty}\le \rho_1 {\abs{\sigma(x)^Ty}}$ so that
$\abs{Q_1(x)\sigma(x)^Ty}=\sqrt{\abs{\sigma(x)^Ty}^2-
\abs{(I_k-Q_1(x))\sigma(x)^Ty}^2}\ge(1-\rho_1)\abs{\sigma(x)^Ty}\ge
(1-\rho_1)\rho_2\abs{y} $\,, the latter inequality holding by
$\sigma(x)\sigma(x)^T$ being uniformly positive definite, where $\rho_2>0$\,.
Thus, the matrix 
$\sigma(x)Q_1(x)\sigma(x)^T$ is uniformly positive definite if
and only if ''the angle condition'' holds. Since the angle condition is
symmetric in $\sigma(x)$ and $b(x)$\,, it is also equivalent to
 the matrix
$c(x)-b(x)\sigma(x)^T(\sigma(x)\sigma(x)^T)^{-1}\sigma(x)  b(x)^T$
being uniformly positive definite.

In order to prove the second assertion of the lemma, let us observe
that 
\begin{equation*}
\beta(x)^T  Q_2(x)\beta(x)=\beta(x)^TQ_1(x)
\bl(I_k-Q_1(x)\sigma(x)^T(\sigma(x)Q_1(x)Q_1(x)\sigma(x)^T)^{-1}\sigma(x)
Q_1(x)\br)Q_1(x)\beta(x)\,,
\end{equation*}
so, if $\beta(x)^T  Q_2(x)\beta(x)$ is bounded away from zero, then,
by $\abs{\beta(x)Q_1(x)}$ being bounded,
there exists $\rho_3\in(0,1)$ such that, for all $x\in\R^l$\,,
\begin{equation*}
  (1-\rho_3)\abs{Q_1(x)\beta(x)}>
\bl(Q_1(x)\sigma(x)^T(\sigma(x)Q_1(x)Q_1(x)\sigma(x)^T)^{-1}\sigma(x)
Q_1(x)\br)Q_1(x)\beta(x)\,.
\end{equation*}
The righthand side representing the orthogonal projection of
$Q_1(x)\beta(x)$ onto the range of $(\sigma(x)Q_1(x))^T$ implies that,
given $y\in \R^l$\,,
\begin{equation*}
  \abs{(Q_1(x)\beta(x))^TQ_1(x)\sigma(x)^Ty}\le
  \rho_3\abs{Q_1(x)\beta(x)}
\abs{Q_1(x)\sigma(x)^Ty}\,,
\end{equation*}
which means that $Q_1(x)\beta(x)$ is at  angles
to $Q_1(x)\sigma(x)^Ty$ which are  bounded below
uniformly over $y$\,. The converse is proved similarly.

\section{Proof of Lemma \ref{le:condition}}
By Lemma \ref{le:saddle_2},
\begin{equation}
  \label{eq:107}
\inf_{f\in\mathbb C^2}\sup_{x\in\R^l}\overline H(x;\hat\lambda,f,\hat u^\rho)=
  \sup_{\nu\in\mathcal{P}}\inf_{f\in\mathbb C^2_0}
\int_{\R^l}\overline H(x;\hat\lambda,f,\hat u^\rho)\nu(dx)\,.
\end{equation}
For function $f$ and $\rho>0$\,, we denote
$f(x)^\rho=f(x)\chi_{[0,
  \rho]}(\abs{x})$\,.
By  \eqref{eq:4},  \eqref{eq:8}, \eqref{eq:69}, and \eqref{eq:61}, 
\begin{multline}
  \label{eq:34a}
  \overline H(x;\hat\lambda,f,\hat u^\rho)\\=
-\,\frac{\hat\lambda}{2(1-\hat\lambda)}\,\bl(\norm{b(x)\sigma(x)^T\nabla\hat
  f(x)^\rho}^2_{c(x)^{-1}}
-\norm{(a(x)-r(x)\mathbf1)^\rho}^2
_{c(x)^{-1}}\br)
\\+\hat\lambda(r(x)-\alpha(x)
+\frac{1}{2}\,\abs{\beta(x)}^2)
+\frac{\hat\lambda}{2(1-\hat\lambda)}\,
\norm{\hat\lambda b(x)\beta(x)^\rho}^2_{c(x)^{-1}}
\\
+
\frac{\hat\lambda}{1-\hat\lambda}\,\Bl(-\bl(\bl(a(x)-r(x)\mathbf1
\br)^\rho\br)^Tc(x)^{-1}b(x)
\hat\lambda\beta(x)
\\+
\bl(\bl(a(x)-r(x)\mathbf1
-\hat\lambda b(x)\beta(x)+b(x)\sigma(x)^T\nabla\hat
  f(x)\br)^\rho\br)^Tc(x)^{-1}b(x){\sigma(x)}^T\nabla f(x)\Br)\\
+\frac{1}{2}\,\abs{-\hat\lambda\beta(x)+{\sigma(x)}^T\nabla f(x)}^2
+\nabla f(x)^T\,\theta(x)
+\frac{1}{2}\, \text{tr}\,\bl({\sigma(x)}{\sigma(x)}^T\nabla^2
f(x)\br)\,.
\end{multline}
As in the proof of Lemma \ref{le:sup-comp},
it follows that, under the hypotheses,  there exist
$\overline\kappa>0$\,, $\overline K_1>0$ and $\overline K_2>0$
such that
$ \overline H(x;\hat\lambda,f_{\overline\kappa},\hat u^\rho)\le 
\overline K_1-\overline K_2\abs{x}^2$\,, for all
$x\in\R^l$ and all $\rho>0$\,.
Consequently,  $ \inf_{f\in\mathbb C_0^2}\int_{\R^l}\overline
H(x;\hat\lambda,f,\hat u^\rho)\nu(dx)$ is a $\sup$--compact function
of $\nu\in\mathcal{P}$\,, so,
  the supremum over $\nu$ on the righthand side of \eqref{eq:107} is attained at
some $\nu_\rho$\,. Moreover, if the $\limsup$ on the lefthand side of
\eqref{eq:97} is greater than $-\infty$\,, then
\begin{equation}
  \label{eq:25a}
  \limsup_{\rho\to\infty}\int_{\R^l}\abs{x}^2\nu_\rho(dx)<\infty\,,
\end{equation}
so, the $\nu_\rho$ make up a 
relatively compact subset of $\mathcal{P}$\,.

If \eqref{eq:31} holds, then,
  given $\tilde f\in\mathbb C_0^2$\,, by \eqref{eq:34a},
there exist $\tilde C_1$ and $\tilde C_2$\,, such that, for all
$x\in\R^l$ and all $\rho>0$\,,
\begin{equation}
  \label{eq:35a}\overline H(x;\hat\lambda,\tilde f,\hat u^\rho)\le
\tilde C_1\abs{x}+\tilde C_2\,.
\end{equation}
Assuming   that $\nu_\rho\to\tilde \nu$\,, we have,
by the convergence  $\overline H(x_\rho;\hat\lambda,\tilde f,\hat
u^\rho)
\to \overline H(\tilde x;\hat\lambda,\tilde f,\hat u)$ when $x_\rho\to\tilde x$\,, by
\eqref{eq:25a}, \eqref{eq:35a}, the definition of the topology on 
$\mathcal{P}$\,, Fatou's lemma, and the dominated convergence theorem, that
\begin{equation*}
  \limsup_{\rho\to\infty}
\int_{\R^l}\overline H(x;\hat\lambda,\tilde f,\hat u^\rho)\nu_\rho(dx)
\le\int_{\R^l}\overline H(x;\hat\lambda,\tilde f,\hat u)\tilde\nu(dx)\,,
\end{equation*}
so, on recalling \eqref{eq:47},
\begin{equation*}
    \limsup_{\rho\to\infty}
\inf_{f\in\mathbb C_0^2}
\int_{\R^l}\overline H(x;\hat\lambda,f,\hat u^\rho)\nu_\rho(dx)
\le \inf_{f\in\mathbb C_0^2}\int_{\R^l}\overline H(x;\hat\lambda,f,\hat
u)\tilde\nu(dx)\le F(\hat\lambda)\,.
\end{equation*}

\section{  }
\label{sec:B}
\begin{lemma}
  \label{le:dopoln}
Given $L\in\R$\,,
  $m\in\hat{\mathbb P}$\,, and $v\in \mathbb L^2(\R^l,\R^n,m(x)\,dx)$\,, 
the sets 
\begin{equation*}
  \{\lambda\in\R:\,\lambda L-\inf_{f\in\mathbb C_0^2}\int_{\R^l}
\overline H(x;\lambda,f,v)\, m(x)\,dx\ge \alpha\}
\end{equation*}
are compact for all $\alpha\in\R$\,.
\end{lemma}
\begin{proof}
By \eqref{eq:53},
  \begin{multline*}
    \inf_{f\in\mathbb C_0^2}\int_{\R^l}
\overline H(x;\lambda,f,v)\, m(x)\,dx=
\inf_{f\in\mathbb L^{1,2}_0(\R^l,\R^l,m(x)\,dx)}\int_{\R^l}
\bl(\lambda  M(v(x),x)\\
+\frac{1}{2}\,\abs{\lambda
N(v(x),x)+{\sigma(x)}^T\nabla f(x)}^2
+\nabla f(x)^T\,\theta(x)
-\frac{1}{2}\,\nabla f(x)^T
\frac{
\text{div}({\sigma(x)}{\sigma(x)}^Tm(x))}{m(x)}\br)
\, m(x)\,dx\,.
  \end{multline*}
The infimum is attained at 
\begin{equation*}
  \nabla f(x)=\lambda g_1(x)+g_2(x)\,,
\end{equation*}
where
\begin{align*}
  g_1&=-\Pi\bl((\sigma(\cdot)\sigma(\cdot)^T)^{-1}
\sigma(\cdot)^TN(v(\cdot),\cdot)\br),\\
g_2&=\Pi\bl((\sigma(\cdot)\sigma(\cdot)^T)^{-1}
\bl(-\theta(\cdot)+
\frac{\text{div}({\sigma(\cdot)}{\sigma(\cdot)}^Tm(\cdot))}{2m(\cdot)}
\br)\br)\,,
\end{align*}
with $\Pi$ representing the operator of the orthogonal projection on
$\mathbb L^{1,2}_0(\R^l,\R^l,m(x)\,dx)$ in $\mathbb
L^2(\R^l,\R^l,m(x)\,dx)$ with respect to  the inner product
$\langle h_1,h_2\rangle
=\int_{\R^l}h_1(x)^T\sigma(x)\sigma(x)^Th_2(x)\,m(x)\,dx$\,. 
Therefore,
\begin{multline}
  \label{eq:114}
  \lambda L-\inf_{f\in\mathbb C_0^2}\int_{\R^l}
\overline H(x;\lambda,f,v)\, m(x)\,dx\\=
\lambda\bl(L-\int_{\R^l}  M(v(x),x)m(x)\,dx-
\int_{\R^l}g_1(x)^T\sigma(x)\sigma(x)^Tg_2(x)m(x)\,dx\br)
\\+\frac{1}{2}\,\int_{\R^l}g_2(x)^T\sigma(x)\sigma(x)^Tg_2(x)m(x)\,dx
-\frac{\lambda^2}{2}\,
\int_{\R^l}\bl(\abs{N(v(x),x)}^2-g_1(x)^T\sigma(x)\sigma(x)^Tg_1(x)\br)
m(x)\,dx\,.
\end{multline}
Since projection is a contraction operator,
\begin{equation*}
  \int_{\R^l}g_1(x)^T\sigma(x)\sigma(x)^Tg_1(x)m(x)\,dx
\le \int_{\R^l}N(v(x),x)^T\sigma(x)^T(\sigma(x)\sigma(x)^T)^{-1}
\sigma(x)N(v(x),x)m(x)\,dx\,.
\end{equation*}
As mentioned, by condition (N), $\beta(x)$ does not belong to the sum
of the ranges of $b(x)^T$ and of $\sigma(x)^T$\,. By \eqref{eq:8},
$N(u,x)$ does not belong to the range of $\sigma(x)^T$\,, for any $u$
and $x$\,. Therefore, the projection of $N(v(x),x)$ onto the null
space of $\sigma(x)$ is nonzero which implies that 
 $\abs{N(v(x),x)}^2-
N(v(x),x)^T\sigma(x)^T(\sigma(x)\sigma(x)^T)^{-1}
\sigma(x)N(v(x),x)$ is positive for any $x$\,,
so, the coefficient
of $\lambda^2$ on the righthand side of  \eqref{eq:114} is positive,
yielding the needed property.
\end{proof}
The next result seems to be ''well known''. We haven't been
able to find a reference, though.
\begin{lemma}
  \label{le:exp_moment}
For arbitrary $\kappa>0$\,,
\begin{equation*}
  \limsup_{t\to\infty}\mathbf Ee^{\kappa\abs{X_t}}<\infty\,.
\end{equation*}
\end{lemma}
\begin{proof}
We prove that, if $\gamma>0$ and is small enough, then 
\begin{equation*}
  \limsup_{t\to\infty}\mathbf Ee^{\gamma\abs{X_t}^2}<\infty\,.
\end{equation*}
By (2.2),
 there exist $K_1>0$ and $K_2>0$ such that, for all $x\in\R^l$\,,
$\theta(x)^Tx\le -K_1\abs{x}^2+K_2$\,.
On applying It\^o's lemma to  (2.1) and recalling that 
$\sigma(x)\sigma(x)^T$ is bounded, we have that,
 for some $K_3>0$ and all $i\in\mathbb N$\,,
\begin{equation*}
  d\mathbf E\abs{X_t}^{2i}\le- 2i K_1\mathbf E
\abs{X_t}^{2i}\,dt+
2i^2K_3\mathbf E \abs{X_t}^{2i-2}\,dt\,.
\end{equation*} 
Hence,
\begin{equation*}
  \mathbf E\abs{X_t}^{2i}\le \mathbf E\abs{X_0}^{2i}e^{-2iK_1t}
+2i^2K_3e^{-2iK_1t}\int_0^t e^{2iK_1s}
\mathbf E \abs{X_s}^{2i-2}\,ds\,.
\end{equation*}
Let
\begin{equation*}
  M_i(t)=\frac{1}{i!}\,\sup_{s\le t}\mathbf E\abs{X_s}^{2i}\,.
\end{equation*}
We have that 
\begin{equation*}
M_i(t)\le\frac{1}{i!}\, \mathbf E\abs{X_0}^{2i}
+\frac{K_3}{K_1}M_{i-1}(t)\,.
\end{equation*}
Hence, if $\gamma K_3/K_1<1$\,, then
\begin{equation*}
  \sum_{i=0}^\infty\gamma^iM_i(t)\le \frac{1}{1-\gamma K_3/K_1}
\sum_{i=0}^\infty\frac{\gamma^i}{i!}\, \mathbf E\abs{X_0}^{2i}\,,
\end{equation*}
so,
\begin{equation*}
  \mathbf Ee^{\gamma\abs{X_t}^2}\le
\frac{1}{1-\gamma K_3/K_1}\,
\mathbf Ee^{\gamma\abs{X_0}^2}\,.
\end{equation*}
\end{proof}

 \bibliographystyle{plain}
\bibliography{large,idemp,puh,stoch,finance,sprav,optim,pde}
\end{document}